\documentclass[a4paper,11pt]{amsart}

\usepackage{amssymb}
\usepackage{amsfonts}
\usepackage{amsmath}

\usepackage{graphicx}
\usepackage[all]{xy}
\usepackage{array}

\usepackage{amsthm}

\newtheorem{theoreme}{Theorem}[section]
\newtheorem{lemme}[theoreme]{Lemma}
\newtheorem{proposition}[theoreme]{Proposition}
\newtheorem{corollary}[theoreme]{Corollary}

\theoremstyle{definition}
\newtheorem{definition}[theoreme]{Definition}
\newtheorem{probleme}{Problem}
\newtheorem{notation}[theoreme]{Notation}
\newtheorem{remark}[theoreme]{Remark}
\newtheorem{example}[theoreme]{Example}

\newcommand{\Si}{\mathfrak{S}}
\newcommand{\Ext}{{\mathrm{Ext}}}
\renewcommand{\hom}{{\mathrm{Hom}}}
\newcommand{\End}{{\mathrm{End}}}
\newcommand{\Id}{{\mathrm{Id}}}
\renewcommand{\k}{\Bbbk}
\newcommand{\op}{\mathrm{op}}
\renewcommand{\P}{\mathcal{P}}
\newcommand{\V}{\mathcal{V}}

\newcommand{\B}{\overline{B}}
\newcommand{\Gr}{\mathrm{Gr}}
\renewcommand{\t}{\mathrm{t}}

\newcommand{\R}{\mathcal{R}}
\newcommand{\Z}{\mathbb{Z}}
\newcommand{\E}{\mathbb{E}}
\newcommand{\N}{\mathbb{N}}
\newcommand{\F}{\mathcal{F}}
\newcommand{\U}{\mathcal{U}}
\newcommand{\Fp}{{\mathbb{F}_p}}
\newcommand{\Fdeux}{{\mathbb{F}_2}}
\newcommand{\KK}{\mathbb{K}}

\newcommand{\W}{{\underline{w}}}
\newcommand{\Pa}{{\underline{\pi}}}
\renewcommand{\H}{\mathbb{H}}
\newcommand{\ev}{\mathrm{ev}}

\title[Bar complexes and Extensions]{Bar complexes and extensions of classical exponential functors}
\author{Antoine Touz\'e }
\thanks{The author was partially supported by the ANR HGRT (Projet BLAN08-2 338236)} 


\begin{document}

\sloppy

\begin{abstract}
We compute $\Ext$-groups between classical exponential functors (i.e. symmetric, exterior or divided powers) and their Frobenius twists. Our method relies on bar constructions, and bridges these $\Ext$-groups with the homology of Eilenberg-Mac Lane spaces.

Together with \cite{TouzeTroesch}, this article provides an alternative approach to classical $\Ext$-computations  \cite{FS,FFSS,Chalupnik1,Chalupnik2} in the category of strict polynomial functors over fields. We also obtain significant $\Ext$-computations for strict polynomial functors over the integers.
\end{abstract}

\maketitle

\section{Introduction}

\subsection{}
Given a commutative ring $\k$, we denote by $\P_\k$ the category of strict polynomial functors introduced by Friedlander and Suslin in \cite{FS}. These strict polynomial functors are functors $F:\V_\k\to \k\text{-Mod}$ from the category $\V_\k$ of finitely generated projective $\k$-modules to the category of $\k$-modules, equipped with an additional scheme-theoretic structure. The computation of $\Ext$ groups in $\P_\k$ has applications to the cohomology of algebraic groups and the Steenrod algebra, see \cite{Pan}.

The symmetric powers $S^d:V\mapsto S^d(V)$, the exterior powers $\Lambda^d:V\mapsto \Lambda^d(V)$ or the divided powers $\Gamma^d:V\mapsto \Gamma^d(V)=(V^{\otimes d})^{\Si_d}$ yield fundamental examples of strict polynomial functors. We refer to them as the `classical exponential functors'. Indeed they are the most common graded functors $X^*$ satisfying an exponential formula, i.e. a graded isomorphism (see section \ref{sec-exp} for more on exponential functors):
$$X^*(V\oplus W)\simeq X^*(V)\otimes X^*(W)\;.$$
If $\k$ is a field of positive characteristic $p$, we denote by $I^{(r)}$ the $r$-th Frobenius twist functor, for $r\in\N$ , and by $F^{(r)}$ the precomposition of a strict polynomial functor $F$ by the $r$-th Frobenius twist functor (with the convention that $I^{(0)}$ is the identity functor, hence $F^{(0)}=F$). Frobenius twists are additive functors, so the twisted variant $X^{*\,(r)}$ of a classical exponential functor $X^*$ is also an exponential functor.

The classical exponential functors, and their twisted variants over a field $\k$
 of positive characteristic, are among the most elementary examples of strict
polynomial functors. However, the computation of extension groups in $\P_\k$
between these functors is a challenging problem, and the main topic of this article.

Let us first consider the
untwisted case, over an arbitrary commutative ring $\k$. We order the classical
exponential functors as follows: $\Gamma^*<\Lambda^* < S^*$. If $X^*$ and $Y^*$
are classical exponential functors with $X^*\le Y^*$ then it is
well-known\footnote{This follows from the
projectivity of divided powers, and the injectivity of symmetric powers,
except for the case of the extensions $\Ext^i_{\P_\k}(\Lambda^*,\Lambda^*)$.
In the latter case, there are many proofs of this vanishing, and we will
provide one in remark \ref{rk-Ext-Lambda-Lambda}.} that
$\Ext^i_{\P_\k}(X^*,Y^*)$ is zero if $i>0$.
The degree zero part is an elementary computation, so the only
nontrivial computations to achieve are those of
$$(i)\;\Ext^*_{\P_\k}(S^*,\Lambda^*)\;,\quad
(ii)\;\Ext^*_{\P_\k}(\Lambda^*,\Gamma^*)\;,\quad
(iii)\;\Ext^*_{\P_\k}(S^*,\Gamma^*)\;.$$
Moreover, the first two extensions $(i)$ and $(ii)$ are isomorphic by a duality
argument. So we only consider the extension groups
$(i)$ and $(iii)$. Akin studied \cite{A} the extension groups $(i)$.
He actually performed computations in the equivalent category of
modules over the Schur algebra. His results give a formula for the dimension of
these $\Ext$-groups when $\k$ is a field of positive characteristic, as well as
the computation of $\Ext^i_{\P_\k}(S^d,\Lambda^d)$ for $i=1$ and $d-1$ when $\k$
is the ring of integers. The method of Akin uses the classical resolution of
exterior powers by symmetric powers, given by the bar construction of the
symmetric algebra. 

Over a field $\k$ of positive characteristic $p$, we also consider the
extensions between the classical exponential functors precomposed by
Frobenius twists. Precomposition by Frobenius twists actually creates a lot
of extensions, and the extension groups
$$\Ext^*_{\P_\k}(X^{*\,(r)},Y^{*\,(r)})$$
are non-trivial, even when the classical exponential functors $X^*$ is smaller
than $Y^*$ (for the order defined before).
The $\Ext$-groups with $X^*\le Y^*$ were computed in \cite{FFSS}. In this
article, the authors prove that these extension groups are equipped with a
convolution product, as well as a coproduct of the same type and they compute
these extension groups as trigraded bialgebras.
Their method is different from the method of Akin, namely it relies on the
analysis of some hypercohomology spectral sequences associated to Koszul and De
Rham complexes.
The remaining extension groups, i.e. 
those with $X^*> Y^*$, were computed in \cite{Chalupnik2} with the same
techniques, together with a Koszul duality phenomenon. However, \cite[Thm
3.2]{Chalupnik2}, a key result in \cite{Chalupnik2}, is false in characteristic
$2$ (as explained in remark \ref{rk-Chal}).

\subsection{}

The main goal of this article is to provide new methods to compute the extension
 groups in $\P_\k$ between classical exponential functors, and also between
their twisted variants. Our
methods differ from the methods used in \cite{FFSS,Chalupnik2}. They rather
take the same starting point as in Akin's computations \cite{A}, namely
(iterated) bar complexes of the symmetric algebra. We obtain the following
results.
\begin{enumerate}
\item[(1)] Over a field $\k$, our approach provides new independent
computations of all the $\Ext$-groups computed in \cite{FFSS,Chalupnik2}. The
statements of our results are given in section \ref{sec-final-state}. We
carefully compare our results with the results given in \cite{FFSS,Chalupnik2}
in section \ref{subsec-compare}. In particular our computations agree with all
the computations of \cite{FFSS}.
\item[(2)] Our methods also allow the computation of $\Ext$-groups between
classical exponential functors over the ring $\k=\Z$. In particular, we obtain
in theorem \ref{thm-PSZ} an explicit simple formula describing the abelian
groups $\Ext^i_{\P_\Z}(S^d,\Lambda^d)$, for all $i$, thus extending the results
of \cite{A}. Other (more complicated) descriptions  are obtained in theorems
\ref{thm-calculsanstwistIII} and \ref{thm-calculsanstwistIV}).
\end{enumerate}

To obtain our computations of extension groups in $\P_\k$ between (untwisted)
classical exponential functors, we interpret them
 as the singular homology with $\k$ coefficients of some
Eilenberg-Mac Lane spaces. 
The homology of these spaces was computed by Serre with
$\mathbb{F}_2$ coefficients \cite{Serre}, and by Cartan in general
\cite{Cartan}. Then we elaborate on Cartan's results to obtain our
$\Ext$-computations. As an intermediate result, we obtain the following
computation, of independent and general interest.
\begin{itemize}
\item[(3)] We compute in theorem \ref{thm-BnGamma-F} the homology of iterated
bar constructions of the divided power algebra $\Gamma^*(V)$, over any field
$\k$. (See corollary \ref{cor-fonc-alg} for the relation to Cartan's
computations and remark \ref{rk-comput-Car} for an explanation why Cartan's
computations only give the result over \emph{prime fields}).
\end{itemize}

Apart from the results listed above, we feel that many technical results proved
in the course of the article might be useful tools for further
$\Ext$-computations. For example, the key proposition
\ref{prop-key} has a wider range of application than classical exponential
functors.

\subsection{}
We now briefly sketch the structure of our proof.

\subsubsection{Structures on extension groups}\label{subsec-introexp} Let $X^*$
and $Y^*$ be classical exponential functors, (or their Frobenius twisted variant
if $\k$ is a field of positive characteristic). Before
 undertaking computations, we review in part
\ref{part-1} of the paper various structures which equip the extension groups
$\Ext^*_{\P_\k}(X^*,Y^*)$, and determine which ones it is important to compute.

{\bf Functoriality.} Parameterized extensions were introduced (for
different reasons) in \cite{Chalupnik2} and \cite{TouzeTroesch}, and they are
a key tool for our computations. If $V\in\V_\k$, we let $X^{*\, V}$ be the
functor $U\mapsto X^*(\hom_\k(V,U))$. The parameterized extensions 
$\E(X,Y)$ are the (tri)graded strict polynomial functor defined by (the
parameter is the variable `$V$')
$$\E(X,Y)(V)= \bigoplus_{i,d,e\ge
0}\E^i(X^{d}, Y^e)=\bigoplus_{i,d,e\ge
0}\Ext^i_{\P_\k}(X^{d\,V}, Y^e)\;.$$
We denote by $\H(X,Y)$ the `$\hom$-part', i.e. the summand with $i=0$, of
$\E(X,Y)$. We will study the functors $\E(X,Y)$ rather
than the mere extension groups $\Ext^*_{\P_\k}(X^*, Y^*)$ because functoriality
reveals useful hidden structure (cf. the gradings and the coproducts below), and
parameterized extension groups are not much more difficult to compute.

{\bf Gradings.} The trigrading of $\E(X,Y)$ is artificial because the second
and the third partial degrees are linearly dependent. For example,
$\E^*(S^{d},\Lambda^e)$ is zero if $d\ne e$, so we actually have to
compute the bigraded extensions
$$\E(S,\Lambda)=\bigoplus_{i,d\ge 0}\E^i(S^d,\Lambda^d)\;.$$
Moreover, since we study parameterized extension groups, even the bigrading
is artificial. Indeed, the second grading (i.e. the exponent `$d$') is
encoded by the functoriality\footnote{We call `weight' the
grading implicitly encoded by functoriality. See section \ref{sec-hierarchy} and
definition \ref{def-conv}.}. Thus we only have to determine $\E(X,Y)$ as a
strict polynomial functor, with a single $\Ext$-grading.

{\bf Products.}
The parameterized extension groups $\E(X,Y)$ are equipped with a convolution
product defined in \cite{FFSS}. We say for short a `$\P_\k$-graded algebra' for
a graded strict polynomial functor endowed with an algebra structure. We will
actually compute the $\P_\k$-graded algebra $\E(X,Y)$.

{\bf Coproducts.}
When the ground ring $\k$ is a field, $\E(X,Y)$ is also equipped with a
coproduct defined in \cite{FFSS}. The coproducts are computed in
\cite{FFSS,Chalupnik2}. 
However, we prove in section \ref{sec-exp} that these parameterized
extensions actually satisfy an
exponential formula:
$$\E(X,Y)(V\oplus W)\simeq \E(X,Y)(V)\otimes \E(X,Y)(W)\;.$$
The coproduct is determined by the exponential formula, and the exponential 
formula is in turn determined by the product on $\E(X,Y)$. Thus, if we compute
the $\P_\k$-graded algebra $\E(X,Y)$, we can easily deduce the coproduct from
the product (see remark \ref{rk-intro}).
That's why we limit ourselves to computing $\E(X,Y)$ as $\P_\k$-graded algebras
in this article. 

\subsubsection{Extensions between classical exponential functors}
As in the case of extension groups without parameters, the only non-trivial
computations to achieve are the computations of $\E(S,\Lambda)\simeq
\E(\Lambda,\Gamma)$ and $\E(S,\Gamma)$.

In part \ref{part-2} of the paper, show that these extension groups can be
computed (up to regrading) as the homology of the bar construction, resp. twice
iterated bar construction, of the divided power algebra. 

To prove this, we start in the same way as \cite{A}, that is, we consider
the resolutions of exterior and divided powers yield by iterated bar
constructions of the symmetric algebra. Thus our extension groups may be
computed (up to regrading) as the  homology of the differential graded
$\P_\k$-algebras
$\H(S,\B S)$ and $\H(S,\B^2 S)$. The regrading involves sign issues which are
taken care of by the `regrading functors' described in section
\ref{sec-regrading}. Then the key result is proposition \ref{prop-key}, which
says that in our situation (and many others), the symbols `$\H$' and `$\B$' can
be interchanged. For example
$$\H(S,\B S)\simeq \B(\H(S,S))\simeq\B\Gamma\;.$$

\subsubsection{The homology of Eilenberg-Mac Lane spaces} So we are left with
the problem of computing the $\P_\k$-graded algebra $H(\B^n\Gamma)$. 
Let $\Gamma(\k^n[2])$ be the free divided power algebra over $\k$,
generated by the $\k$-module $\k^n$ placed in degree $2$. We prove in corollary \ref{cor-fonc-alg}
an isomorphism of graded $\k$-algebras, natural with respect to $\Z^m$: 
$$H_*(\B^n(\Gamma(\Z^m\otimes_\Z\k[2])))\simeq H_*^{\mathrm{sing}}(K(\Z^m,n+2),\k)$$
between the homology of iterated bar constructions of divided powers, and the
singular homology of Eilenberg-Mac Lane spaces. The later was computed by Cartan
\cite{Cartan}. However, Cartan's result is not enough for our purposes: we really need the
left hand side as a $\P_\k$-graded algebra (see more details on $\P_\k$-graded algebras in section \ref{sec-hierarchy}).
So we have to elaborate
on Cartan's computations to get our results.

\subsubsection{Frobenius twists}
In part \ref{part-4} of the paper, we explain how to retrieve the
$\P_\k$-graded algebras  $\E(X^{(r)}, Y^{(s)})$ from
$\P_\k$-graded algebra $\E(X,Y)$ in a simple way ($\k$ is now a field of
positive characteristic $p$). 
This part is a sequel to
\cite{TouzeTroesch} and some theorems of \cite{TouzeTroesch} are actually used
in section \ref{sec-tw-ss}, while the other sections of part  \ref{part-4} are
dedicated to complementary techniques.

\subsection{Some conventions used in the article}

In the article, we work over a commutative base ring $\k$, we denote by $\k-\mathrm{mod}$ the category of $\k$-modules and by $\V_\k$ the full subcategory of 
finitely generated projective $\k$-modules. The
$\k$-linear dual of a $\k$-module $V$ is $V^\vee=\hom_\k(V,\k)$. Otherwise
explicitly stated, tensor products are taken over $\k$. In particular, if $V$ is
a graded $\k$-module, the symbols $S(V)$, $\Lambda(V)$ and $\Gamma(V)$ refer to
the symmetric, the  exterior and the divided power algebras over $\k$ (i.e.
$S(V)=S_\k(V)$). 

Otherwise explicitly stated, the word `degree' means `homological degree'. Homological degrees are denoted by indices, i.e. if $M$ is a complex, $M_i$ denotes the direct summand of degree $i$, and differentials lower the degree by one. Sometimes, cohomological degrees appear: they are denoted by exponents as in $\Ext^i_{\P_\k}(F,G)$. Such cohomological degrees can be converted into homological degrees by the usual formula: $M^i=M_{-i}$. If $M$ is an ungraded object, we denote by $M[i]$ the graded object with $M[i]_j=M$ if $i=j$ and zero if $i\ne j$. When cohomological degrees naturally appear in a formula, we use the notation $M\langle i\rangle$ to denote a copy of $M$ placed in cohomological degree $i$ (thus $M\langle i\rangle= M[-i]$).

It is well known that the abelian category $\P_\k$ of strict polynomial functors decomposes as a direct sum $\P_\k=\bigoplus_{d\ge 0}\P_{d,\k}$. Elements of $\P_{d,\k}$ are called homogeneous functors, and if $F\in\P_{d,\k}$, the integer $d$ is often called the `degree' of the functor \cite{FS}. To avoid confusion with homological degrees, we do not use this terminology, we rather call $d$ the \emph{weight} of the functor (cf section \ref{sec-recollections} below). 

\setcounter{tocdepth}{1} 
\tableofcontents

\part{General structure results}\label{part-1}

In this part, we describe the framework and the notations for our computations.

\section{Recollections of strict polynomial functors}\label{sec-recollections}
In this section, we briefly review basic facts of the theory of strict
polynomial functors needed in the article. We refer the reader to
\cite{FS,SFB,Pan,TouzeClassical,Krause} for more details.

\subsection{Strict polynomial functors}
For $d\ge 0$, let $\Gamma^d\V_\k$ be the category with the same
objects as $\V_\k$ (i.e. finitely generated projective $\k$-modules) but whose morphisms are $\Si_d$-equivariant maps:
$$\hom_{\Gamma^d\V_\k}(V,W)=\hom_{\Si_d}(V^{\otimes d},W^{\otimes d})\;,$$
where $\Si_d$ acts on $V^{\otimes d}$ by permuting the factors of the tensor product, and composition in $\Gamma^d\V_\k$ is composition of $\Si_d$-equivariant maps (by convention,
$V^{\otimes 0}=\k$, and $\Si_0=\{1\}$).
 The notation `$\Gamma^d\V_\k$' is justified by the
isomorphism $\Gamma^{d}(\hom_\k(V,W))\simeq\hom_{\Gamma^d\V_\k}(V,W)$, induced
by the isomorphism
$$\hom_{\Si_d}(V^{\otimes d},W^{\otimes d})=\hom_\k(V^{\otimes d},W^{\otimes
d})^{\Si_d}\simeq (\hom_\k(V,W)^{\otimes d})^{\Si_d}\;. $$

The category $\P_{d,\k}$ of homogeneous strict polynomial functors of weight
$d$ is the abelian category of $\k$-linear
functors from $\Gamma^d\V_\k$ to $\k$-mod. 

The abelian category $\P_{\k}$ of
strict polynomial functors is the direct sum: 
$$\P_{\k}=\bigoplus_{d\ge 0}\P_{d,\k}\;.$$ 
So the weight is an implicit grading on strict
polynomial functors.
Each strict polynomial functor $F$ splits as a direct sum  $F=\oplus F_d$ of
homogeneous strict polynomial functors $F_d$ of weight $d$, and
$\hom_{\P_\k}(F_d,G_e)$ equals $\hom_{\P_{d,\k}}(F_d,G_e)$ if $e=d$ and zero
otherwise.

\subsection{Strict polynomial vs ordinary functors}

Let $\F_\k$ denote the category of `ordinary functors', i.e. functors from
$\V_\k$ to $\k$-mod. There is a exact and faithful forgetful functor 
$$\U:\P_\k\to \F_\k\;.$$
(The restriction $\U:\P_{d,\k}\to \F_\k$ is defined as
the precomposition by the functor $\gamma_d:\V_\k\to \Gamma^d\V_\k$, where
$\gamma_d$ is the identity on objects and $\gamma^d(f):=f^{\otimes d}$.)
Thus, we can think of strict polynomial functors as ordinary functors, equipped
with a `strict polynomial structure'.
For example, the following functors
\begin{enumerate}
\item[(i)] the $d$-th tensor product $\otimes^d:V\mapsto V^{\otimes d}$,
\item[(ii)] the $d$-th exterior power $\Lambda^d:V\mapsto V^{\otimes d}$,
\item[(iii)] the $d$-th symmetric power $S^d:V\mapsto S^d(V)=(V^{\otimes
d})_{\Si_d}$ (coinvariants under the action of $\Si_d$),
\item[(iv)] the $d$-th divided power $\Gamma^d:V\mapsto \Gamma^d(V)=(V^{\otimes
d})^{\Si_d}$ (invariants under the action of $\Si_d$, the reason for the name
`divided power' is given in section \ref{sec-dpa}).
\end{enumerate}
have a canonical structure of homogeneous strict polynomial functors of weight
$d$. 

\begin{remark}\label{rk-attention}
If $\k$ is an infinite field, the forgetful functor 
$\U:\P_\k\to \F_\k$
is full and faithful, so $\P_\k$ can be thought of as a full subcategory of
$\F_\k$. In general, the behavior of the forgetful functor is more subtle.
For example, if $\k=\Fp$, the $r$-th Frobenius twist functor $I^{(r)}$
\cite[(v) p.224]{FS} and the
identity functor $I$ are nonisomorphic strict polynomial functors (they are
homogeneous of different weights), but $\U I^{(r)}=\U I$. Even if we only consider homogeneous strict polynomial functors of a given weight, the restriction $\U:\P_{d,\k}\to \F_\k$ 
behaves badly. For example if $\k=\Fdeux$, one can show that
$F=S^{2(1)}\otimes I^{(1)}$ and $G=S^2\otimes I^{(2)}$ are nonisomorphic
objects of $\P_{6,\Fdeux}$, but $\U(F)\simeq\U(G)$.
\end{remark}

\subsection{Operations on strict polynomial functors}
Usual operations on ordinary functors (composition, tensor products and duality)
lift at the level of strict polynomial functors. If $F$ and $G$ are strict
polynomial functors, their
tensor product $(F\otimes G)(V)=F(V)\otimes G(V)$ is canonically equipped with
the structure of a strict polynomial functor. If we restrict our attention to
homogeneous functors, tensor product induces a functor: 
$$\otimes :\P_{d,\k}\times \P_{e,\k}\to \P_{d+e,\k}.$$

Given a strict polynomial functor $F$, the dual functor $F^\sharp:V\mapsto
F(V^\vee)^\vee$ (where $^\vee$ refers to the $\k$-linear dual) is canonically
a strict polynomial functor. For example $S^{d\,\sharp}\simeq\Gamma^d$,
$\Lambda^{d\,\sharp}\simeq\Lambda^d$ and $\Gamma^{d\,\sharp}\simeq S^d$. The
functor 
$$^\sharp:\P_{d,\k}^\op\to \P_{d,\k}$$
is called `Kuhn duality' in \cite{FS}. 
If $\P_{d,\k}'$ denotes the full subcategory of
$\P_{d,\k}$ of functors with values in $\V_\k$, duality induces an
equivalence of categories ${\P_{d,\k}'}^\op\simeq \P_{d,\k}'$.

Finally, the composition $G\circ F$ of a strict polynomial functor $G$
by a strict polynomial functor $F$ makes sense, provided $F$ takes finitely
generated projective values. So composition
yields a functor: 
$$\circ:\P_{d,\k}\times \P_{e,\k}'\to \P_{de,\k}.$$
In particular, we can consider the precomposition by the following
homogeneous functors. The functors $V\otimes :W\mapsto V\otimes W$, and
$\hom_\k(V,-):W\mapsto \hom_\k(V,W)$ (these functors have weight $1$), and if
$\k$ is a field of characteristic $p>0$, the $r$-th Frobenius twist $I^{(r)}$
(of weight $p^r$) \cite[(v) p.224]{FS}. Such compositions will occur frequently
in the article.
\begin{notation}\label{nota-1}
For all $F\in \P_{\k}$, we use the following notations:
$$F^{(r)}= F\circ I^{(r)}\;,\quad F_V:= F\circ (V\otimes)\;,\quad F^V:= F\circ
\hom_\k(V,-)\;. $$
The functors $F_V$ and $F^V$ are called `parameterized functors' (the parameter
is the finitely generated projective $\k$-module $V$). 
\end{notation}

\subsection{Homological algebra in $\P_\k$}
The abelian categories $\P_{d,\k}$, have
enough projectives. To be more specific, a projective generator is given by
the parameterized divided power
$$\Gamma^{d,V}:=\hom_{\Gamma^d\V_\k}(V,-)\simeq(\Gamma^d)^V\;,$$
for $V$ a free $\k$-module with rank greater or equal to $d$. The projectivity
of $\Gamma^{d,V}$ (for all $V\in\V_\k$) is a consequence of the Yoneda
isomorphism
(frequently used in the sequel of the article):
$$\hom_{\P_{d,\k}}(\Gamma^{d,V},F)\simeq F(V)\;.$$

Since $\P_{\k}=\bigoplus_{d\ge 0}\P_{d,\k}$ the category $\P_\k$ also has
enough projectives. If $F_d$ and $G_e$ are homogeneous strict polynomial
functors of weight $d$, resp. $e$, then $\Ext^*_{\P_\k}(F_d,G_e)$ is equal to
$\Ext^*_{\P_{d,\k}}(F_d,G_e)$ if $d=e$, and zero otherwise.

The following lemma explains that the parameterized symmetric powers $S^d_V$
play the role of injective functors, provided we restrict to functors with
values in $\V_\k$ (which is the case for all the $\Ext$-computations of the
paper).

\begin{lemme}\label{lm-sym-res}
Let $F,G \in\P_{d,\k}$. Assume that $F$ and $G$ take finitely generated
projective values. Duality induces an isomorphism (natural in $F$,$G$)
$$\Ext^*_{\P_\k}(F,G)\simeq \Ext^*_{\P_\k}(G^\sharp,F^\sharp)\;.$$
Moreover $G$ admits a coresolution by finite direct sums of copies of functors
of the form $S^d_V$, and if $J_G$ is such a resolution, $\Ext^*_{\P_\k}(F,G)$
equals the homology of the complex $\hom_{\P_\k}(F,J_G)$.
\end{lemme}
\begin{proof}
Let us denote by $\P_\k'$ the full subcategory of $\P_\k$ whose objects are the
strict polynomial functors with values in $\V_\k$.  This is an exact
subcategory of $\P_{\k}$, with projective objects $\Gamma^{d,V}$, cf.
\cite[Section 2]{SFB}. A projective resolution in $\P_\k'$ remains a
projective resolution if we view it in $\P_\k$, so inclusion
$\P'_\k\hookrightarrow \P_\k$ induces an isomorphism at the level of $\Ext$s.

The duality functor is an equivalence of categories when
restricted to $\P'_\k$. So the first assertion of lemma
\ref{lm-sym-res} follows from the commutative diagram:
$$\xymatrix{
\Ext^*_{\P_\k}(F,G)\ar[r]^-{^\sharp} & \Ext^*_{\P_\k}(G^\sharp,F^\sharp)\\
\Ext^*_{\P'_\k}(F,G)\ar[r]^-{^\sharp}_-{\simeq}\ar[u]^-{\simeq} &
\Ext^*_{\P'_\k}(G^\sharp,F^\sharp)\ar[u]^-{\simeq}}
.$$
The second assertion of lemma \ref{lm-sym-res} follows from the fact that
category $\P_\k'$ has enough injectives, and an injective cogenerator is
provided by the functors $S^d_V$, cf. \cite[Section 2]{SFB}. 
\end{proof}

\section{A hierarchy of algebras}\label{sec-hierarchy}

In this article, the algebras considered are often endowed with an additional structure such as functoriality or an extra grading (called `weight'). The categories of $\k$-algebras we will use are organized according to the following diagram.
$$
\xymatrix{
{\text{\{$\mathcal{P}_\k$-dg-alg\}}}\ar[rr]^-{\ev_V}\ar[d]^-{\U}&& {\text{\{$\k$-wdg-alg\}}}\ar[d]&\ar@{}[d]^-{(\mathcal{H})} \\
{\text{\{$\mathcal{F}_\k$-dg-alg\}}}\ar[rr]^-{\ev_V}&& {\text{\{$\k$-dg-alg\}}}& \\
}
$$
In this diagram, the arrows are functors induced by forgetting part of the structure. Thus, the category \{$\k$-dg-alg\} on the bottom right corner of the diagram is the one having the less structured objects. In this section, we review the definitions of the various categories  appearing in the hierarchy diagram $(\mathcal{H})$. 

\subsection{Differential graded algebras}

By \{$\k$-dg-alg\}, we denote the usual category of differential graded algebras over the commutative base ring $\k$.

If $A$ is a differential graded algebra, the degree of a homogeneous element $x\in A$ is denoted by $|x|$. Recall that the tensor product $A\otimes B$ of two dg-algebras has product $(a\otimes b)\cdot (a'\otimes b')=(-1)^{|a'||b|} aa'\otimes bb'$ and differential $\partial(a\otimes b)=(\partial a)\otimes b+(-1)^{|a|}a\otimes (\partial b)$.

Graded $\k$-algebras are viewed as dg-algebras with trivial differential, and ungraded $\k$-algebras are viewed as dg-algebras concentrated in degree zero. 

\subsection{Weighted differential graded algebras}

By \{$\k$-wdg-alg\}, we denote the category of weighted differential graded algebras over $\k$. The objects of this category are the weighted graded differential algebras over $\k$, that is, the dg-algebras $A$ equipped with an extra nonnegative grading, called `weight'. The weight of a homogeneous element $x \in A$ is denoted by $w(x)$ and the weights are required to satisfy:
$$ w(xy)=w(x)+w(y)\,,\qquad w(\partial x)= w(x)\,. $$
Morphisms of wdg-algebras are $\k$-linear maps which preserve the gradings and the weights, which are multiplicative and which commute with the differentials.

The tensor product $A\otimes B$ of two wdg-algebras is defined by letting for all homogeneous elements $a$, $b$ :
\begin{itemize}
\item[(i)] $ |a\otimes b|=|a|+|b|$, and $w(a\otimes b)=w(a)+w(b)$,
\item[(ii)] $(a\otimes b)\cdot (a'\otimes b')= (-1)^{|a'||b|}aa'\otimes bb'$,
\item[(iii)] $ \partial (a\otimes b)=(\partial a)\otimes b+(-1)^{|a|}a\otimes (\partial b)$.
\end{itemize}
Observe that the weights do not contribute to the signs in the definition of tensor products. Thus, forgetting the weights yields a functor, which preserves the tensor products:
$$\{\text{$\k$-wdg-alg}\}\to \{\text{$\k$-dg-alg}\} \;.$$

\subsection{Functorial differential graded algebras}

We denote by \{$\mathcal{F}_\k$-dg-alg\} the category of functorial differential graded algebras. The objects of this category are the functors $A:\V_\k\to \{\text{$\k$-dg-alg}\}$, and the morphisms are the natural transformations between such functors.

Let $V\in\V_\k$ be a finitely generated projective $\k$-module. Evaluation on $V$ yields a forgetful functor (we forget the naturality with respect to $V$):
$$\{\text{$\mathcal{F}_\k$-dg-alg}\}\xrightarrow[]{\ev_V} \{\text{$\k$-dg-alg}\}.$$
The tensor product of two $\mathcal{F}_\k$-dg-algebras is defined in the target category, so that the forgetful functor commutes with tensor products. 

\subsection{Strict polynomial differential graded algebras}\label{subsec-strpoldgalg}
We denote by \{$\mathcal{P}_\k$-dg-alg\} the category of strict polynomial differential graded algebras. This is a category of highly structured algebras. To be more specific, we make the following definitions.
\begin{definition}\label{def-str-pol-alg}
Let $\k$ be a commutative ring. A strict polynomial differential graded algebra ($\P_\k$-dg-algebra) $A$ is a collection of strict polynomial functors $\{A_{i,d}\}$ for $i\in\Z$ and $d\ge 0$, equipped with morphisms of strict polynomial functors $\eta:\k\to A_{0,0}$, $m :A_{i,d}\otimes A_{j,e}\to A_{i+j,d+e}$, and $\partial :A_{i,d}\to A_{i-1,d}$ satisfying the following properties.
\begin{itemize}
\item[(i)] For all $d\ge 0$, the strict polynomial functor $A_{i,d}$ is homogeneous of weight $d$.
\item[(ii)] For all $V\in\V_\k$, the unit $\eta$, the multiplication $m$ and the differential $\partial $ make the $\k$-module $A(V)=\bigoplus_{i,d} A_{i,d}(V)$ into a weighted differential graded algebra. 
\end{itemize}
The functor $A_{i,d}$ will be referred to as the homogeneous part of degree $i$ and weight $d$ of $A$.
\end{definition}

The $\P_\k$-graded algebras are the $\P_\k$-dg-algebras with trivial
differential, 
and the $\P_\k$-algebras are the $\P_\k$-dg-algebras concentrated in degree
zero. 

Symmetric, exterior and divided powers
yield fundamental examples of $\P_\k$-graded algebras. To be more specific, let
$X=S,\Lambda$ or $\Gamma$. Then $\hom_{\P_\k}(X^d\otimes
X^e,X^{d+e})$ is a rank one free $\k$-module. If $X=S$ or $\Lambda$, it is
generated by the unique morphism fitting into diagram $(D1)$ below. For
$X=\Gamma$, it is generated by the unique morphism fitting in diagram $(D2)$,
where $\mathrm{S}(d,e)\subset \Si_{d+e}$ is the set of all $(d,e)$-shuffles. We
call this generator the canonical map between $X^d\otimes X^e$ and $X^{d+e}$.
$$\xymatrix{
(\otimes^d)\otimes
(\otimes^e)\ar@{->>}[d]\ar[r]^-{=}&\otimes^{d+e}\ar@{->>}[d]\ar @{}
[dl]|{(D1)}\\
X^d\otimes X^e\ar@{-->}[r]_-{\exists !}&X^{d+e}
}\;,\qquad 
\xymatrix{
(\otimes^d)\otimes (\otimes^e)\ar[rr]^-{\sum_{\sigma\in \mathrm{S}(d,e)}
\sigma}&&\otimes^{d+e}\ar @{} [dll]|{(D2)}\\
\Gamma^d\otimes \Gamma^e\ar@{^{(}->}[u]\ar@{-->}[rr]_-{\exists
!}&&\Gamma^{d+e}\ar@{^{(}->}[u]
}.$$ 
In the remainder of the article, we will mostly use the $\P_\k$-graded algebras
of the following type.
\begin{itemize}
\item We denote by $S[i]$ the symmetric algebra over generators of homological
degree $i$. That is, $S[i]$ is the strict polynomial graded algebra whose
homogeneous part of degree $di$ and weight $d$ equals $S^d$, with multiplication
defined by the canonical maps $S^d\otimes S^e\to S^{d+e}$. The algebra $S[0]$
will simply be denoted by $S$. We define $\Lambda[i]$ and $\Gamma[i]$ (and
$\Lambda$, $\Gamma$) similarly.
\item More generally, 
if $F$ is a graded strict polynomial functor with values in $\V_\k$, we denote
by $S(F)$ the symmetric algebra on $F$. It is formed by the graded strict
polynomial functors $S^d\circ F$ (with degrees defined as usual), and the
multiplication is given by evaluating the canonical map $S^d\otimes S^e\to
S^{d+e}$ on $F$. We define the exterior algebra $\Lambda(F)$ over $F$, and the
divided power algebra $\Gamma(F)$ over $F$ similarly.
\end{itemize} 

\begin{definition}
A morphism of $\mathcal{P}_\k$-dg-algebras $f:A\to B$ is a collection of morphisms of strict polynomial functors $f_{i,d}:A_{i,d}\to B_{i,d}$ which commute with the multiplications and the differentials of $A$ and $B$.
\end{definition}

We define the tensor product $A\otimes B$ of two $\P_\k$-dg-algebras by formulas similar to the ones for the wdg-algebras (i.e. the degrees introduce Koszul signs in the definition of the product and the differentials, but the weights do not introduce such  Koszul signs). The forgetful functor  $\mathcal{U}:\P_\k\to \F_\k$ induces a functor:
$$\{\text{$\mathcal{P}_\k$-dg-alg}\}\xrightarrow[]{\mathcal{U}} \{\text{$\mathcal{F}_\k$-dg-alg}\}$$
which commutes with tensor products. If $V\in\V_\k$ is a finitely generated projective $\k$-module, evaluation on $V$ induces a forgetful functor:
$$\{\text{$\mathcal{P}_\k$-dg-alg}\}\xrightarrow[]{\ev_V} \{\text{$\k$-wdg-alg}\},$$
which also commutes with tensor products.

\begin{remark}\label{rk-attentionbis}
Let $A$ be a $\P_\k$-dg-algebra. Then knowing the collection of wdg-algebras
$A(V)$, $V\in\V_\k$, and the $\F_\k$-dg-algebra
$\mathcal{U} A$ is usually not sufficient
to determine $A$ as a $\P_\k$-dg-algebra.
Indeed, one can find (see remark \ref{rk-attention}) nonisomorphic functors
$F,G\in\P_{d,\k}$, such that $\U(F)\simeq \U(G)$. Then the symmetric algebras
$A=S(F)$ and $B=S(G)$ are not isomorphic as $\P_\k$-algebras, although we have
isomorphisms of $\F_\k$-algebras $\mathcal{U}A = \mathcal{U}B$, and isomorphisms
of wg-algebras $A(V)\simeq B(V)$ for all $V$.
\end{remark}

\subsection{More strict polynomial objects}
In the article, we also need other strict polynomial objects. 
For instance, the \emph{strict polynomial commutative differential graded
algebras} ($\P_\k$-cdg-algebras) are the  cdg-algebras which are commutative as
differential graded algebras, i.e. $xy=(-1)^{|x||y|}yx$ for all homogeneous
elements $x,y\in A(V)$, for all $V\in\V_\k$. The algebras can also be augmented
($\P_\k$-dga-algebras), or commutative and augmented ($\P_\k$-cdga-algebras). We
also use strict polynomial differential graded coalgebras
($\P_\k$-dg-coalgebras),  etc. It is quite obvious how to adapt definition
\ref{def-str-pol-alg} to these cases, and we leave this to the reader.

\section{Structure of extension groups}

\subsection{Parameterized extension groups}\label{subsec-param}
In this section we introduce the parameterized extension groups $\E(F,G)$
alluded to in the introduction of the paper.
 We begin by parameterized $\hom$-groups $\H(F,G)=\E^0(F,G)$.
Let $F\in\P_{d,\k}$, and recall from notation \ref{nota-1} the parameterized
functors:
$$F_V:W\mapsto F(V\otimes W)\quad\text{ and }\quad F^V:W\mapsto F(\hom_\k(V,W))\;.$$
Morphisms $f\in\hom_{\Gamma^d\V_\k}(V,V')$ induce morphisms of strict polynomial
functors
$F_V\to F_{V'}$ and $F^{V'}\to F^V$.
Hence, for all $F,G\in\P_{d,\k}$ we have functors:
$$\begin{array}{ccc}
\Gamma^d\V_\k&\to & \k\mathrm{-mod}\\
V&\mapsto &\hom_{\P_{d,\k}}(F^V,G)
\end{array}\;,
\qquad
\begin{array}{ccc}
\Gamma^d\V_\k&\to & \k\mathrm{-mod}\\
V&\mapsto &\hom_{\P_{d,\k}}(F,G_V)
\end{array}
\;.
$$
\begin{lemme}\label{lm-def-H}
The two strict polynomial functors $$V\mapsto \hom_{\P_{d,\k}}(F^V,G)\quad\text{ and }\quad V\mapsto \hom_{\P_{d,\k}}(F,G_V) $$
are canonically isomorphic. We denote them by $\H(F,G)$.
\end{lemme}
\begin{proof}
Let us first take $F=\Gamma^{d,U}$. Then $(\Gamma^{d,U})^V=\Gamma^{d,U\otimes
V}$ and the Yoneda lemma yields an isomorphism:
$$\hom_{\P_{d,\k}}((\Gamma^{d,U})^V,G)\simeq G(U\otimes V) \simeq
\hom_{\P_{d,\k}}(\Gamma^{d,U},G_V) \,,$$
natural with respect to $G$, to $f\in\hom_{\Gamma^d\V_\k}(V,V')$ and $g\in\hom_{\Gamma^d\V_\k}(U,U')$ (or equivalently to 
$g\in\hom_{\P_{d,\k}}(\Gamma^{d,U'},\Gamma^{d,U})$). 
Since the $\Gamma^{d,U}$, $U\in\V_\k$, form a projective generator of
$\P_{d,\k}$, we can take presentations of $F$ to extend this isomorphism to all
$F\in\P_{d,\k}$.
\end{proof}

Let us turn to the category $\P_{\k}$ of strict polynomial functors. If $F,G$ are strict polynomial functors, they split as finite direct sums $F=\bigoplus F_d$ and $G=\bigoplus G_d$, where $F_d$ and $G_d$ are homogeneous of weight $d$, and by lemma \ref{lm-def-H}, the functors
\begin{align*}&V\mapsto \hom_{\P_{\k}}(F^V,G)= \textstyle\bigoplus \hom_{\P_{d,\k}}(F_d^V,G_d)\\
&V\mapsto \hom_{\P_{\k}}(F,G_V)= \textstyle\bigoplus \hom_{\P_{d,\k}}(F_d,(G_d)_V)
\end{align*}
are canonically isomorphic strict polynomial functors, which we still denote by $\H(F,G)$. The following lemma summarizes the main properties of parameterized $\hom$ groups.
\begin{lemme}\label{lm-prop-param-hom}
Let $\k$ be a commutative ring. Parameterized $\hom$ groups yield a bifunctor:
$$
\begin{array}{ccc}
\P_\k^\op\times\P_\k &\to &\P_\k\\
(F,G)&\mapsto & \H(F,G)
\end{array}.
$$
If $F$, $G$ are homogeneous of weight $d$, then so is $\H(F,G)$. Moreover:
\begin{enumerate}
\item[(1)] If $F,G$ take values in $\V_\k$, Kuhn duality yields an isomorphism of strict polynomial functors $\H(F,G)\simeq \H(G^\sharp,F^\sharp)$, natural in $F,G$.
\item[(2)] If $G\in\P_{d,\k}$, there is an isomorphism $\H(\Gamma^d,G)\simeq G$, natural in $G$.
\item[(3)] Tensor products induce morphisms of strict polynomial functors:
$$\H(F,G)\otimes \H(F',G')\xrightarrow[]{\otimes} \H(F\otimes F',G\otimes G') \;.$$
\end{enumerate}
\end{lemme}

\begin{proof}
The first part of lemma \ref{lm-prop-param-hom} follows from lemma \ref{lm-def-H}. To prove (1), we can assume that $F,G$ are homogeneous of weight $d$. Since $(F^V)^\sharp=(F^\sharp)_V$, there is an isomorphism natural in $F,G$ and $f\in\hom_{\Gamma^d\V_\k}(V,V')$
$$\hom_{\P_{d,\k}}(F^V,G)\simeq \hom_{\P_{d,\k}}(G^\sharp,(F^V)^\sharp)=\hom_{\P_{d,\k}}(G^\sharp,(F^\sharp)_V)\;.$$
Whence the result. (2) is the Yoneda lemma. For (3), we can assume that $F,G$ (resp. $F',G'$) are homogeneous of degree $d$ (resp. $e$). The map:
$$\hom_{\P_{d,\k}}(F^V,G)\otimes \hom_{\P_{e,\k}}((F')^W,G')\xrightarrow[]{\otimes} \hom_{\P_{d+e,\k}}(F^V\otimes (F')^W,G\otimes G')$$
is natural with respect to $f\in\hom_{\Gamma^d\V_\k}(V,V')$ and $g\in\hom_{\Gamma^e\V_\k}(W,W')$ (i.e. it is a morphism of strict polynomial bifunctors). Hence it becomes a morphism of strict polynomial functors if one takes $V=W$.
\end{proof}

To define parameterized extension groups, we fix for each $F$ a projective resolution $P_F$ in $\P_\k$. We let $\E(F,G)$ be the homology of the complex of strict polynomial functors $\H(P_F,G)$. The following proposition follows directly from lemma \ref{lm-prop-param-hom}.
\begin{proposition}\label{prop-strpolstruct}
Let $\k$ be a commutative ring and let $F,G$ be strict polynomial functors over $\k$. For all $i\ge 0$, the functors 
$$V\mapsto \Ext^i_{\P_\k}(F^V,G)\quad\text{ and }\quad V\mapsto \Ext^i_{\P_\k}(F,G_V)$$
are isomorphic strict polynomial functors. We denote them by $\E^i(F,G)$. 
This yields bifunctors:
$$
\begin{array}{ccc}
\P_\k^\op\times\P_\k &\to &\P_\k\\
(F,G)&\mapsto & \E^i(F,G)
\end{array}.
$$
The homogeneous part of weight $d$ of the strict polynomial functor $\E^i(F,G)$ is $\E^i(F_d,G_d)$, where $F_d$ and $G_d$ denote the homogeneous parts of $F$ and $G$ of weight $d$. Moreover, if $F,G$ take values in $\V_\k$, Kuhn duality induces an isomorphism 
$\E^i(F,G)\simeq \E^i(G^\sharp,F^\sharp)$. Finally, tensor products induce morphisms of strict polynomial functors
$$\E^i(F,G)\otimes \E^j(F',G')\to \E^{i+j}(F\otimes F',G\otimes G')\;.$$
\end{proposition}

\subsection{Convolution products}\label{subsec-convol}

We now introduce the convolution product `$\star$' on the parameterized $\hom$-groups between a $\P_\k$-coalgebra $C$ and a $\P_\k$-dg-algebra $A$. This convolution product can be defined more generally when $C$ is a differential graded object, see e.g. \cite[Chap. 2]{LodayValette}, but we shall only need the case when $C$ is concentrated in degree zero. So we only describe the latter case, where the signs are slightly simpler.

\begin{definition}
Let $C$ be a $\P_\k$-coalgebra and let $A$ be a $\P_\k$-dg-algebra. We denote by $\H(C,A)$ the $\P_\k$-dg-algebra defined as follows.
\begin{itemize}
\item[(i)] The homogeneous part of $\H(C,A)$ of degree $i$ and weight $d$ equals $\H(C_{0,d}, A_{i,d})$.
\item[(ii)] The differential $\partial'$ of $\H(C,A)$ is given by postcomposing by the differential of $A$:
$$ \partial':=\H(C_{0,d},\partial ):\H(C_{0,d},A_{i,d})\to \H(C_{0,d}, A_{i-1,d})\;. $$ 
\item[(iii)] The convolution product $\star$ is defined as the composite:
\begin{align*}\H(C_{0,d}, A_{i,d})\otimes \H(C_{0,e}, A_{j,e})\xrightarrow[]{\otimes}\H(C_{0,d}\otimes C_{0,d}&, A_{i,e}\otimes A_{j,e})\\\xrightarrow[]{\H(\Delta,m)} &\H(C_{0,d+e}, A_{i+j,d+e})\;.\end{align*}
\item[(iv)] The unit $\eta'$ is induced by the unit of $A$ and the counit of $C$: 
$$\eta':=\H(\epsilon,\eta):\k= \H(\k,\k)\to \H(C_{0,0},A_{0,0})\;. $$
\end{itemize}
\end{definition}

If $C$ is coaugmented (with coaugmentation $\eta:\k\to C$) and $A$ is augmented (with augmentation $\epsilon:A\to \k$), the convolution algebra $\H(C,A)$ is augmented, with augmentation $\H(\eta,\epsilon): \H(C,A)\to \H(\k,\k)=\k$. Thus, convolution algebras yield a functor:
$$\H(C,-):\{\P_\k\text{-dga-alg}\}\to \{\P_\k\text{-dga-alg}\}\;.$$
The following lemma is an easy check.

\begin{lemme}
Assume that $C$ is a \emph{commutative} $\P_\k$-coalgebra, and that $A$ is a \emph{graded commutative} $\P_\k$-dg-algebra. Then $\H(C,A)$ is graded commutative.
\end{lemme}

We have the following elementary computations of convolution
algebras.
\begin{lemme}\label{lm-element}
Let $\k$ be a commutative ring. Let $C$ be a $\P_\k$-coalgebra, and let $A$ be a
$\P_\k$-algebras. The following isomorphisms of $\P_\k$-algebras hold.
\begin{align*}
\H(\Gamma,A)\simeq A\;,\qquad \H(C,S)\simeq C^\sharp\;,\qquad
\H(S,\Gamma)\simeq \Gamma\;. 
\end{align*}
Moreover, $\H(S,\Lambda)\simeq\H(\Lambda,\Gamma)$ is isomorphic to $\Gamma$ if
$2=0$ in $\k$, and to $\Lambda$ otherwise. 
\end{lemme}
\begin{proof}
We shall only prove the case of $\H(\Lambda,\Gamma)$. First, we have:
$$\H(\Lambda^d,\otimes^d)(V)\simeq \hom_{\P_\k}(\Lambda^d,\otimes^d_V)\simeq
\hom_{\P_\k}(\Lambda^d,\otimes^d)\otimes V^{\otimes d}\;.$$
The $\k$-module $\hom_{\P_{\k}}(\Lambda^d,\otimes^d)$ is free of rank one, with
generator the antisymmetrization map $\Delta_d:\Lambda^d\to \otimes^d$. So the isomorphism above can be rewritten as 
an isomorphism of functors 
$$\H(\Lambda^d,\otimes^d)\simeq
\otimes^d\;.\qquad (*)$$
For all $\sigma\in\Si_d$, we
identify $\sigma$ with the morphism $\otimes^d\to \otimes^d$ which permutes the
factors of the tensor product. Since $\sigma\circ \Delta_d=\epsilon(\sigma)\Delta_d$, the $\Si_d$-module 
$\hom_{\P_{\k}}(\Lambda^d,\otimes^d)$, where $\sigma$ acts as $\hom(\Lambda^d,\sigma)$, is isomorphic to $\k$ with action of $\Si_d$ given by the signature. So the isomorphism $(*)$ is $\Si_d$-equivariant if
$\sigma$ acts as $\H(\Lambda^d,\sigma)$ on the left hand side, and as $\epsilon(\sigma)\sigma$ on the right hand side.

Now, $\Gamma^d$ (resp. $\Lambda^d$ if $2\ne 0$ in $\k$) is the intersection of the kernels of the maps
$\Id-\sigma:\otimes^d\to\otimes^d$ (resp. $\Id-\epsilon(\sigma)\sigma$), for $\sigma\in\Si_d$. So by left exactness of parameterized $\hom$s we obtain that $\H(\Lambda^d,\Gamma^d)$ is isomorphic to $\Lambda^d$ if $2\ne 0$ in $\k$ and to $\Gamma^d$ if $2=0$ in $\k$.

It remains to identify the products. Let $\Delta_{d,e}:\Lambda^{d+e}\to \Lambda^d\otimes\Lambda^e$ denote the comultiplication of $\Lambda$. Since $(\Delta_d\otimes\Delta_e)\circ\Delta_{d,e}=\Delta_{d+e}$, the composite
$$\H(\Lambda^d,\otimes^d)\otimes \H(\Lambda^e,\otimes^e)\xrightarrow[]{\otimes}\H(\Lambda^d\otimes\Lambda^e,\otimes^{d+e})\to \H(\Lambda^{d+e},\otimes^{d+e})$$
identifies with the identity map $(\otimes^d)\otimes(\otimes^e)\to \otimes^{d+e}$. We identify the product on $\H(\Lambda^d,\Gamma^d)$ by viewing it as a subalgebra of the signed shuffle algebra, with $\otimes^d$ in weight $d$ and with product $\sum_{s\in S(d,e)}\epsilon(\sigma)\sigma:(\otimes^d)\otimes(\otimes e)\to \otimes^{d+e}$, where $S(d,e)$ is the set of $(d,e)$-shuffles.
\end{proof}

\begin{remark}\label{rk-Chal}
In characteristic $2$, \cite[Thm 3.2]{Chalupnik2} asserts that the $\P_\k$-algebra $\H(\Lambda^d,\Gamma^d)$ is isomorphic to $\Lambda\otimes\Gamma^{(1)}$. The elementary computation of lemma \ref{lm-element} shows that this is false. (The problem in the proof of \cite[Thm 3.2]{Chalupnik2} is that there is no reason why $v\mapsto \alpha^{(i)}_s(v)$ and $v\mapsto \beta^{(i)}_s(v)$ should be $\k$-linear).
\end{remark}

We also have a convolution product (still denoted by $\star$) 
on the parameterized extension groups $\E(C,A)$ between a $\P_\k$-coalgebra and
a $\P_\k$-graded algebra (compare \cite[p. 675]{FFSS}).

\begin{definition}\label{def-conv}
Let $C$ be a $\P_\k$-coalgebra and let $A$ be a $\P_\k$-algebra. We denote by $\E(C,A)$ the $\P_\k$-graded algebra defined as follows.
\begin{itemize}
\item[(i)] The homogeneous part of degree $-i$ and weight $d$ of $\E(C,A)$ is the functor
$\E^i(C_{0,d}, A_{0,d})$.
\item[(ii)] The unit is the map $\E^0(\epsilon,\eta)=\H(\epsilon,\eta)$.
\item[(iii)] The convolution product $\star$ is defined as the composite:
\begin{align*}\E^i(C_{0,d}, A_{0,d})\otimes \E^j(C_{0,e}, A_{0,e})\xrightarrow[]{\otimes}\E^{i+j}(C_{0,d}\otimes & C_{0,e}, A_{0,d}\otimes A_{0,e})\\\xrightarrow[]{\E(\Delta,m)}& \E^{i+j}(C_{0,d+e}, A_{0,d+e})\;.\end{align*}
\end{itemize}
\end{definition}

The $\Ext$-algebras considered in \cite{FFSS} and in \cite{Chalupnik2}
correspond to our strict polynomial graded algebras $\E(X^{(r)},Y^{(s)})$, for
pairs $(X,Y)$ of classical exponential functors.

\subsection{$(1,\epsilon)$-commutativity}\label{subsec-1E-com}

The algebras $\E(C,A)$ are usually \emph{not} graded commutative, even if $C$ and $A$ are commutative. Indeed, one easily checks the following lemma (cf. \cite[Lemma 1.11]{FFSS}).

\begin{lemme}\label{lm-commute-sign}
Let $X,Y$ be a pair of classical exponential functors and let $r,s$ be nonnegative integers (take $r=s=0$ if $\k$ is not a field of positive characteristic). Let $\epsilon(S)=\epsilon(\Gamma)=0$ and let $\epsilon(\Lambda)=1$. The following diagram commutes up to a $(-1)^{ij+\epsilon(X)k\ell+\epsilon(Y)mn}$ sign.

$$
\xymatrix{
\E^i(X^{k(r)},Y^{m(s)})\otimes \E^j(X^{\ell (r)},Y^{n(s)})\ar[rd]^-{\star}\ar[dd]_-{x\otimes y\mapsto y\otimes x} &\\
& \E^{i+j}(X^{(k+\ell)\,(r)},Y^{m+n(s)})\\
\E^i(X^{\ell (r)},Y^{n(s)})\otimes \E^j(X^{k(r)},Y^{m(s)})\ar[ru]_-{\star} 
}
$$
\end{lemme}
\begin{proof}
The $(-1)^{ij}$ sign is just the usual homological sign which comes from the commutativity of tensor products at the level of chain complexes $C\otimes D\simeq D\otimes C$, see e.g. \cite[proof of lemma 6.7.12]{Weibel}. The signs $(-1)^{\epsilon(X)k\ell}$ and $(-1)^{\epsilon(Y)mn}$ are the signs needed to have the following diagrams commute.
$$ 
\xymatrix{
X^{(k+\ell)\,(r)}\ar[r]\ar[rd]& X^{k (r)}\otimes X^{\ell (r)}\ar[d]^-{\simeq}\\
& X^{\ell (r)}\otimes X^{k (r)}
}
\quad
\xymatrix{
Y^{m(s)}\otimes Y^{n(s)}\ar[d]^-{\simeq}\ar[r] &Y^{m+n(s)}\\
Y^{n(s)}\otimes Y^{m(s)}\ar[ru]&&
}
 $$
\end{proof}

In order to formalize the graded commutativity defect of the convolution product $\star$, we make the following definition.

\begin{definition}
Let $A$ be a wdg-algebra, and let $\epsilon\in\{0,1\}$. We say that $A$ is $(1,\epsilon)$-commutative if for all homogeneous elements $x,y$ in $A$ we have:
$$ x\cdot y = (-1)^{|x||y|+\epsilon w(x)w(y)} y\cdot x\;. $$ 
We say that a $\P_\k$-dg-algebra is $(1,\epsilon)$-commutative if for all $V\in\V_\k$, the wdg-algebra $A(V)$ is $(1,\epsilon)$-commutative.
\end{definition}

Thus, the $(1,0)$-commutative $\P_\k$-dg-algebras are precisely the $\P_\k$-cdg-algebras. The $\P_\k$-graded algebras $S[i]$ and $\Gamma[i]$ are $(1,0)$-commutative for $i$ even, and $(1,1)$-commutative for $i$ odd, whereas $\Lambda[i]$ is $(1,1)$-commutative for $i$ even and $(1,0)$-commutative for $i$ odd. We  reformulate lemma \ref{lm-commute-sign} in terms of $(1,\epsilon)$-commutativity.
\begin{lemme}
Let $X,Y$ be a pair of classical exponential functors and let $r,s$ be nonnegative integers (take $r=s=0$ if $\k$ is not a field of positive characteristic). Let $\epsilon(S)=\epsilon(\Gamma)=0$ and let $\epsilon(\Lambda)=1$. Let $\epsilon\in\{0,1\}$ such that $\epsilon=\epsilon(X)+\epsilon(Y)\mod 2$. Then $\E(X^{(r)},Y^{(s)})$ is $(1,\epsilon)$-commutative.
\end{lemme}

If $A$ and $B$ are wdg-algebras which are $(1,1)$-commutative, their tensor product $A\otimes B$ is not $(1,1)$-commutative. For this reason, we introduce a `signed tensor product' $\otimes^\epsilon$.

\begin{definition}
Let $\epsilon\in\{0,1\}$. And let $A$ and $B$ be wdg-algebras. We define the wdg-algebra $A\otimes^\epsilon B$ in the following way.
\begin{itemize}
\item[(i)] If $a\in A$ and $b\in B$ are homogeneous elements, the degree of $a\otimes b$ is defined by $|a\otimes b|=|a|+|b|$ and its weight by $w(a\otimes b)=w(a)+ w(b)$.
\item[(ii)] The unit of $A\otimes^\epsilon B$ is the tensor product of the units of $A$ and $B$.
\item[(iii)] For all homogeneous elements $a\in A$ and $b\in B$, the differential of $A\otimes^\epsilon B$ is given by $d(a\otimes b)=da\otimes b+(-1)^{|a|}a\otimes db$.
\item[(iv)] For all homogeneous elements $a,a'\in A$ and $b,b'\in B$, the product of $A\otimes^\epsilon B$ is given by
$$(a\otimes b)\cdot (a'\otimes b')= (-1)^{|a'||b|+\epsilon w(a')w(b)} (aa')\otimes (bb')\;.$$
\end{itemize}
\end{definition}

When $\epsilon=0$, the tensor product $\otimes^\epsilon$ is the usual tensor product of graded algebras. When $\epsilon=1$, the definition differs from the usual tensor product by the sign involved in the product. One defines a tensor product $\otimes^\epsilon$ on $\P_\k$-dg-algebras analogously. 

\begin{lemme}\label{lm-tens-E}
Let $A$ and $B$ be wdg-algebras. Then $A\otimes^\epsilon B$ is $(1,\epsilon)$-commutative if and only if $A$ and $B$ are $(1,\epsilon)$-commutative. 
\end{lemme}
\begin{proof} If $A\otimes^\epsilon B$ is $(1,\epsilon)$-commutative, then $A$ (resp. $B$) identifies with the subalgebra of the elements of the form $a\otimes 1$ (resp. $1\otimes b$), so it is also $(1,\epsilon)$-commutative.
Conversely, let $A$ and $B$ be $(1,\epsilon)$-commutative.
Let $a,a'\in A$ and $b,b'\in B$ be homogeneous elements. We denote $f(x,y)=|x||y|+\epsilon w(x)w(y)$. We have 
\begin{align*}
(a\otimes b)\cdot (a'\otimes b')&= (-1)^{f(a',b)} (aa')\otimes (bb') ,\\&= (-1)^{f(a', b)+f(a,a')+f(b,b')}(a'a)\otimes (b'b),\\
&= (-1)^{f(a', b)+f(a,a')+f(b,b')+f(a,b')}(a'\otimes b')\cdot (a\otimes b).
\end{align*}
The sign appearing on the last line equals $f(a\otimes b,a'\otimes b')$, so $A\otimes^\epsilon B$ is $(1,\epsilon)$-commutative.
\end{proof}

\section{The exponential property}\label{sec-exp}

In this section, we recall the basics of exponential functors. Archetypes of exponential functors are symmetric, exterior and divided powers, this is why we call them `classical exponential functors'. We prove that when $C$ and $A$ are exponential functors, the parameterized extension groups $\E(C,A)$ are exponential functors.

\subsection{Exponential functors}\label{subsec-exp}

\begin{definition}
Let $A$ be a $\P_\k$-graded algebra. Then $A$ is called an exponential functor if (i) each summand $A_{i,d}$ of degree $i$ and weight $d$ takes values in $\V_\k$, and (ii) the following composite is an isomorphism of graded $\k$-modules (here $\iota_V$ and $\iota_W$ are the canonical inclusions of $V$ and $W$ in $V\oplus W$, and $m$ is the product in $A$).
$$A(V)\otimes A(W)\xrightarrow[]{A(\iota_V)\otimes A(\iota_W)} A(V\oplus W)^{\otimes 2}\xrightarrow[]{m}A(V\oplus W)$$

\end{definition}

\begin{example}
If $i$ is an integer, the $\P_\k$-graded algebras $S[i]$, $\Lambda[i]$ and $\Gamma[i]$ from section \ref{subsec-strpoldgalg} are exponential functors. Also, if $F$ is an additive functor (e.g. one of the functors $V\otimes$, $\hom_\k(V,-)$ or $I^{(r)}$ from notation \ref{nota-1}), and if $A$ is an exponential functor, then the composite $A(F)=A\circ F$ is an exponential functor.
\end{example}

We refer to the isomorphism $A(V\oplus W)\simeq A(V)\otimes A(W)$ as the \emph{exponential isomorphism}. By definition, the multiplication of an exponential functor determines the exponential isomorphism. Conversely, if $A$ is an exponential functor, the multiplication of $A$ can be recovered from the exponential isomorphism, as the composite (where $\Sigma$ is the map $(x,y)\mapsto x+y$):
$$A(V)\otimes A(V)\xrightarrow[]{\simeq} A(V\oplus V)\xrightarrow[]{A(\Sigma)}A(V)\;. $$
Similarly, if $A$ is an exponential functor, the exponential isomorphism can also be used to define a comultiplication $\Delta$ on $A$, as the composite (where $\delta$ is the map $x\mapsto (x,x)$):
$$A(V)\xrightarrow[]{A(\delta)}A(V\oplus V)\xrightarrow[]{\simeq} A(V)\otimes A(V)\;. $$ 
And the exponential isomorphism can be recovered from the comultiplication $\Delta$ as the composite (where $\pi_V$ and $\pi_W$ denote the canonical projections) 
$$A(V\oplus W)\xrightarrow[]{\Delta}A(V\oplus W)^{\otimes 2}\xrightarrow[]{A(\pi_V)\otimes A(\pi_W)}A(V)\otimes A(W) \;.$$  
So, if $A$ is an exponential functor, each one of the following data determines the two other ones: the multiplication, the exponential isomorphism, the comultiplication.
\begin{example}
For the classical exponential functors $X=S$, $\Lambda$ or $\Gamma$, the coproduct determined by the usual multiplication (i.e. the one defined by the canonical maps $X^d\otimes X^e\to X^{d+e}$ indicated in section \ref{subsec-strpoldgalg}) is the usual coproduct.
\end{example}
The following lemma is straightforward, we record it for further use.

\begin{lemme}\label{lm-morph}
Let $A,B$ be exponential functors and let $f:A\to B$ be a morphism of graded functors. The following assertions are equivalent.
\begin{enumerate}
\item[(i)] $f$ is a morphism of $\P_\k$-graded algebras, 
\item[(ii)] $f$ commutes with the exponential isomorphisms,
\item[(iii)] $f$ is a morphism of $\P_\k$-graded coalgebras.
\end{enumerate}

\end{lemme}

\subsection{The exponential property for extension groups}\label{subsec-exp-appl}

We now show that (under mild hypotheses), the parameterized extension groups $\E(C,A)$ between exponential functors are an exponential functor. We begin with a statement on the $\hom$-level. 

\begin{lemme}\label{lm-facile}
Let $E$ be an exponential functor, and let $F,G$ be strict polynomial functors with values in $\V_\k$. Let us denote $E_i$ by the summand of $E$ of degree $i$: $E_i=\bigoplus_d E_{i,d}$. Assume that for all $i$, $\H(E_i,F)$ and $\H(E_i,G)$ take finitely generated projective values. Then the composite  is a graded isomorphism (where the morphism on the right is induced by the comultiplication of $E$, and the total degree is taken on the left hand side)
$$\bigoplus_{i\ge 0}\H(E_i,F)\otimes \bigoplus_{j\ge 0}\H(E_j,G) \xrightarrow[]{\otimes} \bigoplus_{i,j\ge 0}\H(E_i\otimes E_j,F\otimes G)\to  \bigoplus_{k\ge 0}\H(E_k,F\otimes G)\;.$$
Similarly, if the $\hom$-groups $\H(F,E_i)$ and $\H(G,E_i)$ are in $\V_\k$, the following composite (where the morphism on the right is induced by the multiplication of $E$, and the total degree is taken on the left hand side) is an isomorphism:
$$\bigoplus_{i\ge 0}\H(F,E_i)\otimes \bigoplus_{j\ge 0}\H(G,E_j) \xrightarrow[]{\otimes} \bigoplus_{i,j\ge 0}\H(F\otimes G, E_i\otimes E_j)\to  \bigoplus_{k\ge 0}\H(F\otimes G,E_k)\;.$$
\end{lemme}
\begin{proof} The Kuhn dual of an exponential functor is an exponential functor. So by lemma \ref{lm-prop-param-hom}, the first isomorphism of lemma \ref{lm-facile} is equivalent to the second one via Kuhn duality. Thus we only prove the latter.

Fix $V\in\V_\k$ and denote by $A_{i,d}$ the functor $(E_{i,d})_V$ and by $[F,A_{i,d}]$ the $\k$-module $\hom_{\P_\k}(F,A_{i,d})$. We have to prove that the following composite, called $\Phi$ in the sequel, is an isomorphism. 
$$\bigoplus_{i+j=k}[F,A_{i,d}]\otimes [G,A_{j,e}] \xrightarrow[]{\otimes} \bigoplus_{i+j=k}[F\otimes G, A_{i,d}\otimes A_{j,e}]\to  [F\otimes G,A_{k,d+e}]  $$

{\bf Case $F=\Gamma^{d\,X}$, $G=\Gamma^{e\, Y}$.}
In this case, the source and the target of $\Phi$ are bifunctors with variables $X\in\Gamma^{d}\V_\k$ and $Y\in\Gamma^e\V_\k$, and $\Phi$ is a natural transformation of bifunctors. 
But $F\otimes G$ identifies as the direct summand of weight $d$ with respect to $X$ and weight $e$ with respect to $Y$ of the bifunctor $\Gamma^{\ell,X\oplus Y}$. So $\Phi$ is the homogeneous part of weight $d$ with respect to $X$ and $e$ with respect to $Y$ of the morphism (where the downwards map 
 is induced by the multiplication of $A$ and by the diagonal $\Gamma^{\ell\,X\oplus Y}\to \Gamma^{d\,X}\otimes\Gamma^{e\,Y}$):
$$\xymatrix{
\displaystyle\bigoplus_{\text{\tiny$\begin{array}{c}i+j=k\\d+e=\ell\end{array}$}}[\Gamma^{d\,X},A_{i,d}]\otimes [\Gamma^{e\, Y},A_{j,e}] \ar[r]^-{\otimes}& \displaystyle\bigoplus_{\text{\tiny$\begin{array}{c}i+j=k\\d+e=\ell\end{array}$}}[\Gamma^{d\,X}\otimes \Gamma^{e\, Y}, A_{i,d}\otimes A_{j,e}]\ar[d]\\ &  [\Gamma^{\ell\,X\oplus Y},A_{k,\ell}] }$$
The latter identifies through the Yoneda isomorphism with the composite
$$\xymatrix{
\displaystyle\bigoplus_{\text{\tiny$\begin{array}{c}i+j=k\\d+e=\ell\end{array}$}}A_{i,d}(X)\otimes A_{j,e}(Y) \ar[r]^-{\otimes}& \displaystyle\bigoplus_{\text{\tiny$\begin{array}{c}i+j=k\\d+e=\ell\end{array}$}}A_{i,d}(X\oplus Y)\otimes A_{j,e}(X\oplus Y)\ar[d]\\
&  A_{k,\ell}(X\oplus Y)\;, 
}$$
where the downwards map is induced by the multiplication of $A(X\oplus Y)$.
But $A=E_V$ is an exponential functor, so the latter composite is an isomorphism. Whence the result for $F=\Gamma^{d\,X}$, $G=\Gamma^{e\, Y}$.

{\bf General case.} By additivity of (parameterized) $\hom$ groups, we can restrict to the case of homogeneous functors $F\in\P_{d,\k}$ and $G\in \P_{e,\k}$.
Now $[F,A_{i,d}] $ and $[G,A_{j,e}]$ are projective $\k$-modules, and so are $[\Gamma^{d,X},A_{i,d}] $ and $[\Gamma^{e,Y},A_{j,e}] $ (by the Yoneda lemma, since the homogeneous summands of $A$ have values in $\V_\k$). Hence the result for arbitrary $F,G$ follows from the result for $F=\Gamma^{d,X}$ and $G=\Gamma^{e,Y}$ by left exactness of $\hom$s, when taking projective presentations of $F$ and $G$.
\end{proof}

\begin{remark}
Alternatively, one can prove lemma \ref{lm-facile} by using first the sum-diagonal adjunction as in the proof of \cite[Thm 1.7]{FFSS}, and then identifying the isomorphism obtained, as in \cite[Lemma 5.13]{TouzeClassical}. 
\end{remark}

If $\k$ is a field, then the hypothesis that the $\hom$-groups $\H(E^i,F)$ and $\H(E^i,G)$ are finitely generated and projective over $\k$ is automatically satisfied. The next result follows from lemma \ref{lm-facile} by taking projective resolutions and using the K\"unneth formula.

\begin{lemme}[Compare {\cite[Thm 1.7]{FFSS}}]\label{lm-facile-Ext}
Let $\k$ be a field, let $E$ be an exponential functor, and let $F,G$ be strict polynomial functors with values in $\V_\k$. The following composite is an isomorphism.
$$ \bigoplus_{i+j=k}\E(F,E_i)\otimes \E(F,E_j) \xrightarrow[]{\otimes} \bigoplus_{i+j=k}\E(F\otimes G,E_i\otimes E_j)\to \E(F\otimes G,E_k) $$
\end{lemme}

\begin{proposition}\label{prop-expo-inut}
Let $\k$ be a field, let $X,Y$ be exponential functors concentrated in degree $0$ (i.e. $X_{i,d}=Y_{i,d}=0$ for $i>0$). The $\P_\k$-graded algebra $\E(X,Y)$ is a an exponential functor.
\end{proposition}
\begin{proof} First, since $\k$ is a field and the functors $X_{0,d}$ and $Y_{0,d}$ take finite dimensional values, the functors $\E^i(X_{0,d},Y_{0,d})$ take finite dimensional values. So it remains to check the exponential isomorphism.
The following composite (where the morphism on the left hand side is induced by the canonical inclusions of $V$, $W$ in $V\oplus W$, and the morphism on the right hand side is given by the convolution product of $\E(X,Y)$) :
$$\E(X,Y)(V)\otimes \E(X,Y)(W)\to \E(X,Y)(V\oplus W)^{\otimes 2}\to \E(X,Y)(V\oplus W)$$
is an isomorphism. Indeed, it equals the composite
$$\E(X,Y_{V})(\k)\otimes \E(X,Y_{W})(\k)\to \E(X,Y_{V}\otimes Y_{W})(\k)\to \E(X,Y_{V\oplus W})(\k), $$
where the first map is the isomorphism of lemma \ref{lm-facile-Ext} and the second one is induced by the isomorphism $Y_{V}\otimes Y_{W}\simeq Y_{V\oplus W}$.
\end{proof}

\begin{remark}\label{rk-intro}
We make no use of proposition \ref{prop-expo-inut} in this article. We have stated it only to justify that it is \emph{a priori} not worthy to care about the coproduct on $\E(X,Y)$ as we claimed it in the introduction. Indeed, as observed in section \ref{subsec-exp}, if we know $\E(X,Y)$ as a $\P_\k$-graded algebra, the coproduct is determined by the product (thus, on our computation in section \ref{sec-final-results}, the obvious candidate for the coproduct is the good one!). Notice that such a reasoning does not work if one restricts to computing the unparameterized extension groups $\Ext^*_{\P_\k}(X^*,Y^*)=\E(X,Y)(\k)$, because the functoriality is needed to recover the coproduct.
\end{remark}

\part{Extension groups between $S$, $\Lambda$, $\Gamma$, and bar constructions}\label{part-2}

In this part, $\k$ is a commutative ring. We compute the $\P_\k$-graded algebras $\E(S,\Lambda)$ and $\E(S,\Gamma)$ in terms of the homology of the (iterated) bar constructions of $\Gamma$. We first introduce regrading functors
$$\R_\alpha:\{\P_\k\text{-dg-alg}\}\to \{\P_\k\text{-dg-alg}\}\;.$$
These functors automatically take care of all the strange signs which arise in the computations.
Then, we prove that $\E(S,\Lambda)$ and $\E(S,\Gamma)$ are equal to the homology of the $\P_\k$-dg-algebras $^t\R_{2i+1}\B(\Gamma[2i])$ and $\R_{2i+2}\B^2(\Gamma[2i])$. This is done in theorem \ref{thm-bar}, which generalizes the  main theorem of \cite{A}\footnote{The main theorem of \cite{A} corresponds to the case of $\E(S,\Lambda)(\k)$ (hence without functoriality), without the algebra structure.}. The key point in the proof is the interchange property of proposition \ref{prop-key}, which is very specific to exponential functors.

\section{Regrading functors}\label{sec-regrading}

\begin{definition}
Let $\alpha$ be an integer and let $\big(A,\eta,m,\partial\big)$ be a strict polynomial differential graded algebra. The regraded algebra is the algebra $(\R_\alpha A,\R_\alpha\eta,\R_\alpha m,\R_\alpha \partial)$  defined by the following formulas.
\begin{enumerate}
\item[(i)] $(\R_\alpha A)_{i,d}= A_{i+\alpha d,d}$.
\item[(ii)] $\R_\alpha\eta=\eta$
\item[(iii)] $\R_\alpha m: (\R_\alpha A)_{i,d}\otimes (\R_\alpha A)_{j,e}\to (\R_\alpha A)_{i+j,d+e}$ equals $(-1)^{\alpha (i+d)e}m$.
\item[(iv)] $\R_\alpha \partial:(\R_\alpha A)_{i,d}\to (\R_\alpha A)_{i-1,d}$ equals $(-1)^{\alpha d}\partial$.
\end{enumerate}
\end{definition}

Let $V\in\V_\k$ be a finitely generated projective $\k$-module, and let $x,y$ be a pair of homogeneous elements of the weighted graded $\k$-module $A(V)$. We denote by $|x|$ the degree of $x$ in $A(V)$ and by $w(x)$ its weight. We denote by $s^{\alpha w(x)}x$ and $s^{\alpha w(y)}y$ the same elements, viewed as elements of the weighted graded $\k$-module $\R_\alpha A(V)$. Then the definitions of the product and of the differential take the more suggestive form (where the signs which appear are determined by the usual Koszul sign rule, see e.g. \cite[Chap 1]{LodayValette}):
\begin{align*}&s^{\alpha w(x)}x\cdot s^{\alpha w(y)}y:= (-1)^{\alpha|x|w(y)} s^{\alpha w(x\otimes y)}x\cdot y\;,\\&(R_\alpha\partial)(s^{\alpha w(x)}x):= (-1)^{\alpha w(x)}s^{\alpha w(x)}(\partial x)\;. 
\end{align*}
It is not hard to see from these formulas that the definition of $\R_\alpha A$ makes sense, i.e. that the product $\R_\alpha m$ is associative, and the differential $R_\alpha\partial$ is a derivation. 

\begin{definition}
If $f:A\to B$ is a morphism of $\P_\k$-dg-algebras, we define $\R_\alpha f$ by $(\R_\alpha f)_{i,d}=f_{i+\alpha d,d}$. 
This yields a regrading functor: 
$$\R_\alpha:\{\P_\k\text{-dg-alg}\}\to \{\P_\k\text{-dg-alg}\}\;.$$
\end{definition}

Let us denote by $H$ the homology functor $H:\{\P_\k\text{-dg-alg}\}\to\{\P_\k\text{-g-alg}\}$, and recall from section \ref{subsec-convol} the convolution algebra functor $\H(C,-)$. It is easy to check that the regrading functor commutes with these two functors.

\begin{lemme}\label{lm-prop-facile}
Let $\alpha$ be an integer and let $C$ be a $\P_\k$-coalgebra. For all $\P_\k$-dg-algebra $A$, we have equalities of $\P_\k$-graded-algebras, and $\P_\k$-differential graded algebras:
$$\R_\alpha (HA)=H(\R_\alpha A)\;,\qquad \R_\alpha\H(C,A)=\H(C,\R_\alpha A)\;.$$ 
\end{lemme}

If $\alpha$ is even, one easily checks that $\R_\alpha$ is an isomorphism of categories, with inverse $\R_{-\alpha}$. This is not the case if $\alpha$ is odd. We first need to introduce a new definition.

\begin{definition}
Let $(A,\eta,m,\partial)$ be $\P_\k$-dg-algebra. The weight twisted algebra $^\t A$ is the $\P_\k$-dg-algebra $(A,\eta,{^\t m},\partial)$ where $^t m:  A_{i,d}\otimes A_{j,e}\to A_{i+j,d+e}$ equals $(-1)^{de}m$. If $f:A\to B$ is a morphism of algebras, we define $^\t f=f$. This yields an involutive functor:
$$^\t:\{\P_\k\text{-dg-alg}\}\to \{\P_\k\text{-dg-alg}\}\;.$$
\end{definition}

\begin{lemme}\label{lm-propt}
Let $A,B$ be $\P_\k$-dg-algebras, and let $\epsilon\in\{0,1\}$.
\begin{itemize}
\item[(i)] $A$ is $(1,\epsilon)$-commutative if and only if $^\t A$ is.
\item[(ii)] There is an isomorphism of $\P_\k$-dg-algebras $^\t(A\otimes^\epsilon B)\simeq({}^\t A)\otimes^\epsilon ({}^\t B)$.
\end{itemize}
\end{lemme}
\begin{proof} Let $V\in\V_\k$, and let $a,a'\in A(V)$ and $b,b'\in B(V)$ be homogeneous elements. Let us prove $(i)$. Since $^\t$ is an involution, it suffices to prove the only if part. So we assume that $A$ is $(1,\epsilon)$-commutative. We denote by $a\cdot^\t a'$ the product in $A(V)$ and by $a\cdot a'$ the product in $A(V)$. We have:
\begin{align*}a\cdot^\t a'=(-1)^{w(a)w(a')}a\cdot a'&=(-1)^{w(a)w(a')+|a||a'|+\epsilon w(a)w(a')}a'\cdot a\\&= (-1)^{|a||a'|+\epsilon w(a)w(a')}a'\cdot^\t a\;.\end{align*}
Thus, $^\t A$ is $(1,\epsilon)$-commutative. Now we turn to $(ii)$. Let $\psi(x,y)=w(x)w(y)$. We define: 
$$\begin{array}{cccc}
\Psi : &  {}^\t(A\otimes^\epsilon B) & \to & (^\t A)\otimes^{\epsilon} (^\t B)\\
& a\otimes b &\mapsto & (-1)^{\psi(a,b)}a\otimes b
\end{array}.$$
It is straightforward to check that $\Psi$ preserves the degrees, the weights, the units and the differentials. We have to check that $\psi$ is multiplicative. We denote  $f(x,y)=|x||y|+\epsilon w(x)w(y)$. In ${}^\t(A\otimes^\epsilon B)$, the product is defined by:
$$(a\otimes b)\cdot (a'\otimes b')=(-1)^{\psi(a\otimes b,a'\otimes b')+f(a',b)}aa'\otimes bb'\;. $$
In $(^\t A)\otimes^{\overline{\epsilon}} (^\t B)$, the product is defined by:
$$(a\otimes b)\cdot (a'\otimes b')=(-1)^{\psi(a, b)+\psi(a',b')+f(a',b)}aa'\otimes bb'\;. $$
Then $\Psi((a\otimes b)\cdot (a\otimes b'))$ equals $\Psi(a\otimes b)\cdot \Psi(a\otimes b')$ up to a $(-1)^n$ sign, where $n$ equals the sum:
$$\psi(a\otimes b,a'\otimes b')+2f(a',b)+\psi(aa',bb')+\psi(a,b)+\psi(a',b')+\psi(a,a')+\psi(b,b')\;.$$
One readily checks that this sum is even. Hence $\psi$ is multiplicative.
\end{proof}

\begin{lemme}\label{lm-RalphaR-alpha}Let $\alpha$ be an integer. We have $^\t \R_\alpha = \R_\alpha {^\t}$.  If $\alpha$ is even, we have 
$\R_{-\alpha} \R_\alpha =\Id$. If $\alpha$ is odd, we have $\R_{-\alpha} \R_\alpha=\,^\t$.
\end{lemme}
\begin{proof}
The proof that $^\t$ and $\R_\alpha$ commute is straightforward. We concentrate on the identification of $\R_{-\alpha}\R_\alpha$. Let $A$ be a $\P_\k$-dg-algebra. It is clear that $\R_{-\alpha}\R_\alpha A= (A,\eta,\R_{-\alpha}\R_\alpha m,\partial)$, so we only have to identify the product of $\R_{-\alpha}\R_\alpha A$. Let $V\in\V_\k$, and let $x,y$ be homogeneous elements in $A(V)$. We define $\varpi(x)=\alpha w(x)$. Then we have:
\begin{align*}
s^{-\varpi(x)}s^{\varpi(x)} x\cdot s^{-\varpi(y)} s^{\varpi(y)} y &= (-1)^{-\varpi(y)(\varpi(x)+|x|)}s^{-\varpi(x\otimes y)} (s^{\varpi(x)} x\cdot s^{\varpi(y)} y)\;,\\
&=(-1)^{-\varpi(y)\varpi(x)}s^{-\varpi(x\otimes y)}s^{\varpi(x\otimes y)}x\cdot y\;.
\end{align*} If $\alpha$ is even, the sign in the latter equality equals one, so $\R_{-\alpha}\R_\alpha A=A$. If $\alpha$ is odd, the sign equals $(-1)^{w(x)w(y)}$ so $\R_{-\alpha}\R_\alpha A={^\t}A$.
\end{proof}

In particular, $\R_\alpha$ is an isomorphism of categories. The behavior of $\R_\alpha$ with respect to graded commutativity also depends on the parity of $\alpha$.

\begin{lemme}
Let $\alpha$ be an integer, let $\epsilon\in\{0,1\}$ and let $\overline{\epsilon}\in\{0,1\}$ such that $\overline{\epsilon}=\epsilon+1\mod 2$. Let $A$ be a $\P_\k$-dg-algebra. 
\begin{itemize}
\item If $\alpha$ is even, then $\R_\alpha A$ is $(1,\epsilon)$-commutative if and only if $A$ is $(1,\epsilon)$-commutative.
\item If $\alpha$ is odd, then $\R_\alpha A$ is $(1,\epsilon)$-commutative if and only if $A$ is $(1,\overline{\epsilon})$-commutative.
\end{itemize}
\end{lemme}
\begin{proof}
By lemma \ref{lm-propt} $^\t$ preserves $(1,\epsilon)$-commutativity. Thus, by lemma \ref{lm-RalphaR-alpha}, it suffices to prove the `only if' part of the statements. Let us assume that $A$ is $(1,\epsilon)$-commutative. Let $V\in\V_\k$. Let $x,y$ be homogeneous elements of $A(V)$. We define $\varpi(x)=\alpha w(x)$. Then we have:
\begin{align*}
s^{\varpi(y)} y\cdot s^{\varpi(x)} x &= (-1)^{|y|\varpi(x)} s^{\varpi(x\otimes y)} y\cdot x \\&= (-1)^{|y|\varpi(x)+|x||y|+\epsilon w(x)w(y)}s^{\varpi(x\otimes y)}x\cdot y\;,\\
s^{\varpi(x)} x\cdot s^{\varpi(y)} y &= (-1)^{|x|\varpi(y)} s^{\varpi(x\otimes y)} x\cdot y\;.
\end{align*} 
Thus $s^{\varpi(y)} y\cdot s^{\varpi(x)} x$ equals $s^{\varpi(x)} x\cdot s^{\varpi(y)} y$ up to a $(-1)^n$ sign, with
\begin{align*}n= &|y|\varpi(x)+|x||y|+ |x|\varpi(y)+\epsilon w(x)w(y)\\
= & (|x|+\varpi(x))(|y|+\varpi(y)) + (\alpha+\epsilon) w(x)w(y) \mod 2\;.
\end{align*}
But $|x|+\varpi(x)$ is the degree of $s^{\varpi(x)} x\in\R_\alpha A(V)$, thus the last equality shows that $\R_\alpha A$ is $(1,\epsilon+\alpha)$-commutative.
\end{proof}

\begin{lemme}\label{lm-Rodd-prod}
Let $\alpha$ be an integer, let $\epsilon\in\{0,1\}$ and let $\overline{\epsilon}\in\{0,1\}$ such that $\overline{\epsilon}=\epsilon+1\mod 2$. If $\alpha$ is even, then for all $\P_\k$-dg-algebras $A$, $B$ there is an isomorphism of $\P_\k$-dg-algebras:
$$R_\alpha(A\otimes^\epsilon B)\simeq \R_\alpha(A)\otimes^\epsilon \R_{\alpha}(B)\;.  $$
If $\alpha$ is odd, there is an isomorphism of $\P_\k$-dg-algebras:
$$R_\alpha(A\otimes^\epsilon B)\simeq \R_\alpha(A)\otimes^{\overline{\epsilon}} \R_{\alpha}(B)\;.  $$
\end{lemme}
\begin{proof}We do the proof for $\alpha$ odd (the case $\alpha$ even is elementary).
We define $\varpi(x)=\alpha w(x)$. If $a,a'\in A$ and $b,b'\in B$ are homogeneous elements, we denote by $s^{\varpi(a\otimes b)}a\otimes b$ the element $a\otimes b$ viewed as an element of $R_\alpha(A\otimes^\epsilon B)$, and by $(s^{\varpi(a)}a)\otimes (s^{\varpi(b)}b)$ the element $a\otimes b$ viewed as an element of $\R_\alpha(A)\otimes^{\overline{\epsilon}} \R_{\alpha}(B)$. We also define $\phi(a,b)= |a|\varpi(b)$. It is straightforward to check that the map
$$\begin{array}{cccc}
\Phi : &  R_\alpha(A\otimes^\epsilon B) & \to & \R_\alpha(A)\otimes^{\overline{\epsilon}} \R_{\alpha}(B)\\
& s^{\varpi(a\otimes b)}a\otimes b &\mapsto & (-1)^{\phi(a,b)}(s^{\varpi(a)}a)\otimes (s^{\varpi(b)}b)
\end{array}$$
preserves the degrees, the weights, the units and the differentials. We have to prove that $\Phi$ preserves the products. Let us write for short $f(x,y)=|x||y|+\epsilon w(x)w(y)$ and $g(x,y)=(\varpi(x)+|x|)(\varpi(y)+|y|)+\overline{\epsilon} w(x)w(y)$. The multiplication in $R_\alpha(A\otimes^\epsilon B)$ is defined by:
\begin{align*}
\left(s^{\varpi(a\otimes b)}a\otimes b\right)\cdot &\left(s^{\varpi(a'\otimes b')}a'\otimes b'\right) \\&=(-1)^{\phi(a\otimes b,a'\otimes b')}s^{\varpi(a\otimes b\otimes a'\otimes b')}(a\otimes b)\cdot (a'\otimes b')\;,\\ 
&= (-1)^{\phi(a\otimes b,a'\otimes b')+f(a',b)} s^{\varpi(aa')+\varpi(bb')}(aa')\otimes (bb')\;.
\end{align*}
The multiplication in $\R_\alpha(A)\otimes^{\overline{\epsilon}} \R_{\alpha}(B)$ is defined by:
\begin{align*}
\left(s^{\varpi(a)}a\otimes  s^{\varpi(b)}b\right)\cdot &\left(s^{\varpi(a')}a'\otimes s^{\varpi(b')}b'\right) \\
&=(-1)^{g(b,a')}\left(s^{\varpi(a)}a\cdot s^{\varpi(a')}a'\right)\otimes \left(s^{\varpi(b)}b\cdot s^{\varpi(b')}b'\right)\;,\\
&=(-1)^{g(b,a')+\phi(a,a')+\phi(b,b')} (s^{\varpi(aa')}aa')\otimes (s^{\varpi(bb')}bb')\;.
\end{align*}
Thus, we obtain:
$$\Phi(s^{\varpi(a\otimes b)}a\otimes b\cdot s^{\varpi(a'\otimes b')}a'\otimes b')= (-1)^n\Phi(s^{\varpi(a\otimes b)}a\otimes b)\cdot \Phi(s^{\varpi(a'\otimes b')}a'\otimes b')\;, $$
where the integer $n$ equals the following sum:
$$g(b,a')+\phi(a,a')+\phi(b,b')+\phi(a\otimes b,a'\otimes b')+f(a',b)+\phi(a,b)+\phi(a',b')+\phi(aa',bb').$$
One readily checks that this sum is even. So $\Phi$ preserves products.
\end{proof}
We finish this section by an explicit computation involving $\R_{\alpha}$. Recall that $X[i]$ denotes the $\P_\k$-graded algebra with $X^d$ in degree $di$ and weight $d$, equipped with multiplication given by the canonical maps $X^d\otimes X^e\to X^{d+e}$. We also denote by $X[i]^{(r)}$ the $\P_\k$-graded algebra $X[i]$ precomposed by the Frobenius twist (so it has $X^{d\,(r)}$ in degree $di$ and weight $pd$).

\begin{lemme}\label{lm-exe} Let $X$ be $S,\Lambda$ or $\Gamma$.
Assume that $\alpha$ is even or that $i$ is even. Then we have the following equalities of $\P_\k$-graded algebras.
$$R_\alpha \left(X[i]^{(r)}\right)=X[i-\alpha p^r]\;.$$ 
Assume that both $\alpha$ and $i$ are odd. Then we have the following equalities of $\P_\k$-graded algebras.
$$R_\alpha \left(X[i]^{(r)}\right)={^\t \left(X[i-\alpha p^r]^{(r)}\right)}\;.$$ 
\end{lemme}
\begin{proof} The equalities of lemma \ref{lm-exe} are straightforward as equalities of graded strict polynomial functors. Let us identify the multiplication on $R_\alpha \left(X[i]^{(r)}\right)$. Let $V\in\V_\k$ and let $x,y$ be homogeneous elements in $X[i]^{(r)}(V)$. Then:
$$s^{\alpha w(x)} x\cdot s^{\alpha w(y)} y= (-1)^{\alpha|x|w(y)} s^{\alpha w(x\otimes y)} x\cdot y\;.$$
If $\alpha$ is even the sign equals $1$. If alpha is odd, and $i$ is even, all elements of $X[i](V)$ are in even degree, so the sign also equals $1$. If $\alpha$ and $i$ are odd, then $\alpha|x|=w(x)\mod 2$, so the sign equals $(-1)^{w(x)w(y)}$. Whence the result. 
\end{proof}

\section{Bar constructions and extension groups}

\subsection{Recollections of bar constructions}\label{subsubsec-recoll-bar} 

Let $A$ be a differential graded augmented $\k$-algebra (dga-$\k$-algebra). We denote by $\epsilon$ its augmentation and by $A'=\ker \epsilon$ the augmentation ideal of $A$. The degree of an homogeneous element $a\in A$ is denoted by $|a|$. The (reduced, normalized) bar construction over $A$ is the 
differential graded coaugmented $\k$-coalgebra $\overline{B}A$ defined as follows see e.g. \cite[X.10]{ML} or \cite[Chap. 2]{LodayValette}. 

\begin{itemize}
\item $\overline{B}A$ equals $\bigoplus_{n\ge 0} {A'}^{\otimes n}$ as a $\k$-module.
\item A scalar $\lambda$ of $\k={A'}^{\otimes 0}$ is denoted by $\lambda[]$ and has degree $0$. For $n\ge 1$, let $a_i$ be homogeneous elements of $A$, $1\le i\le n$. The element $a_1\otimes\dots\otimes a_n\in {A'}^{\otimes n}$ is denoted by $[a_1|\dots|a_n]$ and has degree $n+\sum |a_i|$.
\item The counit is the map $\B A\twoheadrightarrow {A'}^{\otimes 0}=\k$, and the coaugmentation is the map $\k={A'}^{\otimes 0}\hookrightarrow \B A$.
\item The differential $\partial:\overline{B}A_k\to \overline{B}A_{k-1}$ sends an element $[a_1|\dots|a_n]$ to the sum:
$$\sum_{i=1}^{n-1}(-1)^{e_i}[a_1|\dots|a_ia_{i+1}|\dots|a_n]-\sum_{i=1}^{n}(-1)^{e_{i-1}}[a_1|\dots|\partial a_i|\dots|a_n]\;,  $$
where $e_0=0$ and for $i\ge 1$, $e_i$ equals $i+\sum_{j\le i}|a_i|$.
\item The coproduct $\Delta: \overline{B}A\to \overline{B}A\otimes \overline{B}A$ sends an element $[a_1|\dots|a_n]$ to the sum $$\sum_{i=0}^{n} [a_1|\dots|a_i]\otimes [a_{i+1}|\dots|a_n]\;.$$
\end{itemize} 

When $A$ is \emph{graded commutative}, we can define a `shuffle product' on $\overline{B}A$, which makes $\overline{B}A$ into a cdga-$\k$-algebra. 
So we can iterate bar constructions, and we denote by $\overline{B}^n A$ the $n$-th iterated bar construction of $A$.
To be more specific, if $a_i$ are homogeneous elements of $A$, the shuffle product
$[a_1|\dots|a_p]*[a_{p+1}|\dots|a_{p+q}]$ equals
$$\sum\epsilon(\sigma)\,[a_{\sigma^{-1}(1)}|\dots |a_{\sigma^{-1}(p+q)}] $$
where the sum is taken over all $(p,q)$-shuffles $\sigma$, and $\epsilon(\sigma)$ is the Koszul sign such that $x_1\wedge \dots \wedge x_n= \epsilon(\sigma)x_{\sigma(1)}\wedge \dots \wedge x_{\sigma(n)} $ in the exterior algebra $\Lambda(x_1,\dots,x_n)$ over generators $x_i$ with degree $|a_i|+1$.

We have presented the (reduced, normalized) bar construction of a dga-$\k$-algebra, but the formulas above also make sense for the other categories of algebras described in section \ref{sec-hierarchy} (if the algebras have weights, the weights in $\B A$ are defined by $w([a_1|\dots|a_n])=\sum w(a_i)$).  In particular, for $\P_\k$-cdga-algebras, iterated bar constructions yield functors ($n\ge 0$):
$$
\{\P_\k\text{-cdga-alg}\}\xrightarrow[]{\B^n}\{\P_\k\text{-cdga-alg}\}.
$$

\subsection{Bar constructions of symmetric and exterior algebras}\label{subsubsec-bar-sym-ext}

Let us now concentrate on the concrete example of bar constructions of symmetric and exterior algebras. Recall that for all integer $i$, $S[i]$ (resp. $\Lambda[i]$) denotes the $\P_\k$-graded algebra, with $S^d$ (resp. $\Lambda^d$) in degree $di$ and weight $d$. We can write down explicitly $\B(S[i])$. 
For $d\ge 0$, the homogeneous part of weight $d$ of $\overline{B}(S[i])$ is the following complex of strict polynomial functors, which we denote by $I_{\bullet,d}$:
$$ \underbrace{{\otimes^d}}_{\text{\footnotesize$\begin{array}{c}\text{degree}\\ $d(i+1)$\end{array}$}}\to \bigoplus_{k=0}^{d-2}S^{1\,\otimes k}\otimes S^2\otimes S^{1\,\otimes d-k-2} \to \dots \to {\begin{array}{c}S^{d-1}\otimes S^1\\\oplus S^1\otimes S^{d-1} \end{array}}\to \underbrace{S^d}_{\text{\footnotesize$\begin{array}{c}\text{degree}\\ $di+1$\end{array}$}}\;.$$
Similarly, one gets the homogeneous part of weight $d$ of $\overline{B}(\Lambda[i])$ by replacing symmetric powers  by exterior ones in the complex above.

Assume that $S[i]$, resp. $\Lambda[i]$, are graded commutative (i.e. assume that $i$ is even, resp. odd, or that $\k$ has characteristic $2$). Then
the canonical inclusions $\Lambda^d\hookrightarrow {\otimes^d}$, resp. $\Gamma^d\hookrightarrow {\otimes^d}$,  define morphisms of $\P_\k$-cdga-algebras: $$\Lambda[i+1]\hookrightarrow \overline{B}(S[i])\quad\text{ and }\quad \Gamma[i+1]\hookrightarrow \overline{B}(\Lambda[i])\;.$$
These morphisms are quasi-isomorphisms, see e.g. \cite[Chap 3, ex 3.2.5]{LodayValette}.
In particular, the complex $I_{\bullet,d}$ drawn above yields a coresolution of $\Lambda^d$ by symmetric powers.

Let us recall from {\cite[X Th. 11.2]{ML}} that bar constructions preserve quasi isomorphisms of algebras. Hence, the composite 
$$\Gamma[i+2]\hookrightarrow \overline{B}(\Lambda[i+1]) \hookrightarrow \overline{B}^2(S[i])$$
is a quasi isomorphism. So the homogeneous part of weight $d$ of $\overline{B}^2(S[i])$ yields a complex $J_{\bullet,d}$ of symmetric powers, whose homology equals $\Gamma^d$ in degree $d(i+2)$ (and zero in other degrees). The complex $J_{\bullet,d}$ has the form:
$$
\underbrace{{\otimes^d}}_{\text{\footnotesize$\begin{array}{c}\text{degree}\\ $d(i+2)$\end{array}$}}\to \bigoplus_{k=1}^{d-2} {\otimes^d} \to \dots \to {\begin{array}{c}S^{d-1}\otimes S^1\\\oplus S^1\otimes S^{d-1} \end{array}}\to \underbrace{S^d}_{\text{\footnotesize$\begin{array}{c}\text{degree}\\ $di+2$\end{array}$}}\;. 
$$
By lemma \ref{lm-sym-res}, we can use the complexes $I_{\bullet,d}$ and $J_{\bullet,d}$ to compute $\E(F,S^d)$ and $\E(F,\Lambda^d)$ when $F$ takes values in the category $\V_\k$ of finitely generated projective $\k$-modules.

\begin{proposition}\label{prop-cplx-ext}Let $\k$ be a commutative ring, let $C$ be a $\P_\k$-coalgebra whose homogeneous components $C_{d,i}$ take values in $\V_\k$, and let $i$ be a nonnegative integer. The $\P_\k$-graded algebras $\E(C, \Lambda)$ and $\E(C, \Gamma)$ are respectively isomorphic to the homology of the convolution $\P_\k$-dg-algebras $^\t\R_{2i+1}\H(C,\overline{B}(S[2i]))$ and $\R_{2i+2}\H(C,\overline{B}^2(S[2i]))$.
\end{proposition}
\begin{proof}
We prove the case of $\E(C, \Lambda)$, the other case is similar. Let $K_d$ be a coresolution of $\Lambda^d$ by symmetric powers. Then $\E(C_{0,d},\Lambda^d)$ is the homology of the complex $\H(C_{0,d},K_d)$, and the convolution product is given on the cochain level by the composite:
$$\H(C_{0,d},K_d)\otimes \H(C_{0,e},K_e)\xrightarrow[]{\otimes}\H(C_{0,d}\otimes C_{0,e},K_d\otimes K_e)\xrightarrow[]{\H(\Delta_{d,e},\widetilde{m})}\H(C_{0,d+e},K_{d+e})\;, $$
where $\Delta_{d,e}$ is the comultiplication $C_{0,d+e}\to C_{0,d}\otimes C_{0,e}$ and $\widetilde{m}:K_d\otimes K_e\to K_{d+e}$ is a lifting of the multiplication $m:\Lambda^d\otimes\Lambda^e \to \Lambda^{d+e}$.

Consider the $\P_\k$-dg-algebra $^\t\R_{2i+1}(\B S[2i])$. It is quasi isomorphic to $^\t\R_{2i+1}(\Lambda[2i+1])=\Lambda$. So we may take for $K_d$ the homogeneous part of weight $d$ of this $\P_\k$-dg-algebra (this is $I_{\bullet,d}$, shifted), and for $\widetilde{m}$ the multiplication of $^\t\R_{2i+1}(\B S[2i])$.
With this choice, the composite above is the convolution product of $\H(C,\,^\t\R_{2i+1}(\B S[2i]))$. Thus, $\E(C,\Lambda)$ equals the homology of $\H(C,\,^\t\R_{2i+1}(\B S[2i]))$. By lemma \ref{lm-prop-facile}, the algebras 
$^\t\R_{2i+1}\H(C,\B S[2i])$ and $\H(C,^\t\R_{2i+1}(\B S[2i]))$ are equal, whence the result.
\end{proof}

\subsection{The interchange property}

This subsection is the core of our computation. Recall from section \ref{subsec-convol} that if $C$ is a coaugmented commutative $\P_\k$-coalgebra, we have a convolution algebra functor:
$$\H(C,-):\{\P_\k\text{-cdga-alg}\}\to \{\P_\k\text{-cdga-alg}\}\;.$$
The following key result asserts that when $C$ is exponential, the functor $\H(C,-)$ commutes with the bar construction functor $\B:\{\P_\k\text{-cdga-alg}\}\to \{\P_\k\text{-cdga-alg}\}$.

\begin{proposition}[Interchange property]\label{prop-key}
Let $\k$ be a commutative ring, let $E=\{E_{0,d}\}_{d\ge 0}$ be a commutative exponential functor concentrated in degree zero, and let $A=\{A_{i,d}\}_{i,d\ge 0}$ be a $\P_\k$-cdga-algebra. Assume that  $\H(E,A)$ takes $\k$-projective values. Then there is an isomorphism of $\P_\k$-cdga-algebras, natural with respect to $A$:
$$\H(E,\overline{B}A)\simeq \overline{B}\H(E,A) \;.$$
\end{proposition}

\begin{remark}
We encourage the reader to try to prove a statement similar to proposition \ref{prop-key} in the framework of Schur algebras to understand the subtle difference between strict polynomial functors and modules over Schur algebras.
\end{remark}

\begin{remark}
If $\k$ is a Dedekind domain (e.g $\k$ is a field or $\mathbb{Z}$), and $A$ has $\k$-projective values, the technical assumption that $\H(E,A)$ has $\k$-projective values is always satisfied.
\end{remark}

\begin{proof}[Proof of proposition \ref{prop-key}]
Let us write for short $[X]$ instead of $\H(E,X)$. Lemma \ref{lm-facile} yields an isomorphism $\theta_n: [A]^{\otimes n}\to [A^{\otimes n}]$. Taking the direct sum over all $n\ge 0$, we get an isomorphism of graded strict polynomial functors: $\theta:\B[A]\xrightarrow[]{\simeq}[\B A]$, natural with respect to $A$. We have to check that $\theta$ is an isomorphism of $\P_\k$-cdga-algebras. It is obvious that $\theta$ preserves augmentations.

By definition, the product $\star$ in the convolution algebra $[A]$ is the composite $[m]\circ \theta_2$, where $m$ is the multiplication of $A$. Thus we have a commutative diagram for all $k\le n$:
$$\xymatrix{
[A]^{\otimes n}\ar[r]^-{\simeq}_-{\theta_n}\ar[d]_-{[A]^{\otimes k}\otimes\star\otimes [A]^{\otimes n-k-2}}&[A^{\otimes n}]\ar[d]^-{[A^{\otimes k}\otimes m\otimes A^{\otimes n-k-2}]}\\
[A]^{\otimes k}\otimes[A]\otimes [A]^{\otimes n-k-2}\ar[r]^-{\simeq}_-{\theta_{n-1}}&[A^{\otimes n-1}]\;.
}$$
By definition of the differentials of $[\B A]$ and $\B [A]$, this commutative diagram implies that $\theta$ commutes with the differentials. 

Let $\sigma\in\Si_n$. Since $E$ is commutative, we have a commutative diagram:
$$\xymatrix{
[A]^{\otimes n}\ar[r]^-{\simeq}_-{\theta}\ar[d]^-{\sigma}&[A^{\otimes n}]\ar[d]^-{[\sigma]}\\
[A]^{\otimes n}\ar[r]^-{\simeq}_-{\theta}&[A^{\otimes n}]
}$$
where $\sigma:[A]^{\otimes n}\to [A]^{\otimes n}$ and $[\sigma]:[A^{\otimes n}]\to [A^{\otimes n}]$ are the maps induced by permuting the factors of the tensor products. By definition of the products of $\B [A]$ and $[\B A]$, this commutative diagram implies that $\theta$ is an isomorphism of algebras. Whence the result.
\end{proof}

\begin{theoreme}\label{thm-bar}Let $\k$ be a commutative ring, and let $i$ an integer.
The $\P_\k$-graded algebras $\E(S,\Lambda)$ and $\E(S,\Gamma)$ are respectively computed by the homology of the $\P_\k$-dg-algebras
$^\t\R_{2i+1}\overline{B}(\Gamma[2i])$ and $\R_{2i+2}\overline{B}^2(\Gamma[2i])$.

\end{theoreme} 

\begin{proof} We first observe that $\H(S,S[2i])=\Gamma[2i]$, hence has $\k$-projective values. Thus, proposition \ref{prop-key} and lemma \ref{lm-element} yield an isomorphism:
$$\H(S,\B (S[2i]))\simeq \B\H(S,S[2i])\simeq \B(\Gamma[2i])\;.$$
But $\Gamma[2i]$, hence also $\B(\Gamma[2i])$ has $\k$-projective values. Thus, by the isomorphism above, $\H(S,\B(S[2i]))$ has $\k$-projective values. Applying proposition \ref{prop-key} again, we get an isomorphism:
$$\H(S,\B^2 (S[2i]))\simeq \B\H(S,\B(S[2i]))\simeq \B^2(\Gamma[2i])\;. $$ 
Now the result follows from proposition \ref{prop-cplx-ext}.
\end{proof}

\begin{remark}\label{rk-Ext-Lambda-Lambda}
In proposition \ref{prop-key}, it is essential that $E$ is commutative. If this is not the case, then $\H(E,A)$ is not graded commutative, so $\B \H(E,A)$ does not bear a multiplication. However, it is easy to check (same proof) that if we drop the commutativity of $E$, we still have an isomorphism of complexes of strict polynomial functors:
$$\H(E,\B A)\simeq \B \H(E,A)\;.$$
Applying this to $E=A=\Lambda$, we get an isomorphism of complexes $\H(\Lambda,\B S)\simeq \B\H(\Lambda,S)$. But $\H(\Lambda,S)\simeq \Lambda$, so the homology of $\H(\Lambda,\B S)$ equals $\Gamma[1]$ (as a graded strict polynomial functor). Since the homology of $\R_1(\H(\Lambda,\B S))$ computes $\E(\Lambda,\Lambda)$ (same proof as proposition \ref{prop-cplx-ext}, without taking the algebra structure into account), we finally obtain (cf. lemma \ref{lm-element}):
$$\E(\Lambda,\Lambda)=\H(\Lambda,\Lambda)\simeq\Gamma\;.$$
This provides a proof that there are no extension groups between $\Lambda$ and $\Lambda$, as asserted in the introduction.
\end{remark}

\part{Homology of Eilenberg-Mac Lane spaces}\label{part-3}

In theorem \ref{thm-bar} in part \ref{part-2}, we proved that $\E(S,\Lambda)$ and $\E(S,\Gamma)$ are equal (up to regrading) to the homology of iterated bar constructions of $\Gamma[2]$. In part \ref{part-3}, we tackle the following problem.

\begin{probleme}
\label{pb-HB}
Compute the homology of the differential graded $\P_\k$-algebras $\B^n(\Gamma[2])$. 
\end{probleme}

\section{Bar constructions and Eilenberg-Mac Lane spaces}
Akin observed \cite{A} that over a field $\k$, the homology of $\B\Gamma(\k[2])$ can be interpreted as the dual of the rational cohomology of the additive group, computed in \cite{CPSVdK}. 
However, this approach solves problem 
only for $n=1$ and when the ground ring $\k$ is a field. To solve problem \ref{pb-HB} in general, we rather interpret the homology of $\B^n(\Gamma[2])$ as the singular homology of some Eilenberg Mac Lane (EML) spaces. Indeed, for an abelian group $\pi$, the singular homology with coefficients in $\k$ of the EML space $K(\pi,n)$  is isomorphic \cite{EML} to the homology of $\B^n(\k\pi)$, where $\k\pi$ denotes the group algebra of $\pi$. 
When $\pi$ is free abelian, we can choose a $\k$-linear section of the canonical map $\k\pi\twoheadrightarrow \pi\otimes_\Z\k$. By the universal property of the symmetric algebra, this section induces a morphism of $\k$-algebras $\psi:S(\pi\otimes_\Z \k)\to \k\pi$.

\begin{proposition}\label{prop-S-k}
Let $\k$ be a commutative ring, let $\pi$ be a free abelian group, and let $\psi:S(\pi\otimes_\Z \k)\to \k\pi$ be a morphism of $\k$-algebras build from a section of the canonical map $\k\pi\twoheadrightarrow \pi\otimes_\Z\k$.
Then for all positive integer $n$, the map 
$$\B^n(\psi):\B^n(S(\pi\otimes_\Z \k))\to \B^n(\k\pi)$$ 
is a quasi-isomorphism. It is not natural with respect to $\pi$, however the map $H(\B^n(\psi))$ induced on homology is natural with respect to $\pi$.
\end{proposition}
\begin{proof}
Let us first prove that $\B(\psi)$ is a quasi-isomorphism. Let $s$ denote the section of the canonical map $q:\k\pi\twoheadrightarrow \pi\otimes_\Z\k$. Since 
all the elements of $\overline{B}(\k\pi)_1=\k\pi$ are cycles and the differential graded algebra $\B(\k\pi)$ is graded commutative, products yield a morphism of differential graded $\k$-algebras (take the trivial differential on the left hand side)
$$\Lambda(s):\Lambda(\pi\otimes_\Z\k[1])\to \B(\k\pi)\;. $$ 
It is well known that $\Lambda(s)$ is a quasi-isomorphism, see e.g. \cite[V.6 Th 6.4(ii)]{Brown}. Now $\B(\psi)$ is a morphism of algebras, so $\Lambda(s)$ factors as: 
$$\Lambda(\pi\otimes_\Z\k[1])\hookrightarrow \B(S(\pi\otimes_\Z\k))\xrightarrow[]{\B(\psi)} \B(\k\pi) $$
As recalled in section \ref{subsubsec-bar-sym-ext} the first map is a quasi-isomorphism, and since $\Lambda(s)$ is a quasi-isomorphism, so is $\B(\psi)$. 
Bar constructions preserve quasi-isomorphisms, so for all $n\ge 1$, $\B^n(\psi)$ is a quasi-isomorphism.

It remains to show that the map induced on homology by $\B^n(\psi)$ is natural with respect to $\pi$. We have a commutative diagram of differential graded $\k$-algebras:
$$\xymatrix{
&\B^n(S(\k\pi))\ar[ld]_-{\B^n(S(q))}\ar[rd]^-{\B^n(\phi)} &\\
\B^n(S(\pi\otimes_\Z\k))\ar[rr]^-{\B^n(\psi)}&& \B^n(\k\pi)
}\;,$$
where $\phi:S(\k\pi)\to \k\pi$ is the morphism of $\k$-algebras induced by the universal property of the symmetric algebra and the identity map of $\k\pi$. The map $\B^n(S(q))$ is surjective in homology, indeed $\B^n(S(s))$ provides a section of $\B^n(S(q))$. Since $\B^n(\psi)$ is a quasi-isomorphism, $\B^n(\phi)$ is also surjective in homology. So for all $\pi$ we have a commutative diagram:
$$\xymatrix{
&H(\B^n(S(\k\pi)))\ar@{->>}[ld]_-{H(\B^n(S(q)))}\ar@{->>}[rd]^-{H(\B^n(\phi))} &\\
H(\B^n(S(\pi\otimes_\Z\k)))\ar[rr]^-{H(\B^n(\psi))}&& H(\B^n(\k\pi))\;.
}$$
By their definition, the two maps with source $H(\B^n(S(\k\pi)))$ are natural with respect to $\pi$. Since they are surjective, an easy diagram chase proves that $H(\B^n(\psi))$ is natural with respect to $\pi$.
\end{proof}

\begin{remark}
Let $\pi$ be an abelian group. The fact that the homology of the bar constructions $\B^n S(\pi\otimes_\Z\k)$ and $\B^n \k\pi$ are isomorphic is well known \cite[4.16]{DP}. The fact that we can build an isomorphism natural with respect to $\pi$ is less known (and it is actually false if $\pi$ is not a free abelian group).
\end{remark}

Let us denote for short by $H(n,\k)$ the functor which sends a free abelian group $\pi$ to the graded algebra $H_*^{\mathrm{sing}}(K(\pi,n),\k)=H_*(\B^n(\k\pi))$. Proposition \ref{prop-S-k} has the following consequence.

\begin{corollary}\label{cor-fonc-alg} Let $n$ be a nonnegative integer.
If $\k=\Z$, there is an isomorphism of $\F_\Z$-graded algebras $$H(\B^n \Gamma[2])\simeq H(n+2,\Z)\;.$$ 
If $p$ is a prime and $\k=\Fp$, $H(n,\Fp)$ factors into a $\F_\Fp$-graded algebra $\overline{H}(n,\Fp)$ and we have an isomorphism of $\F_\Fp$-graded algebras
$$H(\B^n \Gamma[2])\simeq \overline{H}(n+2,\Fp)\;.$$
\end{corollary}

Corollary \ref{cor-fonc-alg} shows that we can recover the homology of $\B^n(\Gamma[2])$, as a $\F_\k$-graded algebra, when the ground ring $\k$ is $\Z$ or $\Fp$, from the algebraic topology computations of \cite{Cartan}. But for our purposes, we need more. We need the $\P_\k$-graded algebra structure of the homology of $\B^n(\Gamma[2])$ (or at least the data of the weights) in order to be able to apply our regrading functors from section \ref{sec-regrading}. We also wish to compute the homology of $\B^n(\Gamma[2])$ over an arbitrary ground field $\k$ (not just a prime field). So in this part of the paper, we elaborate on Cartan's results to obtain a more satisfactory answer to problem \ref{pb-HB}. 
Our first main result is
theorem \ref{thm-BnGamma-F}, which solves problem \ref{pb-HB} when $\k$ is an arbitrary field. Our second main result is theorem \ref{thm-BGamma-Z}, which describes $H_*(\B^{n}(\Gamma(\Z^m\, [2])))$ as a graded algebra with weights for $\k=\Z$. We deduce from these theorems an explicit computation of 
$\E(S,\Lambda)$ and $\E(S,\Gamma)$. Section \ref{sec-dpa} is a preparatory work for our first main result. 

\section{Strict polynomial structures on divided power algebras}\label{sec-dpa}

Let $\k$ be a field of prime characteristic. If a functor $F\in\F_\k$ is graded, we can form the $\F_\k$-graded algebra:
$$U(F)=\Gamma(F_{\mathrm{even}})\otimes \Lambda(F_{\mathrm{odd}})$$ 
We denote this algebra by $U(F)$ because it satisfies a universal property, see theorem \ref{thm-univ-Car} below. Recall from section \ref{sec-hierarchy} the forgetful functor $$\U:\{\P_\k\text{-graded alg.}\}\to \{\F_\k\text{-graded alg.}\}\;.$$ 
In this section, we determine all the $\P_\k$-graded algebras $A$ such that $\U A= U(F)$, if $F$ is additive. We also do the same for $\Gamma(F)$ in characteristic $2$.

\subsection{Strict polynomial structures on $U(F)$}
We first need a few results about additive strict polynomial functors. We say that a functor $F\in\P_\k$ is additive if the underlying functor $\U F\in\F_\k$ is additive.
\begin{lemme}\label{lm-classif-add} Let $\k$ be a field of prime characteristic. 
\begin{itemize}
\item (Classification) If $F\in\P_\k$ is additive with finite dimensional values, then $F$ either equals zero or is a finite direct sum of Frobenius twists $I^{(r)}$ (with possibly different $r\ge 0$).
\item (Retracts) Let $F,G$ be additive functors with finite dimensional values, and let $f\in \hom_{\P_\k}(F,G)$. Then $(i)$ and $(ii)$ are equivalent.
\begin{itemize}
\item[(i)] There exists $V\in\V_\k$ such that the $\k$-linear map $f_V:F(V)\to G(V)$ is surjective
\item[(ii)] There exists $\iota\in \hom_{\P_\k}(G,F)$ such that $f\circ\iota=\Id_F$.
\end{itemize}
\end{itemize}
\end{lemme}
\begin{proof} To prove the classification, we can assume that $F$ is homogeneous of degree $d$. If $d=0$, then $F$ is constant and additive, hence $F=0$. So let us assume that $d\ge 1$. There are two cases.

{\bf Case 1: $d$ is not a power of $p$.} The map 
$$\begin{array}{ccc}
\End_\k(\k)&\to & \End_\k(F(\k))\\
f &\mapsto &F(f)  
\end{array}$$
is given by a homogeneous polynomial of degree $d$, which is additive. Since $d$ is not a power of $p$, the only such polynomial is the zero polynomial. Thus, $F$ sends the identity map of $\k$ to zero.  So the identity map of $F(\k)$ equals zero. Thus $F(\k)=0$. By additivity of $F$ this implies that $F=0$.

{\bf Case 2: $d=p^r$, for $r\ge 0$.} Assume that $F\ne 0$ and fix an integer $n\ge p^r$. Recall from \cite[Cor 3.13]{FS} that evaluation on $\k^n$, yields an equivalence of categories $\P_{p^r,\k}\simeq S(n,p^r)\mathrm{-mod}$ (where $S(n,p^r)$ is the Schur algebra). So it suffices to prove that $F(\k^n)$ is a direct sum of copies of $(\k^n)^{(r)}$. 

Let $\k_i$ denote the vector space $\k$ acted on by the torus $\mathbb{G}_m^{\times n}$ by $(\lambda_1,\dots,\lambda_n)\cdot x = \lambda_i\cdot x$. Since $F$ is homogeneous of degree $p^r$, $F(\k_i)$ is acted on by $\mathbb{G}_m^{\times n}$ by $(\lambda_1,\dots,\lambda_n)\cdot x = \lambda_i^{p^r}\cdot x$. Additivity of $F$ yields a $\mathbb{G}_m^{\times n}$-equivariant isomorphism 
$$F(\k^n)=F(\textstyle\bigoplus_{i\le n} \k_i)\simeq\textstyle\bigoplus_{i\le n} F(\k_i)\;.$$ 
As a consequence, all the weights of the $S(n,p^r)$-module $F(\k^n)$ are of the form $(\mu_1,\dots,\mu_n)$ with all $\mu_i=0$ but one which equals $p^r$. In particular, if 
$S_1(\k^n),\dots,S_N(\k^n)$ is the composition series of $F(\k^n)$, the $S_i(\k^n)$ are finite direct sums of simples with highest weight $(p^r,0,\dots,0)$, that is of $(\k^n)^{(r)}$.
We know (see e.g. 
\cite{FS}) that $\Ext^1((\k^n)^{(r)},(\k^n)^{(r)})=0$. Thus there cannot be nontrivial extensions between finite direct sums of  $(\k^n)^{(r)}$. This implies that the composition series of $F(\k^n)$ has length $N=1$. Thus $F(\k^n)$ is a finite direct sum of copies of $(\k^n)^{(r)}$.

Finally, let us prove the characterization of retracts. We can assume that $F,G$ are homogeneous of degree $p^r$, $r\ge 1$. The identity map is a basis of $\hom_{\P_\k}(I^{(r)},I^{(r)})$, so tensor products yield isomorphisms for $k,\ell\ge 1$:
$$
\hom_{\k}(\k^k,\k^\ell) \simeq \hom_{\P_\k}(I^{(r)}\otimes\k^k,I^{(r)}\otimes\k^\ell),\quad
f\mapsto \Id\otimes f,
$$
and the result follows.\end{proof}

If $A$ is an augmented $\P_\k$-(or $\F_\k$-)graded  algebra, we denote by $Q(A)$ the indecomposables of $A$, that is $Q(A)$ is the cokernel of the multiplication $A'\otimes A'\to A'$, where $A'$ is the augmentation ideal of $A$. 
If $A$ is a $\P_\k$-graded augmented algebra, then $Q(A)$ is a graded strict polynomial functor and $\U Q(A)=Q(\U A)$. 
Similarly the primitives $P(C)$ of a $\P_\k$-graded coaugmented coalgebra $C$ form a graded strict polynomial subfunctor of $C$ and $\U P(C)=P(\U C)$. 
The following lemma explains the link between additive and exponential functors.

\begin{lemme}\label{lm-exp-add}
Let $\k$ be a field, and let $E$ be an exponential functor. The graded strict polynomial functors $P(E)$ and $Q(E)$ are additive.
Moreover, if there exists $V\in\V_\k$ such that the composite $P(E)(V)\hookrightarrow E(V) \twoheadrightarrow Q(E)(V)$ is surjective, then $Q(E)$ is a direct summand in $E$.
\end{lemme}
\begin{proof}
Let $E'$ be the augmentation ideal of $E$. Since $E$ is exponential, we have 
$E'(V\oplus W)$ is isomorphic to $E'(V)\otimes\k\,\oplus \,\k\otimes E'(W)\,\oplus\, E'(V)\otimes E'(W)$. Moreover, the multiplication $E'(V\oplus W)^{\otimes 2}\to E'(V\oplus W)$ identifies through this decomposition with the direct sum of three maps (which are induced by multiplications):
\begin{align*}
&(E'(V)\otimes\k)^{\otimes 2}\to E'(V)\otimes\k,\\
&(\k\otimes E'(W))^{\otimes 2}\to \k\otimes E'(W),\\
&(E'(V)\otimes\k)\otimes (\k\otimes E'(W)) \oplus \text{\small\begin{tabular}{c} other summands\\ of $E'(V\oplus W)^{\otimes 2}$\end{tabular}} \to E'(V)\otimes E'(W).
\end{align*}
The first two maps have respective cokernels $Q(E)(V)$ and $Q(E)(W)$ and the last one is surjective. This shows that $Q(E)$ is additive.
The proof that the primitives are additive is similar.

Finally, if the map $P(E)(V)\to Q(E)(V)$ is surjective for some $V\in\V_\k$, then by lemma \ref{lm-classif-add}, it admits a section $\iota$. So the composite $Q(E)\xrightarrow[]{\iota} P(E)\hookrightarrow E$ is a section of $E\twoheadrightarrow Q(E)$.
\end{proof}
We are now ready to prove the main result of section \ref{sec-dpa}.
\begin{theoreme}\label{thm-strictpolalg1}
Let $\k$ be a field of prime characteristic $p$, and for all $d\ge 0$, let $F_d\in\F_\k$ be a finite direct sum of $n_d$ copies of the identity functor, and let $F=\bigoplus_{d\ge 0} F_d$
The graded $\P_\k$-algebras $A$ satisfying $\U A= U(F)$  are of the form:
$$ A=\Gamma(G_{\mathrm{even}})\otimes \Lambda(G_{\mathrm{odd}})\;, $$ 
where $G=\bigoplus_{d\ge 0} G_d$ and each $G_d\in\P_\k$ is a direct sum of $n_d$  Frobenius twists $I^{(r)}$ (with possibly different $r\ge 0$).
\end{theoreme}
\begin{proof}
{\bf Step 1: Duality.} Let $E$ be a graded exponential (non strict polynomial) functor. Then finding the $\P_\k$-graded algebras $A$ such that $\U A=E$ is equivalent to finding the $\P_\k$-graded algebras $B$ such that $\U B=E^\sharp$ as algebras.
Indeed, if $\U A=E$, then  for all $V,W\in\V_\k$ the composite
$$A(V)\otimes A(W)\to A(V\oplus W)^{\otimes 2}\to A(V\oplus W)$$
is an isomorphism (indeed, this is true for $\U A=E$, and the forgetful functor $\U$ reflects isomorphisms). Thus $A$ is an exponential functor, and $\U A$ coincides with $E$ as an exponential functor. Equivalently, $\U (A^\sharp)$ coincides with $E^\sharp$ as an exponential functor. This is in turn equivalent to the fact that $B=A^\sharp$ is a graded strict polynomial algebra such that $\U B$ coincides with $E^\sharp$ as a $\F_\k$-graded algebra.

So, to prove theorem \ref{thm-strictpolalg1}, it suffices to prove that the graded strict polynomial algebras $B$ such that $\U B=S(F_{\mathrm{even}})\otimes \Lambda(F_{\mathrm{odd}})$ are of the form $S(G_{\mathrm{even}})\otimes \Lambda(G_{\mathrm{odd}})$ with $G$ as indicated.

{\bf Step 2: Indecomposables.} If $B$ is as indicated in step 1, then the indecomposables of $B$ are a direct summand in $B$. Indeed, since $\U B=S(F_{\mathrm{even}})\otimes \Lambda(F_{\mathrm{odd}})$, there exists $V\in\V_\k$, e.g. $V=\k$, such that the composite $P(B)(V)\to B(V)\to Q(B)(V)$ is surjective. Then one applies lemma \ref{lm-exp-add}.

Now, the indecomposables of $B$ is an additive strict polynomial functor $Q(B)$ satisfying $\U Q(B)= Q(\U B)= F$. 
So by lemma \ref{lm-classif-add}, $Q(B)_d$ is a finite direct sum of $n_d$ Frobenius twists for all $d\ge 0$.

{\bf Step 3: Universal property.} The morphism of graded strict polynomial functors $Q(B)\hookrightarrow B$ induces a morphism of $\P_\k$-graded algebras $S(Q(B)_{\mathrm{even}})\otimes \Lambda(Q(B)_{\mathrm{odd}})\to B$. For all $V\in \V_\k$, this morphism is an isomorphism after evaluation on $V$. Hence, it is an isomorphism.
Thus  $B$ is of the form $S(G_{\mathrm{even}})\otimes \Lambda(G_{\mathrm{odd}})$ with $G=Q(B)$ as indicated in the statement of theorem \ref{thm-strictpolalg1}, which concludes the proof.
\end{proof}

\subsection{Strict polynomial structures on $\Gamma(F)$} Now we work over a field $\k$ of characteristic $2$. If $F$ is a graded strict polynomial functor, then the symmetric algebra $S(F)$ is commutative, so we can adapt the proof of theorem \ref{thm-strictpolalg1} to get the following result (the commutativity of $S(F)$ is needed in the last step of the proof, the remainder of the proof is unchanged).
\begin{theoreme}\label{thm-strictpolalg2}
Let $\k$ be a field of characteristic $2$, and for all $d\ge 0$, let $F_d\in\F_\k$ be a finite direct sum of $n_d$ copies of the identity functor, and let $F:=\bigoplus_{d\ge 0} F_d$.
The graded $\P_\k$-algebras $A$ satisfying $\U A= \Gamma(F)$  are of the form: $A=\Gamma(G)$, 
where $G=\bigoplus_{d\ge 0} G_d$ and each $G_d\in\P_\k$ is a direct sum of $n_d$  Frobenius twists $I^{(r)}$ (with possibly different $r\ge 0$).
\end{theoreme}

\section{Explicit computations over a field}\label{sec-bar-Gamma-F}

\subsection{Systems of divided powers on an algebra}

In this section, $\k$ is a commutative ring. We recall the basics of systems of divided powers. The reader can take \cite[Exp. 7, 8]{Cartan}, \cite[Appendix 2]{Eisenbud} 
as references. 

\begin{definition}
Let $A$ be a graded commutative $\k$-algebra. A system of divided powers on $A$ is a set of maps $(\gamma_r)_{r\ge 0}$ defined over the even degree part of $A$, and satisfying the following axioms:
\begin{itemize}
\item[$(a)$] $\gamma_0(x)=1$, $\gamma_1(x)=x$ and $\gamma_k$ maps $A_i$ into $A_{ki}$ for $i\ge 2$.
\item[$(b)$] $\gamma_k(x)\gamma_\ell(x)=\left(^{k+\ell}_{\;\;k}\right) \gamma_{k+\ell}(x)$.
\item[$(c)$] $\gamma_k(x+y)=\sum_{i+j=k}\gamma_i(x)\gamma_j(y)$.
\item[$(d)$] $\begin{array}[t]{cl}\gamma_k(xy)&=0\text{ if $k\ge 2$ and $x$ and $y$ have odd degrees, }\\
               & = x^k\gamma_k(y)\text{ if $k\ge 2$ and $|x|\ge 2$ is even  and $|y|$ is even. }
              \end{array}$
\item[$(e)$] $\gamma_\ell(\gamma_k(x))= \frac{(k\ell)!}{\ell!(k!)^\ell}\gamma_{k\ell}(x)$.
\end{itemize}
A morphism of algebras $f:A\to B$ preserves divided powers if $\gamma_k(f(x))=f(\gamma_k(x))$ for all $k$.
\end{definition}

Observe that the maps $\gamma_r$ are not $\k$-linear: actually, equation $(b)$ implies that $k!\gamma_k(x)=x^k$, which is the reason for the name `divided powers'. 
\begin{theoreme}[{\cite[Exp 8, Section 4]{Cartan}}]\label{thm-univ-Car}Let $V$ be a graded free $\k$-module. There exists a unique system of divided powers on the graded algebra $U(V)=\Gamma(V_{\mathrm{even}})\otimes \Lambda(V_{\mathrm{odd}})$ such that $\gamma_k(x)=x^{\otimes k}$ for $x\in V_{\mathrm{even}}$ and $k\ge 0$.

Moreover, for all commutative graded $\k$-algebra $A$ equipped with divided powers, and all graded $\k$-linear map $f:V\to A$, there exists a unique morphism of algebras $\overline{f}:U(V)\to A$ extending $f$ and preserving divided powers.
\end{theoreme}

Theorem \ref{thm-univ-Car} is a mean to construct morphisms of algebras with a very big image (compare with the image of the map induced by the universal property of the free graded commutative algebra on $V$). 
However, theorem \ref{thm-univ-Car} is not so efficient in characteristic $2$. 
For example, the graded $\k$-algebra $\Gamma(V)$ is graded commutative even if $V$ is concentrated in odd degree, and in that case, theorem \ref{thm-univ-Car} yields a morphism: $\Lambda(V_{\mathrm{odd}})\to \Gamma(V_{\mathrm{odd}})$ with quite a small image. To bypass this problem, there is a stronger notion of divided powers for strictly anticommutative (i.e.  graded commutative with $x^2=0$ if $x$ has odd degree) algebras in characteristic $2$. To avoid confusion, we call this notion `strong divided powers', although this notion is simply called `divided powers' in the literature.
\begin{definition}
Let $A$ be a strictly anticommutative graded $\k$-algebra, for $\k$ of characteristic $2$. A system of strong divided powers is a collection of maps $(\gamma_{r})_{r\ge 0}$ defined over the part of positive degree of $A$ and satisfying equations $(a)$-$(e)$ above.
\end{definition}
With this stronger notion of divided powers, theorem \ref{thm-univ-Car} becomes \cite[Exp 8, Thm 2 bis]{Cartan}:
\begin{theoreme}[{\cite{Cartan}}]Let $\k$ be a ring of characteristic $2$. Let $V$ be a graded free $\k$-module. There exists a unique system of strong divided powers on the graded algebra $\Gamma(V)$ such that $\gamma_k(x)=x^{\otimes k}$ for $x\in V$ and $k\ge 0$.

Moreover, for all strictly anticommutative graded $\k$-algebra $A$ equipped with strong divided powers, and all graded $\k$-linear map $f:V\to A$, there exists a unique morphism of algebras $\overline{f}:U(V)\to A$ extending $f$ and preserving divided powers.
\end{theoreme}

\subsection{Homology operations in bar constructions}

Let $\k$ be a commutative ring, and let $A$ be a cdg-$\k$-algebra. The homology of $\B^n A$ is equipped with the following homology operations, natural with respect to $A$.
\begin{enumerate}
\item {\bf Suspension.} 
For all $n\ge 0$ and all $k\ge 0$, there is a $\k$-linear suspension operation  \cite[Exp 6, sections 1 and 2]{Cartan}
$$\sigma:H_k(\B^nA)\to H_{k+1}(\B^{n+1} A)\;. $$
\item {\bf Transpotence.}
Assume $\k$ has prime characteristic $p\ne 2$. Then for all $n\ge 1$ and all $k\ge 1$ there is an additive (hence $\k$-linear if $\k=\Fp$) transpotence operation \cite[Exp 6, section 4]{Cartan}
$$\phi_p:H_{2k}(\B^nA)\to H_{p2k+2}(\B^{n+1} A)\;. $$
\item {\bf Divided powers.}
Assume that $A$ is strictly anticommutative\footnote{This applies to $\B A$, for $A$ graded commutative} (i.e. $a^2=0$ for $a$ of odd degree). Then for all $n\ge 1$, there is a system of divided powers $(\gamma_r)_{r\ge 0}$ on $H(\B^n A)$ \cite[Exp 7, Thm 1 and section 5]{Cartan}. If $\k$ has characteristic $2$, there is a system of strong divided powers on $H(\B^n A)$ \cite[Exp 7, section 8]{Cartan}.
\end{enumerate}
In our computation of $\B(\Gamma[2])$ in theorem \ref{thm-BnGamma-F}, we shall need the following complement on homological operations.
\begin{proposition}\label{prop-poids}
Let $\k$ be a commutative ring, and let $A$ be a cdg-$\k$-algebra with weights. The homology of $\B^n(A)$ is a wg-$\k$-algebra, and the homological operations have the following behavior with respect to weights. The suspension preserves the weights. The transpotence $\phi_p$ multiplies the weights by $p$, and the divided power operation $\gamma_r$ multiplies the weights by $r$.
\end{proposition}
\begin{proof}To prove proposition \ref{prop-poids}, we have to go back to the definition of the homology operations. Let us denote by $C$ the wdg-$\k$-algebra $C=\B^n A$. Let us first treat the case of the suspension. It is defined on the chain level by the map $C\hookrightarrow \B C$ $c\mapsto [c]$. So it preserves the weights.

The cases of the transpotence and divided powers are slightly more involved. Recall from \cite[Exposes 3 et  4]{Cartan} that there is a wdg-$\k$-algebra $B C$ characterized by the following properties.
\begin{itemize}
\item[(i)] $B C:=C\otimes \B C$ as a graded algebra with weights. Thus, both $C$ and $\B C$ can be viewed as weighted graded subalgebras of $B C$.
\item[(ii)] $B C$ is equipped with a weight-preserving differential $\partial$, such that following maps are morphisms of wdg-$\k$-algebras:
$$C\xrightarrow[]{\Id\otimes 1} B C\;,\quad B C \xrightarrow[]{\epsilon\otimes \Id}\B C\;.$$
\item[(B)] For all $k\ge 1$, the composite $\B C_{k+1}\hookrightarrow B C_{k+1}\xrightarrow[]{\partial}B C_{k}$ is injective and induces an isomorphism $\B C_{k+1}\simeq Z(B C)_{k}$ onto the cycles of degree $k$ of $BC$ (in particular, $H_i(B C)=0$ for $i>0$). 
\end{itemize}
The transpotence $\phi_p$ is defined on the chain level as follows \cite[p. 6-05]{Cartan}. Let $c\in ZC_{2k}\subset Z (BC)_{2k}$ be a cycle representing a cohomology class $\alpha\in H_{2k}(C)$. There is an $x\in \B C_{2k}$ such that $\partial x=c$. Then there is an element $y\in Z(\B C)_{p2k+2}$ representing $\phi_p(\alpha)$ which satisfies $\partial y=c^{p-1}*\partial x$ ($*$ denotes the product in $BC$). So the weight $w(\phi_p(\alpha))$ of $\phi_p(\alpha)$ satisfies:
$$w(\phi_p(\alpha))= w(y)= w(c^{p-1})+w(x)=pw(c)=pw(\alpha)\;.$$ 

The divided powers are defined on the chain level $\overline{B}C$ by \cite[Exp 7, proof of Thm 1]{Cartan}: $\gamma_1(c)=c$ and $\partial\gamma_r(c)=\partial c*\gamma_{r-1}(c)$ for $r\ge 1$. Since $\partial$ preserves the weights $w(\gamma_r(c))=w(c)+w(\gamma_{r-1}(c))$ for $r\ge 1$ and the result follows by induction on $r$. 
\end{proof}

\subsection{Cartan's result}
We are now ready to state Cartan's computation of $H^{\mathrm{sing}}_*(K(\pi,n),\Fp)$. In view of corollary \ref{cor-fonc-alg}, we state the results in terms of the homology of $\B^n(\Gamma[2])$. Let us assume first that $p\ne 2$.

\begin{definition}\label{def-p-admiss} A \emph{$p$-admissible word}\footnote{Such words are called `mots admissibles de premi\`ere esp\`ece' in \cite[Exp 9, section 1]{Cartan}. We drop the words `premi\`ere esp\`ece' in the definition since we shall not need the `mots admissibles de deuxi\`eme esp\`ece' in the article.} is a finite sequence $\W$ of letters $\sigma$, $\phi_p$ and $\gamma_p$ starting on the left with the letter $\sigma$ or $\phi_p$, and finishing on the right with the letter $\sigma$, satisfying the following property. For all letter $\phi_p$ or $\gamma_p$ of $\W$, the number of $\sigma$ on the right of the letter is even.
The \emph{height} of a word $\W$ is the number of letters equal to $\sigma$ or $\phi_p$ in $\W$. The \emph{degree} $\deg(\W)$ of a word $\W$ is defined recursively by $\deg(\alpha)=0$ if $\alpha$ is empty, and
$$\deg(\sigma \alpha)=1+\deg(\alpha)\;,\quad \deg(\gamma_p \alpha)=p\deg(\alpha)\;,\quad \deg(\phi_p\alpha)=p\deg(\alpha)+2\;.$$
\end{definition}

\begin{example}
The $p$-admissible words of height $3$ are the words: 
$$\sigma\gamma_p^k\sigma\sigma\;, k\ge 0\;,\quad\text{ and }\quad \phi_p\gamma_p^k\sigma\sigma\;,k\ge 0\;.$$
(By convention $\gamma_p^k$ is empty if $k=0$). Moreover, $\deg(\sigma\gamma_p^k\sigma\sigma)=2p^k+1$ and $\deg(\phi_p\gamma_p^k\sigma\sigma)=2p^{k+1}+2$. The $p$-admissible words of height $4$ are the words:
$$\sigma\sigma\gamma_p^k\sigma\sigma\;,\; \sigma\gamma_p^k\phi_p\gamma_p^\ell\sigma\sigma\;,\text{ and } \phi_p\gamma_p^k\phi_p\gamma_p^\ell\sigma\sigma\;, \text{ for $k,\ell\ge 0$,} $$
of respective degrees $2p^k+2$, $2p^{k+\ell+1}+2p^k+1$ and $2p^{k+\ell+2}+2p^{k+1}+2$.
\end{example}

Recall that if $M$ is a graded $\Fp$-vector space, $U(M)$ denotes the universal divided power algebra over $M$, i.e. $U(M)=\Gamma(M_{\mathrm{even}})\otimes \Lambda(M_{\mathrm{odd}})$.

\begin{theoreme}[{\cite[Exp 9, th\'eor\`eme fondamental]{Cartan}}]\label{thm-Car-comput1} Let $p$ be an odd prime. For all $n\ge 0$,
there is an isomorphism of graded $\F_\Fp$-algebras 
$$H(\B^{n}(\Gamma[2]))\simeq U\left(\bigoplus_{\W} I [\deg(\W)]\right)\;,$$
where the direct sum on the right hand side is taken over all $p$-admissible words $\W$ of height $n+2$ and $I [\deg(\W)]$ means a copy of the identity functor $I:\V_\Fp\to \V_\Fp$ placed in homological degree $\deg(\W)$. 
\end{theoreme}

\begin{remark}\label{rk-comput-Car}
Let $\Gamma_\k$ denote the divided power algebra over a field $\k$. For all $\Fp$-module $V$ there is an isomorphism of graded $\k$-algebras $\Gamma_\Fp(V[2])\otimes_\Fp\k\simeq \Gamma_\k(V\otimes_\Fp\k\,[2])$, hence an isomorphism of graded $\k$-algebras:
$$H(\B^n(\Gamma_\Fp(V[2])))\otimes_\Fp\k\simeq H(\B^n(\Gamma_\Fp(V\otimes_\Fp\k\,[2])))\;.$$
Observe that the isomorphism above is natural with respect to the $\Fp$-vector space $V$ but does not tell us anything about the functoriality with respect to the $\k$-vector space $V\otimes_\Fp\k$. So we cannot deduce the description of the homology of iterated bar constructions over a general field $\k$, as a graded $\F_\k$-algebra, from theorem \ref{thm-Car-comput1}.
\end{remark}

Actually, Cartan's result is more precise. It says that each copy of the identity functor on the right hand side is simply obtained by applying a suitable sequence of homology operations to $I[2]=(\Gamma[2])_2=H_2(\B^2 S)$. For example, if $n=2$, the copy of $I$ corresponding to the word $\sigma\phi_p\gamma_p^3\sigma\sigma$ is the image of $I[2]$ by the sequence of operations:
\begin{align*}(\Gamma[2])_2\xrightarrow[]{\gamma_p^3}(\Gamma[2])_{2p^3}\xrightarrow[]{\phi_p}H_{2p^4+2}(\B(\Gamma[2]))\xrightarrow[]{\sigma}H_{2p^4+3}(\B^2(\Gamma[2])).\\ 
\end{align*}
In general, the operations to be applied are the one needed to complete the final letters `$\sigma\sigma$' in order to obtain the word indexing the copy of the identity functor $I$ considered, starting from the right to the left. This produces a morphism of graded functors:
$$\textstyle\bigoplus_{\W} I [\deg(\W)]\to H(\B^{n}(\Gamma[2]))$$
and the isomorphism of theorem \ref{thm-Car-comput1} is produced from this morphism and the universal property of $U$.

Let us now describe the case $p=2$. This case is actually simpler. The definition of admissible words is modified as follows.
\begin{definition} A \emph{$2$-admissible word} is a finite sequence $\W$ of letters $\sigma$, and $\gamma_2$ starting with the letter $\sigma$ and finishing with the letters $\sigma\sigma$.
The \emph{height} of a word $\W$ is the number of $\sigma$ in $\W$. Its \emph{degree} $\deg(\W)$ is defined recursively by $\deg(\alpha)=0$ if $\alpha$ is empty, and
$\deg(\sigma \alpha)=1+\deg(\alpha)$ and $\deg(\gamma_2\alpha)=2\deg(\alpha)$.
\end{definition}

\begin{example}
The $2$-admissible words of height $3$ are the words $\sigma\gamma_2^k\sigma\sigma$ with $k\ge 0$, and 
$\deg(\sigma\gamma_2^k\sigma\sigma)=2^{k+1}+1$.
The $2$-admissible words of height $4$ are the words
$\sigma\gamma_2^k\sigma\gamma_2^\ell\sigma\sigma$ for $k,\ell\ge 0$ of
degree $2^{k+\ell+1}+2^{k}+1$.
\end{example}

The analogue of theorem \ref{thm-Car-comput1} over $\Fdeux$ takes the following form \cite[Exp 10, th\'eor\`eme fondamental]{Cartan}.

\begin{theoreme}[{\cite{Cartan}}]\label{thm-Car-comput2} For all $n\ge 0$,
there is an isomorphism of graded $\F_\Fdeux$-algebras 
$$H(\B^{n}(\Gamma[2]))\simeq \Gamma\left(\bigoplus_{\W} I [\deg(\W)]\right)\;,$$
where the direct sum on the right hand side is taken over all $2$-admissible words $\W$ of height $n+2$ and $I [\deg(\W)]$ means a copy of the identity functor $I:\V_\Fdeux\to \V_\Fdeux$ placed in homological degree $\deg(\W)$. 
\end{theoreme}

\subsection{Homology of $\overline{B}^n(\Gamma[2])$ over an arbitrary field}
We are now ready to prove the main result of section \ref{sec-bar-Gamma-F}, namely the computation of the homology of $\overline{B}^n(\Gamma[2])$, as a strict polynomial algebra, and over an arbitrary field $\k$. We first need to introduce a definition.
\begin{definition}\label{def-twisting-word}
Let $\W$ be a $p$-admissible word ($p$ even or odd). The \emph{twisting} of a $p$-admissible word $\W$ is the number of letters equal to $\phi_p$ or $\gamma_p$ in $\W$. We denote the twisting of $\W$ by  $t_\W$.
\end{definition}

\begin{theoreme}\label{thm-BnGamma-F}
Let $\k$ be a field of positive characteristic $p$. If $p$ is odd, we have an isomorphism of graded $\P_\k$-algebras:
$$H(\B^n(\Gamma[2]))\simeq U\left(\bigoplus_\W I^{(t_\W)}[\deg(\W)]\right)\;, $$  
where the sum is taken over all $p$-admissible words $\W$ of height $n+2$, and $I^{(t_\W)}[\deg(\W)]$ denotes a copy of the $t_\W$-th Frobenius twist functor, placed in homological degree $\deg(\W)$. If $p=2$ we have an isomorphism of graded $\P_\k$-algebras:
$$H(\B^n(\Gamma[2]))\simeq \Gamma\left(\bigoplus_\W I^{(t_\W)}[\deg(\W)]\right)\;, $$
where the sum is taken over all $2$-admissible words $\W$ of height $n+2$.
\end{theoreme}
\begin{proof} Let us prove the case $p$ odd (the case $p=2$ is similar). We start with the case $\k=\Fp$. By theorem \ref{thm-Car-comput1}, the graded $\F_\Fp$-algebra
$\U H(\B^n(\Gamma[2]))$ is isomorphic to $U\left(\bigoplus_\W I[\deg(\W)]\right)$. Hence by theorem \ref{thm-strictpolalg1}, the graded $\P_\Fp$-algebra $H(\B^n(\Gamma[2]))$ is of the form $U\left(\bigoplus_\W I^{(r_\W)}[\deg(\W)]\right)$, where the $r_\W$ are nonnegative integers which we have to determine. Actually, these integers are given by the weights of the $H(\B^n(\Gamma[2]))$. Let $V\in\V_\Fp$. By proposition \ref{prop-poids} the morphism induced by the cohomology operations 
$$U(\textstyle\bigoplus_\W V[\deg(\W)])\to H(\B^n \Gamma (V[2]))$$
becomes a morphism of weighted graded $\Fp$-algebras if we let the copy of $V$ indexed by $\W$ have weight $t_\W$. Hence, the weighted graded algebra structure of $H(\B^n(\Gamma[2]))$ implies that $r_\W=t_\W$.

To get the result for all fields $\k$ it suffices to use the exact base change for strict polynomial functors (Notice: this feature is specific to \emph{strict polynomial} functors!). If $\k$ is a field of characteristic $p$, there is an exact functor \cite[Prop 2.6]{SFB}: $\otimes_{\Fp}\k:\P_{\Fp}\to \P_{\k}$. This base change functor commutes with tensor products and sends divided powers to divided powers and Frobenius twists to Frobenius twists. Whence the result.
\end{proof}

\begin{remark}\label{rk-poids-nat}
At first sight, the result for $p=2$ seems very different from the result for $p$ odd. However, if $p=2$ and $V\in\V_{\Fdeux}$, there is an isomorphism of graded algebras with weights (not natural with respect to $V$): $$\Gamma(V^{(r)}[i])\simeq \Lambda(V^{(r)}[i])\otimes \Gamma(V^{(r+1)}[2i])$$  
Indeed, we can prove it if $V=\Fdeux$ by direct inspection, and get the general result by the exponential formula. 
\end{remark}

\subsection{Ext-computations over a field}

Combining theorem \ref{thm-BnGamma-F} and theorem \ref{thm-bar}, we obtain the following computations of the $\P_\k$-graded algebras $\E(S,\Gamma)$ and $\E(S,\Lambda)$. Comparison with earlier computations of these algebras by other authors is made in section \ref{subsec-compare}.

\begin{theoreme}\label{thm-calculsanstwistI}
Let $\k$ be a field of odd characteristic $p$, and let $V$ be a finite dimensional $\k$-vector space. Let us denote by $I^{(k)}\langle i\rangle$ a copy of the $k$-th Frobenius twist functor, placed in cohomological degree $i$ (thus, $I^{(k)}\langle i\rangle=I^{(k)}[-i]$). We have an isomorphism of $\P_\k$-graded algebras:
\begin{align*}
&\E(S,\Lambda)\simeq \;\Lambda\left(\bigoplus_{k\ge 0}I^{(k)}\langle p^k-1\rangle\right)\otimes^1\,^\t \Gamma\left(\bigoplus_{k\ge 0} I^{(k+1)}\langle p^{k+1}-2\rangle\right)\,, \\
\end{align*}
where $A\otimes^1 B$ denotes the `signed tensor product' of two $\P_\k$-graded algebras and $^\t A$ denotes the weight twisted algebra associated to $A$, as defined in section \ref{subsec-1E-com}. We also have an isomorphism of $\P_\k$-graded algebras:
\begin{align*}
& \E(S,\Gamma)\simeq \\&\qquad\Gamma\left(\bigoplus_{k\ge 0} V^{(k)}\langle 2p^k-2\rangle\right)\otimes
\Lambda\left(\bigoplus_{k\ge 0,\ell\ge 0}V^{(k+\ell+1)}\langle 2p^{k+\ell+1}-2p^{k}-1\rangle\right)\\& \qquad\qquad\qquad\qquad\qquad\otimes \Gamma\left(\bigoplus_{k\ge 0,\ell\ge 0} V^{(k+\ell+2)}\langle 2p^{k+\ell+2}-2p^{k+1}-2\rangle\right)\;.
\end{align*} 
\end{theoreme}
\begin{proof}
Let us prove  the first isomorphism, the second one is similar. Theorem \ref{thm-bar} yields an isomorphism between $\E(S,\Lambda)$ and 
$^\t\R_3(H\overline{B}(\Gamma[2]))$. The homology of $\B(\Gamma[2])$ is computed in theorem \ref{thm-BnGamma-F}, and applying the computation rules for $\R_\alpha$ developed in section \ref{sec-regrading}, we obtain 
\begin{align*}^t\E(S,\Lambda)=&\R_3\Lambda(\textstyle\bigoplus_{k\ge 0}I^{(k)}[2p^k+1])\otimes^1 \R_3\Gamma(\textstyle\bigoplus_{k\ge 0}I^{(k+1)}[2p^{k+1}+2])\;,\\
=& ^t\Lambda(\textstyle\bigoplus_{k\ge 0}I^{(k)}[-p^k+1])\otimes^1 \Gamma(\textstyle\bigoplus_{k\ge 0}I^{(k+1)}[-p^{k+1}+2])\;.
\end{align*}
Using $I^{(k)}[-i]=I^{(k)}\langle i\rangle$, we obtain the result.
\end{proof}

\begin{theoreme}\label{thm-calculsanstwistII}
Let $\k$ be a field of characteristic $p=2$, and let $V$ be a finite dimensional $\k$-vector space. Let us denote by $I^{(k)}\langle i\rangle$ a copy of the $k$-th Frobenius twist functor, placed in cohomological degree $i$. We have isomorphisms of $\P_\k$-graded algebras:
\begin{align*}
&\E(S,\Lambda)\simeq \Gamma\left(\bigoplus_{k\ge 0}I^{(k)}\langle p^k-1\rangle\right)\,, \\
& \E(S,\Gamma)\simeq \Gamma\left(\bigoplus_{k\ge 0,\ell\ge 0} I^{(k+\ell)}\langle 2p^{k+\ell}-p^{k}-1\rangle\right)\;.
\end{align*}
\end{theoreme}

\section{Explicit computations over the integers}\label{sec-bar-Gamma-Z}

In this section, we work over the ground ring $\k=\Z$. We elaborate on Cartan's computation of the homology of EML-spaces with integral coefficients to compute the homology of $\B^n(\Gamma[2])$. Actually, Cartan made two descriptions of the integral homology of EML spaces $K(\pi,n)$: a compact description, which is not natural with respect to $\pi$ (this is \cite[Exp 11, {th\'eor\`eme} 1]{Cartan}), and a description by generators and relations, which is natural with respect to $\pi$, but unfortunately quite complicated (this is \cite[Exp 11, {th\'eor\`eme} 6]{Cartan}). For the sake of simplicity, we have chosen to use the compact description of the homology of EML spaces, therefore we only compute the homology of $\B^n\Gamma (\Z^m[2])$ as graded algebras with weights. As a corollary, we compute $\E(S,\Lambda)(\Z^m)$ and $\E(S,\Gamma)(\Z^m)$ as wdg-$\Z$-algebras.
\subsection{Dual Koszul and De Rham algebras}
If $V$ is a graded $\Z$-module, we  denote by $V[j]$ its homological suspension, that is, $V[j]_i:=V_{\,i-j}$. 

\begin{definition}[Dual Koszul algebra]
Let $V$ be a positively graded $\Z$-free module concentrated in odd degrees. Then $\Gamma(V[1])\otimes \Lambda(V)$ is a commutative graded-$\Z$-algebra. If $h$ is a positive integer, we define a differential $d_K$ as the composite:
\begin{align*}\Gamma^n(V[1])\otimes \Lambda^k(V)&\xrightarrow[]{\Delta\otimes \Id} \Gamma^{n-1}(V[1])\otimes V[1]\otimes \Lambda^k(V)\\\xrightarrow[]{h\Id} &\Gamma^{n-1}(V[1])\otimes V\otimes \Lambda^k(V)\xrightarrow[]{\Id\otimes m} \Gamma^{n-1}(V[1])\otimes \Lambda^{k+1}(V)\;. \end{align*} 
This makes $(\Gamma(V[1])\otimes \Lambda(V),d_K)$ into a commutative differential graded algebra which we call the `dual Koszul algebra'. We denote it by $K^h(V)$.
\end{definition}

If $h=1$, $K^h(V)$ is the graded dual of the usual Koszul algebra, see e.g. \cite[Section 4]{FFSS}, whence the name. If $V$ is a copy of $\Z$ in odd degree, then $K^h(V)$ is nothing but an elementary complex of type (II) from \cite[Exp 11]{Cartan}. Dual Koszul algebras satisfy an exponential property: the following composite (where the first map is induced by the canonical inclusions into $V\oplus W$) is an isomorphism of differential graded algebras
$$K^h(V)\otimes K^h(W)\to
K^h(V\oplus W)^{\otimes 2}\xrightarrow[]{m}K^h(V\oplus W)\;.$$ 
Finally, we observe that when $V$ is equipped with weights, $K^h(V)$ is canonically made into a wdg-$\Z$-algebra. 

Assume now that $V$ is concentrated in even degrees. We can adapt the definition of the dual Koszul algebra by exchanging the roles of exterior and divided powers. This yields the dual De Rham algebra $\Omega^h(V)$. 
\begin{definition}[Dual De Rham algebra]
Let $V$ be a positively graded $\Z$- free module, and let $h$ be an integer. The dual De Rham algebra $\Omega^h(V)$ is the cdg-$\Z$-algebra which equals $\Gamma(V)\otimes \Lambda(V[1])$ as a commutative graded algebra and whose differential equals the composite
\begin{align*}\Gamma^n(V)\otimes \Lambda^k(V[1])&\xrightarrow[]{\Id\otimes\Delta} \Gamma^{n}(V)\otimes V[1]\otimes \Lambda^{k-1}(V[1])\\\xrightarrow[]{h\Id} &\Gamma^{n}(V)\otimes V\otimes \Lambda^{k-1}(V[1])\xrightarrow[]{m\otimes\Id} \Gamma^{n+1}(V)\otimes \Lambda^{k-1}(V[1])\;. \end{align*} 
\end{definition}

If $h=1$, $\Omega^h(V)$ is the graded dual of the usual De Rham algebra, and if $V$ is a copy of $\Z$ in even degree, it is an elementary complex of type (II) from \cite[Exp 11]{Cartan}. There are isomorphisms of dg-$\Z$-algebras $\Omega^h(V)\otimes \Omega^h(W)\simeq \Omega^h(V\oplus W)$, and when $V$ is equipped with weights, $\Omega^h(V)$ becomes a wdg-$\Z$-algebra.

\subsection{Cartan's result} 

We fix a free $\Z$-module $V=\Z^m$. and a positive integer $n$. We are going to present Cartan's computation of the homology of EML spaces from \cite[Exp. 11, Theoreme 1]{Cartan}. 
Recall from definition \ref{def-p-admiss} the $p$-admissible words of height $n$ attached to a prime $p$ (here we also use definition \ref{def-p-admiss} when $p=2$). The word $\sigma^n$ is $p$-admissible of height $n$, and the other $p$-admissible words of height $n$ can be grouped in a unique way into pairs of the form $(\sigma^{k+1}\gamma_p\alpha,\sigma^k\phi_p\alpha)$, where $\alpha$ is denotes a word and $k\ge 0$. In such a pair, the degrees of the two words differ by one: $\deg(\sigma^k\phi_p\alpha)= \deg(\sigma^{k+1}\gamma_p\alpha)+1$.
\begin{definition}
We call \emph{$p$-pair of height $n$} a pair of $p$-admissible words of height $n$ of the form $\Pa=(\sigma^{k+1}\gamma_p\alpha,\sigma^k\phi_p\alpha)$ (where $\alpha$ is a word, and $k\ge 0$). 
The \emph{degree} of the pair $\Pa$ is the lowest degree of the words of the pair, that is $\deg(\Pa)=\deg(\sigma^{k+1}\gamma_p\alpha)$.
\end{definition}

\begin{example}\label{ex-p-pairs}
If $n=3$, the $p$-pairs of height $n$ are the pairs $$\Pa_k=(\sigma\gamma_p^{k+1}\sigma^2\,, \,\phi_p\gamma_p^k\sigma^2)\;,\quad\text{ for $k\ge 0$,}$$ 
and $\deg(\Pa_k)=2p^{k+1}+1$. If $n=4$ the $p$-pairs of height $n$ are the pairs
\begin{align*}
&\Pa_{\ell}=(\sigma^2\gamma_p^{\ell+1}\sigma^2, \sigma\phi_p\gamma_p^{\ell}\sigma^2)\;, && \text{ for $\ell\ge 0$,} \\
&\Pa_{k,\ell}=(\sigma\gamma_p^{k+1}\phi_p\gamma_p^\ell\sigma^2\,,\,\phi_p\gamma_p^k\phi_p\gamma_p^\ell\sigma^2)\;,&&\text{ for $k\ge 0,\ell\ge 0$,}
\end{align*}
and their degrees are $\deg(\Pa_\ell)=2p^{\ell+1}+2$, $\deg(\Pa_{k,\ell})= 2p^{\ell+k+2}+2p^{k+1}+1$.
\end{example}

\begin{definition}
We denote by $X_p^{[n]}$ the cdg-$\Z$-algebra defined by:
$$X_p^{[n]}=K^p\left(\bigoplus_{\Pa} V[\deg(\Pa)]\right)\otimes \Omega^p\left(\bigoplus_{\Pa'} V[\deg(\Pa')]\right)\;,$$
where the first direct sum is taken over all the $p$-pairs $\Pa$ of height $n$ and odd degree, and the second one is taken over all the $p$-pairs $\Pa'$ of height $n$ and even degree. We also denote by $X_0^{[n]}$ the cdg-$\k$-algebra which equals $\Lambda(V[n])$ if $n$ is odd and $\Gamma(V[n])$ if $n$ is even, with trivial differential.
\end{definition}

Let us denote by $_p M$ the $p$-primary part of a $\Z$-module $M$:
$$_p M=\{m\in M\;; \exists k\; p^km=0\}\;.$$
The homology of $X_p^{[n]}$ equals $\Z$ in degree zero, so its $p$-primary part $_pH(X_p^{[n]})$ is a graded subalgebra of $H(X_p^{[n]})$ without unit (it is concentrated in positive degrees).
We make it into a unital $\Z$-algebra $\widehat{_pH}(X_p^{[n]})$ in the canonical way:
$$\widehat{_pH_0}(X_p^{[n]}) =\Z\,,\qquad \widehat{_pH_i}(X_p^{[n]})=\,_pH_i(X_p^{[n]}) \text{ for $i>0$.}$$
In view of corollary \ref{cor-fonc-alg}, Cartan's computation of the integral homology of the EML-spaces $K(\Z^m,n+2)$ can be formulated in the following way. 
\begin{theoreme}[{\cite[Exp. 11, Theoreme 1]{Cartan}}]\label{thm-CartanZ}
Let $n$ be a positive integer, let $V=\Z^m$ be a free abelian group. The homology of $\B^n\Gamma(V[2])$ is isomorphic as a graded algebra, to the tensor product
$$X_0^{[n+2]}\otimes\bigotimes_{\text{$p$ prime}} \widehat{_pH}(X_p^{[n+2]})\;.$$
\end{theoreme}

\subsection{Integral torsion of $\Ext$s}
To be able to identify $\E^i(S^d,\Lambda^d)(V)$ or $\E^i(S^d,\Gamma^d)(V)$ explicitly as a direct summand of the homology of $\B^n\Gamma(V[2])$, we have to describe the later as a graded algebra with weights. However, some interesting information can already be extracted from theorem \ref{thm-CartanZ} without describing the weights, in particular for the case $n=1$. 

Indeed, observe in example \ref{ex-p-pairs} that all the $p$-pairs of height $3$ have odd degrees. So for all prime $p$, $X_p^{[3]}$ is just the dual Koszul algebra 
$$X_p^{[3]}=K^p\left(\bigoplus_{k\ge 0}V[2p^{k+1}+1]\right)\;. $$
\begin{lemme}\label{lm-p-tors} $H_i(X_p^{[3]})$ is a $\Fp$-vector space if $i>0$ and equals $\Z$ if $i=0$.
\end{lemme}
\begin{proof}
If $W$ is a graded free $\Z$-module in positive odd degrees, the dual Koszul algebra $K^1(W)$ satisfies $H_i(K^1(W))=\Z$ if $i=0$ and zero otherwise (use the exponential formula and the K\"unneth morphism to reduce to the case when $W$ has rank one, then the result is easy). Since the differential $\partial$ of $K^p(W)$ is $p$ times the differential $\delta$ of $K^1(W)$, we have $H_0(K^p(W))=\Z$, and if $i>0$:
$$H_i(K^p(W))= \mathrm{Ker}\partial/\mathrm{Im}\partial = \mathrm{Ker}\delta/p(\mathrm{Im}\delta) = \mathrm{Im}\delta/p(\mathrm{Im}\delta)\simeq\mathrm{Im}\delta\otimes_\Z\Fp\;.$$
This proves the result.
\end{proof}
Lemma \ref{lm-p-tors} implies that the $p$-primary part of the homology of $\B\Gamma(V[2])$, hence of $\E(S,\Lambda)(V)$, is a $\Fp$-vector space. This information is precious because it implies that $\E(S,\Lambda)(V)$ can be retrieved from the analogous extensions over $\Fp$ by the universal coefficient theorem. 
\begin{theoreme}\label{thm-PSZ} Let $V=\Z^m$.
Then $\E^0(S^d,\Lambda^d)(V)\simeq \Gamma^d(V)$, and for $i$ positive, the $p$-primary part of $\E^i(S^d,\Lambda^d)(V)$ is a finite dimensional $\Fp$-vector space. Moreover let $P(t)$, resp. $Q(t)$, be the Poincar\'e series of the graded $\Fp$-vector space $\E(S^d,\Lambda^d)(V)\otimes_\Z\Fp$, resp. $\E(S^d_\Fp,\Lambda^d_\Fp)(V/pV)$. Then we have $(1+t)P(t)=tQ(t)+P(0)$.
\end{theoreme}
\begin{proof}
The computation of $\E^0(S^d,\Lambda^d)(V)$ was made in lemma \ref{lm-element}. It is well known that $\E^{i}(S^d,\Lambda^d)(V)$ is a finite abelian group for $i$ positive. This is actually true for extensions between arbitrary strict polynomial functors $F,G$ with finitely generated values, but in our case it can be directly proved as follows. First, $\E^{i}(S^d,\Lambda^d)(V)$ is a finitely generated abelian group. Indeed, it is the $d-i$-th homology group of the complex $\H(S^d,I_{\bullet,d})(V)$ from section \ref{subsubsec-bar-sym-ext}, which is a complex of free finitely generated $\Z$-modules. Then by base change \cite[prop 2.6]{SFB} and the universal coefficient theorem, $\E^{i}(S^d,\Lambda^d)(V)\otimes_\Z\mathbb{Q}$ is isomorphic to  
$\E^{i}(S^d_\mathbb{Q},\Lambda^d_\mathbb{Q})(V\otimes_\Z\mathbb{Q})$ which is zero because the category $\P_\mathbb{Q}$ is semi-simple. Thus $\E^{i}(S^d,\Lambda^d)(V)$ is a finite abelian group. Now lemma \ref{lm-p-tors} and theorem \ref{thm-CartanZ} show that the $p$-primary part of $\B\Gamma(V[2])$, hence of $\E^i(S^d,\Lambda^d)(V)$, is a $\Fp$-vector space.
Finally, by base change \cite[prop 2.6]{SFB} and the universal coefficient theorem, there is for all $i$ an isomorphism of $\Fp$-vector spaces:
$$\E^i(S^d,\Lambda^d)(V)\otimes_\Z\Fp\oplus \E^{i+1}(S^d,\Lambda^d)(V)\otimes_\Z\Fp \simeq \E^i(S^d_\Fp,\Lambda^d_\Fp)(V/pV)\;.$$
The assertion on the Poincar\'e series follows.
\end{proof}
The Poincar\'e series $Q(t)$ are easy to determine from theorems \ref{thm-calculsanstwistI} and \ref{thm-calculsanstwistII}, and they were first determined in \cite{A} for $V=\Z$. Since we know $P(0)=\dim_\Fp(\Gamma^d(V)\otimes\Fp)$, the Poincar\'e series $P(t)$ are easily computed from the equation $(1+t)P(t)=tQ(t)+P(0)$ of theorem \ref{thm-PSZ}. 
To illustrate this, we give the values of $\Ext^i_{\P_\Z}(S^n,\Lambda^n)$ in low degrees (for $i=1$ they were first computed in \cite[section 4]{A}).
\begin{example} We have the following $\Ext$-computations:
\begin{align*}&\Ext^1_{\P_\Z}(S^n,\Lambda^n)=\Z/2\Z \text{ if $n\ge 2$, and zero if $n\le 1$,} \\
&\Ext^2_{\P_\Z}(S^n,\Lambda^n)=\Z/3\Z \text{ if $n=3,4$ and zero otherwise}\;,\\
&\Ext^3_{\P_\Z}(S^n,\Lambda^n)= \left\{{\begin{array}{cl}
                                         0 &\text{ if $n\le 3$,}\\
                                         \Z/2\Z\oplus\Z/3\Z &\text{ if $n=6,7$,}\\
                                         \Z/2\Z &\text{ otherwise.}  
                                        \end{array}}\right.
\end{align*}
\end{example}

In the case $n=2$, $X_p^{[4]}$ is a tensor product of a dual Koszul algebra and a dual De Rham algebra. The latter brings $p^r$-torsion for all values of $r\ge 0$ in the homology of $X_p^{[4]}$. So in contrast to theorem \ref{thm-PSZ} we have the following result.
\begin{proposition}Let $V=\Z^m$. For all prime $p$ and all $r\ge 0$, there exists positive integers $i,d$ such that the abelian group $\E^i(S^d,\Gamma^d)(V)$ has an element of $p^r$-torsion.
\end{proposition}

\subsection{Computation of weights} The $\P_\Z$-dg-algebra $\B^n(\Gamma[2])$ becomes a wdg-$\Z$-algebra after evaluation on $V=\Z^m$, so the homology of $\B^n\Gamma(V[2])$ is equipped with weights. We are now going to supplement Cartan's result by describing these weights. 
\begin{definition}
Let $\Pa$ be a $p$-pair. The \emph{weight} of $\Pa$ is the integer $w(\Pa)$ defined by $w(\Pa)=p^{t_\Pa}$, where $t_\Pa$ is the number of letters which equal $\phi_p$ or $\gamma_p$ in one of the two words of $\Pa$ (compare definition \ref{def-twisting-word}).
\end{definition}
\begin{example} We keep the notations of example \ref{ex-p-pairs}.  
For the $p$-pairs of height $n=3$, we have $w(\Pa_k)=p^{k+1}$. For the $p$-pairs of height $4$ we have:
$w(\Pa_\ell)=p^{\ell+1}$ and $w(\Pa_{k,\ell})=p^{k+\ell+2}$.
\end{example}

\begin{theoreme}\label{thm-BGamma-Z}
Let $n$ be a nonnegative integer, let $V=\Z^m$ be a free abelian group. There is an isomorphism of graded algebras with weights
$$X_0^{[n+2]}\otimes\bigotimes_{\text{$p$ prime}} \widehat{_pH}(X_p^{[n+2]})\simeq H\left(\B^n(\Gamma( V[2]))\right)\;.$$
On the left hand side, $X_p^{[n+2]}$ denotes the wdg-$\Z$-algebra defined by:
$$X_p^{[n+2]}=K^p\left(\bigoplus_{\Pa} V_{w(\Pa)}[\deg(\Pa)]\right)\otimes \Omega^p\left(\bigoplus_{\Pa'} V_{w(\Pa')}[\deg(\Pa')]\right)\;,$$
where $V_k[i]$ denotes a copy of $V$ having weight $k$ and degree $i$, and where the first direct sum is taken over a the $p$-pairs $\Pa$ of height $n+2$ and odd degree and the second direct sum is taken over the $p$-pairs $\Pa'$ of height $n+2$ and even degree. Moreover, $X_0^{[n+2]}$ equals $\Lambda(V_1[n+2])$ if $n$ is odd, and $\Gamma(V_1[n+2])$ if $n$ is even. 
\end{theoreme}

Before we prove theorem \ref{thm-BGamma-Z}, we recall the universal property of dual Koszul and De Rham algebras. If $W$ is a graded $\Z$-module we denote by $W_{\mathrm{odd}}$, resp. $W_{\mathrm{even}}$, its summand of odd, resp. even, degree.
\begin{lemme}[see {\cite[Exp. 1, section 2]{Cartan}}]\label{lm-univ}
Let $W$ be a positively graded $\Z$-free module, let $h$ be a integer, and let $C^h(W)$ be the complex $W[1]\xrightarrow[]{h\Id} W$. This complex is a direct summand of $K^h(W_{\mathrm{odd}})\otimes \Omega^h(W_{\mathrm{even}})$. Assume that $A$ is a cdg-$\Z$-algebra, equipped with a system of divided powers, and free as a $\Z$-module. For all morphism of complexes $f:C^h(W)\to A$ there exists a unique morphism of cdg-$\Z$-algebras $\overline{f}$ such that
the following diagram commutes:
$$\xymatrix{
C^h(W)\ar[rr]^-{f}\ar@{^{(}->}[d]&& A\\
K^h(W_{\mathrm{odd}})\otimes \Omega^h(W_{\mathrm{even}})\ar@{-->}[rru]_-{\overline{f}}
}$$
If $W$ is equipped with weights and $A$ is a cwdg-$\Z$-algebra, then $\overline{f}$ preserves the weights if and only if $f$ does.
\end{lemme}
\begin{proof}
The morphism of algebra $\overline{f}$ is obtained by the universal property of the graded algebra with divided powers $K^h(W_{\mathrm{odd}})\otimes \Omega^h(W_{\mathrm{even}})=U(W[1]\oplus W)$ (see theorem \ref{thm-univ-Car}). In particular $\overline{f}$ preserves the weights if and only if $f$ does. So one has only to check that $\overline{f}$ commutes with the differentials. Using that $d(\gamma^k(x))=dx\cdot \gamma^{k-1}x$ in $K^h(W_{\mathrm{odd}})\otimes \Omega^h(W_{\mathrm{even}})=U(W[1]\oplus W)$ and $A$, one reduces the proof that $\overline{f}$ commutes with differentials to the proof that the restriction of $\overline{f}$ to the generators of $K^h(W_{\mathrm{odd}})\otimes \Omega^h(W_{\mathrm{even}})=U(W[1]\oplus W)$ commutes with the differential. But this restriction is nothing but $f$. Whence the result.
\end{proof}

\begin{proof}[Proof of theorem \ref{thm-BGamma-Z}] Denote by $H(V,n+2,\k)$ the singular homology of the EML space $K(V,n+2)$ with coefficients in $\k$. Let $A$ be a differential graded $\Z$-algebra such that
\begin{enumerate}
\item[(i)] $A$ is graded commutative, $\Z$-free, equipped with divided powers\;,
\item[(ii)] The homology of $A$ is isomorphic to $H(V,n+2,\Z)$\;,
\item[(iii)] For all prime $p$, the homology of the $\Fp$-differential graded algebra $A\otimes_\Z \Fp$ is isomorphic to $H(V,n+2,\Fp)$\;.
\end{enumerate}
To prove theorem \ref{thm-CartanZ}, Cartan builds morphisms of dg-algebras 
$$f_p:X_p^{[n+2]}\to A\,,$$
(for $p$ zero or prime) which induce after taking homology the isomorphism 
$$\textstyle X_0^{[n+2]}\otimes\bigotimes_{\text{$p$ prime}} \widehat{_pH}(X_p^{[n+2]})\simeq H(V,n+2,\Z)\;.$$
In Cartan's proof, the algebra $A$ is $\B^{n+2}(\Z V)$, but any other algebra $A$ satisfying conditions $(i)$, $(ii)$ and $(iii)$ works as well in his argument, for example we can take $A=\B^n(\Gamma(V[2]))$. 

So, to prove theorem \ref{thm-BGamma-Z}, it suffices to prove that the morphisms $f_p$ preserve weights when $A=\B^n \Gamma(V[2])$ and when the weights on $X_p^{[n+2]}$ are as indicated in the statement of theorem \ref{thm-BGamma-Z}. 

{\bf Case of $f_p$, $p$ prime.} By the universal property of Dual Koszul and De Rham algebras from lemma \ref{lm-univ}, the construction of $f_p$ (preserving weights) reduces to the construction for each $p$-pair $\Pa$ of a morphism of complexes (preserving weights)
$$f_p:C^p(V_{w(\Pa)}[\deg(\Pa)])\to A$$
Fix a pair $\Pa=(\sigma^{k+1}\gamma_p\alpha,\sigma^k\phi_p\alpha)$ and a basis $(v_i)$ of $V[\deg(\Pa)]$. Denote by $(w_i)$ the same basis, considered as a basis of $V[\deg(\Pa)+1]$. So the complex $C^p(V_{w(\Pa)}[\deg(\Pa)])$ equals
$$V[\deg(\Pa)+1]\xrightarrow[]{d}V[\deg(\Pa)]\;,$$
with $d(w_i)=v_i$. Let $\overline{v_i}$ be the reduction modulo $p$ of $v_i$, considered as an element of $V\otimes_\Z\Fp[\deg(\sigma^{k}\phi_p\alpha)]$ which is the direct summand $H(V,n+2,\Fp)$ indexed by the $p$-admissible word $\sigma^{k}\phi_p\alpha$, cf. theorem \ref{thm-BnGamma-F}. This summand is also well defined for $p=2$ by remark \ref{rk-poids-nat}. The morphism of complexes $f_p$ is defined by Cartan as follows.
\begin{enumerate}
\item[(a)] 
The canonical map $A\to A\otimes_\Z\Fp$ is surjective so we can find an element 
$f_p(w_i)\in A$ whose reduction modulo $p$ is a cycle representing the homology class 
$\overline{v_i}\in H(V,n+2,\Fp)$. 
\item[(b)] The reduction mod $p$ of $f_p(w_i)$ is a cycle of $A\otimes_\Z\Fp$, so $d(f_p(w_i))$ is divisible by $p$. We define $f_p(v_i)$ by the equality $p f_p(v_i)= d(f_p(w_i))$.
\end{enumerate}
Take $A=\B^n \Gamma(V[2])$. The canonical map $A\to A\otimes_\Z\Fp$ preserves weights, so the weight of $f_p(w_i)$ is the same as the weight of $\overline{v}_i$, which equals $w(\Pa)$ by theorem \ref{thm-BnGamma-F}. Thus, the map $f_p$ preserves weights.

{\bf Case of $f_0$.} The morphism $f_0:X_0^{[n+2]}\to \B^n(\Gamma(V[2]))$ is induced by the morphism $V_1[n+2]\xrightarrow[]{=}\B^n(\Gamma(V[2]))_{n+2}$ and the universal property of $X_0^{[n+2]}=U(V_1[n+2])$. So $f_0$ preserves weights (it is actually a morphism of $\P_\Z$-dg-algebras).
\end{proof}

Theorem \ref{thm-BGamma-Z} yields an algorithm to compute the homology of $\B^n\Gamma(V[2])$. Indeed, the homology of the dual Koszul and De Rham algebras on a single generator can be computed by direct inspection.
\begin{enumerate}
\item[(i)] The homology of $K^h(\Z[2i-1])$ is $\Z$ in degree $0$ (and weight zero), $\Z/h\Z$ in degrees $d2i-1$ (and weight $d$) for $d\ge 0$, and  zero elsewhere.
\item[(ii)] The homology of $\Omega^h(\Z[2i])$ is $\Z$ in degree $0$ (and weight zero), $\Z/dh\Z$ in degrees $d2i$ (and weight $d$) for $d\ge 0$, and  zero elsewhere.
\end{enumerate}
With the help of the exponential isomorphisms $$K^h(V\oplus W)\simeq K^h(V)\otimes K^h(W)\;,\quad\text{ and }\quad\Omega^h(V\oplus W)\simeq \Omega^h(V)\otimes\Omega^h(W)\;,$$ and iterative uses of the K\"unneth formula, one recovers the homology of the wdg-$\Z$-algebras $X_p^{[n+2]}$. Theses algebras have an infinite number of generators, but only a (relatively small) finite number of generators play a role in the computation of a summand of the homology with given weight or degree. Let us give the computation of the homogeneous part of weight $4$ of the homology of $\B^n \Gamma(\Z[2])$ for $n=1$ and $n=2$. 
\begin{example}\label{ex-comput}
The homogeneous part of weight $4$ of the homology of $\B \Gamma(\Z[2])$ is given by the following table (it is zero outside the table).
 $$\begin{array}{c|ccc}
\text{degree} & 9 & 10 & 11 \\
\hline
\text{homology} & \Z/2\Z & \Z/3\Z & \Z/2\Z   
\end{array}$$
The homogeneous part of weight $4$ of the homology of $\B^2 \Gamma(\Z[2])$ is given by the following table (it is zero outside the table).
 $$\begin{array}{c|ccccccc}
\text{degree} & 10 & 11 & 12 & 13 & 14 & 15 & 16 \\
\hline
\text{homology} & \Z/2\Z & 0 & \Z/12\Z & \Z/2\Z & \Z/2\Z & 0 & \Z  
\end{array}$$

\end{example}

\subsection{The algebra of Koszul kernels}\label{subsec-KK}
When $n=1$, theorem \ref{thm-BGamma-Z} has a nicer formulation in terms of the algebra of Koszul kernels.

If $W$ is a graded $\Fp$-vector space concentrated in positive odd degrees, we denote by $K_\Fp(W)$ the dual Koszul algebra over $W$, that is, the graded $\Fp$-algebra $\Gamma_{\Fp}(W[1])\otimes_{\Fp}\Lambda_{\Fp}(W)$ (the index `$\Fp$' is put here to emphasize that we work in the realm of $\Fp$-vector spaces, in particular $\Gamma_{\Fp}(W[1])$ and $\Lambda_{\Fp}(W)$ equal $\Fp$ in degree zero) equipped with the Koszul differential, defined as the composite
\begin{align*}&\Gamma^n_{\Fp}(W[1])\otimes_{\Fp} \Lambda_{\Fp}^k(W)\xrightarrow[]{\Delta\otimes \Id}\Gamma^{n-1}_{\Fp}(W[1])\otimes_{\Fp} W[1]\otimes_{\Fp} \Lambda^k_{\Fp}(W)\\&\quad\xrightarrow[]{\Id} \Gamma^{n-1}_{\Fp}(W[1])\otimes_{\Fp} W\otimes_{\Fp} \Lambda^k_{\Fp}(W)\xrightarrow[]{\Id\otimes m} \Gamma^{n-1}_{\Fp}(W[1])\otimes_{\Fp} \Lambda^{k+1}_{\Fp}(W)\;. \end{align*} 
The dual Koszul algebra is a cdg-$\Fp$-algebra, it has weights if $W$ is equipped with weights, and it is a $\P_\Fp$-graded algebra if $W$ is a graded strict polynomial functor.

\begin{definition}
Let $W$ be a graded $\Fp$-vector space concentrated in positive odd degrees. The algebra of Koszul kernels $\KK_\Fp(W)$ is the commutative graded $\Z$-algebra defined by:
$\KK_\Fp(W)_0=\Z$ and in positive degrees $\KK_\Fp(W)$ equals the cycles of positive degree of the dual Koszul algebra $K_{\Fp}(W)$.
\end{definition}

\begin{lemme}\label{lm-KK}
Let $V$ be a graded $\Z$-module with weights, concentrated in odd degrees and $\Z$-free of finite rank in each degree. Put the weights on $V/pV$ such that the epimorphism $V\twoheadrightarrow V/pV$ preserves weights. There is an isomorphism of graded $\Z$-algebras with weights:
$$H(K^p(V))\simeq \KK_\Fp(V/pV)\;.$$
\end{lemme}
\begin{proof}
We have already proved in lemma \ref{lm-p-tors} that $H_0(K^p(V))\simeq \Z$ and that for positive $i$, $H_i(K^p(V))$ is isomorphic to the image of the differential of $K^1(V)$ tensored by $\Fp$. But $K^1(V)\otimes\Fp\simeq K_\Fp(V/pV)$ so the latter equals $\mathrm{Im}(d)$ where $d$ is the differential of $K_\Fp(V/pV)$. Finally, since $K_\Fp(V/pV)$ is exact in positive degrees (use the exponential formula and the K\"unneth isomorphism to reduce to the case when $V$ has rank one), $\mathrm{Im}(d)=\KK_\Fp(V/pV)$.
\end{proof}

Using lemma \ref{lm-KK}, the case $n=1$ of theorem \ref{thm-BGamma-Z} can be reformulated in the following way.
\begin{theoreme}\label{thm-BGamma-Z-bis}
Let $V[2]$ be a free $\Z$-module of finite rank concentrated in degree $2$ and weight one. The homology of $\overline{B}(\Gamma(V[2]))$ is isomorphic to the graded algebra with weights (where $(V/pV)^{(r)}[k]$ denotes a copy of $V$ with degree $k$ and weight $p^r$):
$$\Lambda(V[3])\otimes \bigotimes_{\text{$p$ prime}}\KK_{\Fp}\left(\bigoplus_{k\ge 0} (V/pV)^{(k+1)}[2p^{k+1}+1]\right)\;.$$
\end{theoreme}

Observe that the graded algebra with weights appearing in theorem \ref{thm-BGamma-Z-bis} is actually a graded $\P_\Z$-algebra. We do not claim in theorem \ref{thm-BGamma-Z-bis} that the isomorphism is an isomorphism of graded $\P_\Z$-algebras, but however we conjecture that this is the case, at least for the $p$-primary part when $p$ is odd. The proof of this (and further developments) is a work in progress with L. Breen and R. Mikhailov \cite{BMT}.

\subsection{Ext-computations over the integers}

By theorem \ref{thm-bar} the graded algebras with weights $\E(S,\Lambda)(\Z^m)$ and $\E(S,\Gamma)(\Z^m)$ are simply obtained from $\B^n(\Gamma(\Z^m[2]))$ by regrading. For instance, we obtain the following $\Ext$-computation from example \ref{ex-comput}.
\begin{example}\label{ex-comput2}
The groups $\Ext^i_{\P_\Z}(S^4,\Lambda^4)$ and $\Ext^i_{\P_\Z}(S^4,\Gamma^4)$ are given by the following tables (they are zero outside the table).
 $$\begin{array}{c|cccc}
\text{degree $i$} & 0 & 1 & 2 & 3 \\
\hline
\Ext^i_{\P_\Z}(S^4,\Lambda^4) & 0 & \Z/2\Z & \Z/3\Z & \Z/2\Z   
\end{array}$$

 $$\begin{array}{c|ccccccc}
\text{degree $i$} & 0 & 1 & 2 & 3 & 4 & 5 & 6 \\
\hline
\Ext^i_{\P_\Z}(S^4,\Gamma^4) & \Z & 0 & \Z/2\Z & \Z/2\Z & \Z/12\Z & 0 & \Z/2\Z  
\end{array}$$
\end{example}
 
We can also give general formulas by regrading theorems \ref{thm-BGamma-Z} and \ref{thm-BGamma-Z-bis}. 
If $W$ is a graded $\Fp$-module with weights concentrated in even degrees, we let 
$\widetilde{\KK}_{\Fp}(W)$ 
be the wg-$\Z$-algebra which equals $\Z$ in degree zero and which equals the subalgebra of cycles of the wdg-$\Z$-algebra $\,^\t\Gamma_\Fp(W[1])\otimes^1  \Lambda_\Fp(W)$ (equipped with the Koszul differential) in positive degrees. So $\widetilde{\KK}_{\Fp}(W)$ is equal to the algebra of Koszul kernels ${\KK}_{\Fp}(W)$ if $p=2$, and it differs from ${\KK}_{\Fp}(W)$ by signs in the multiplication if $p$ is odd. Moreover:
$$\widetilde{\KK}_{\Fp}(\R_3 V)=\R_3(\KK_\Fp(V))$$
when $V$ is a graded $\Fp$-module with weights concentrated in odd degrees and odd weights, or when $p=2$. Thus, theorem \ref{thm-BGamma-Z-bis} and theorem \ref{thm-bar} yield the following computation of $\E(S,\Lambda)(\Z^m)$.

\begin{theoreme}\label{thm-calculsanstwistIII}
Let $\Z^m$ be a free abelian group. There is an isomorphism of graded $\Z$-algebras with weights
$$\E(S,\Lambda)(\Z^m)\simeq \Lambda(\Z^m\langle 0\rangle)\otimes^1 {\bigotimes_{\text{$p$ prime}}}^1\;\widetilde{\KK}_{\Fp}\left(\bigoplus_{k\ge 0} (\Z^m/p\Z^m)^{(k+1)}\langle p^{k+1}-1\rangle\right)\;.$$
On the right hand side, $\Z^m\langle 0\rangle$ is a copy of $\Z^m$ having degree zero and weight $1$, and $(\Z^m/p\Z^m)^{(k)}\langle i\rangle$ denotes a copy of $\Z^m/p\Z^m$ having cohomological degree $i$ and weight $p^k$.
\end{theoreme}

Using theorems \ref{thm-BGamma-Z-bis} and theorem \ref{thm-bar}, we also obtain a computation of the graded $\Z$-algebras with weights $\E(S,\Gamma)(\Z^m)$.
\begin{theoreme}\label{thm-calculsanstwistIV}
Let $\Z^m$ be a free abelian group. There is an isomorphism of graded $\Z$-algebras with weights
$$\E(S,\Gamma)(\Z^m)\simeq \Gamma(\Z^m\langle 0\rangle)\otimes\bigotimes_{\text{$p$ prime}} \widehat{_pH}(\widetilde{X}_p).$$
On the right hand side, $\widetilde{X}_p$ is the wdg-$\Z$-algebra defined as the tensor product:
$$K^p\left(\bigoplus_{k,\ell\ge 0} {\Z^m}_{p^{k+\ell+2}}\langle 2p^{k+\ell +2}-2p^{k+1}-1\rangle\right)\otimes \Omega^p\left(\bigoplus_{\ell\ge 0} {\Z^m}_{p^{\ell+1}}\langle 2p^{k+1}-2\rangle\right),$$
where ${\Z^m}_k\langle i\rangle$ denotes a copy of ${\Z^m}$ having weight $k$ and cohomological degree $i$ (or equivalently homological degree $-i$).
\end{theoreme}

\part{Frobenius twists}\label{part-4}

In this part, $\k$ is a field of prime characteristic $p$. We compute the $\P_\k$-graded algebras $\E(X^{(r)},Y^{(s)})$, when $X$ and $Y$ are classical exponential functors and $r,s$ are nonnegative integers. Our method offers a unified treatment of all the cases. We proceed in several steps. 

In section \ref{sec-frob-bar}, we show that the $\P_\k$-graded algebras of the form $\E(X^{(t)},Y)$ are equal, up to a regrading, to the algebras $\E(X,Y)^{(t)}$. 

In section \ref{sec-tw-ss}, we use the twisting spectral sequence from \cite{TouzeTroesch}. In our case, it is proved in \cite{TouzeTroesch} that the spectral sequence collapses at the second page. As a consequence, we obtain that the $\P_\k$-graded algebras $\E(X^{(r)},Y^{(s)})$ can be easily computed (up to a filtration) from $\E(X^{(r-s)},Y)$.

To finish the computation, we need the filtrations on $\E(X^{(r)},Y^{(s)})$ to be trivial. This is the purpose of section \ref{sec-split}, which is of independent interest. We prove that for some filtered $\P_\k$-graded algebras $A$ with prescribed $\Gr A$, the filtration must split. 
The results are stated in section \ref{sec-final-results}. 

\section{D\'ecalages}\label{sec-frob-bar}

Let $A=\{A_{0,d}\}_{d\ge 0}$ be a $\P_\k$-algebra. Recall the external product:
$$\E^i(F,A_{0,d})\otimes \E^j(G,A_{0,e})\to \E^{i+j}(F\otimes G,A_{0,d+e})$$ 
induced by tensor products and the multiplication of $A$.
The following d\'ecalage formula was first obtained, by other means, in \cite[Prop 2.6]{Chalupnik2}.

\begin{proposition}\label{prop-twistgauche}
Let $\k$ be a field of prime characteristic $p$, let $F\in\P_{\k}$, and let $t$ be a nonnegative integer. For all integer $i$, there are isomorphisms of strict polynomial functors\footnote{In the statement, we take the convention that $\E^i(F,G)=0$ for $i<0$. To emphasize the analogy between the three cases, we have also written $\E^i(F,S^d)$ although these extension groups actually reduce to $\H(F,S^d)$ since symmetric powers are injective.}, natural in $F$:
\begin{align*}
&\E^i(F,S^d)^{(t)}\simeq \E^i(F^{(t)},S^{dp^t})\;, \\
&\E^{i}(F,\Lambda^d)^{(t)}\simeq \E^{i+(p^t-1)d}(F^{(t)},\Lambda^{dp^t})\;, \\
&\E^i(F,\Gamma^d)^{(t)}\simeq \E^{i+2(p^t-1)d}(F^{(t)},\Gamma^{dp^t}) \;.
\end{align*} 
Moreover, if $Y=S,\Lambda$ or $\Gamma$, the external product
$$\E^*(F,Y^d)^{(t)}\otimes \E^*(G,Y^e)^{(t)}
\to \E^{*}(F\otimes G,Y^{d+e})^{(t)}$$
identifies through the isomorphism above with the external product:
\begin{align*}\E^{*}(F^{(t)},Y^{dp^t})\otimes \E^{*}(G^{(t)},Y^{ep^t})\to \E^{*}(F^{(t)}\otimes G^{(t)},Y^{dp^t+ep^t})\;.
\end{align*}
\end{proposition}

\begin{proof}We will prove that the isomorphisms actually hold at the level of chain complexes. Since there are no extension groups between homogeneous functors of different weights, we can assume that $F$ is homogeneous of weight $d$.
Let us recall from \cite[Lemmas 2.2 and 2.3]{TouzeTroesch} an elementary computation in $\P_\k$ . If $\mu=(\mu_1,\dots,\mu_n)$ is a tuple of positive integers, we denote by $S^\mu$ the tensor product $S^{\mu_1}\otimes\dots\otimes S^{\mu_n}$, and by $\alpha\mu$ the tuple $\alpha\mu:=(\alpha\mu_1,\dots,\alpha\mu_n)$. For all $G\in\P_\k$, there are isomorphisms (the first one is induced by precomposition by $I^{(t)}$, the second one by the canonical inclusion $S^{\mu\,(t)}\hookrightarrow S^{p^t\mu}$):
\begin{align*}&\H(F,S^\mu)^{(t)}\simeq \H(F^{(t)},S^{\mu(t)})\;,& (i)\\
&\H(F^{(t)},S^{\mu(t)})\simeq \H(F^{(t)},S^{p^t\mu})\;,& (ii)\\
&\H(F^{(t)},S^{\lambda})\simeq 0\text{ if $\lambda$ is not of the form $p^t\mu$\;.}&(iii)\end{align*} 

Since $\E^*(F,S^d)^{(t)}=\H(F,S^d)^{(t)}$, the composition of the isomorphisms $(i)$ and $(ii)$ provides the required isomorphism for the case of $S^d$. The compatibility with external products is straightforward.

Now we turn to the case of $\E^{*}(F,\Lambda^d)^{(t)}$. We use bar constructions of symmetric algebras in the same fashion as in the proof of proposition \ref{prop-cplx-ext}. There a quasi-isomorphism of dg-$\P_\k$-algebras $\Lambda[1]\hookrightarrow \B S$, which becomes after regrading a quasi-isomorphism 
$\Lambda\hookrightarrow {}^t\R_{1}(\B S) $. The homogeneous part of weight $d$ of $^\t\R_{1}(\B S)$ is a coresolution of $\Lambda^d$ by symmetric powers, which we denote by $^\t\R_{1}I_d$. So $\E^i(F^{(t)},\Lambda^{dp^t})$ is the $(-i)$-th homology group of the complex $\H(F^{(t)},{}^\t\R_{1}I_{dp^t})$. We consider the morphism of complexes defined as the composition
\begin{align*}\H(F,{}^\t\R_{p^t}I_d)^{(t)}&\simeq \H(F^{(t)},(^\t\R_{p^t}I_d)^{(t)})&\\&= \H(F^{(t)},{}^\t\R_{1}(I_d^{(t)}))\to \H(F^{(t)}, {}^\t\R_{1}I_{dp^t})\;.& (iv)\end{align*}
where the first isomorphism is the isomorphism $(i)$, the equality comes from the equality $\R_{1}(\B S^{(t)})=(\R_{p^t}(\B S))^{(t)}$, and the last morphism is induced by the morphism of graded $\P_\k$-algebras $S^{(t)}\hookrightarrow S$ (which maps $S^{d(t)}$ into $S^{dp^t}$), which induces a morphism $\B S^{(t)}\hookrightarrow \B S$, hence a morphism ${}^\t\R_{1}(I_{d}^{(t)})\to {}^\t\R_{1}I_{dp^t}$. We claim that the last map in the composition $(iv)$ is an isomorphism of complexes. Indeed, the objects of the complex ${}^\t\R_{1}I_{dp^t}$ are symmetric tensors $S^\lambda$. If $\lambda=p^t\mu$, then  there is a summand $S^{\mu\,(t)}$ of $(^\t\R_{-1}I_d)^{(t)}$ such that the restriction of the map ${}^\t\R_{-1}(I_{d}^{(t)})\to {}^\t\R_{-1}I_{dp^t}$ to $S^{\mu\,(t)}$ is the canonical inclusion $S^{\mu\,(t)}\hookrightarrow S^{p^t\mu}$. This canonical inclusion induces an isomorphism after applying $\H(F^{(t)},-)$, by $(v)$. If $\lambda$ is not of the form $p^t\mu$, then $\H(F^{(t)},S^{\lambda})\simeq 0$ by  $(iii)$. Hence the composition $(iv)$ is an isomorphism of complexes. Taking the $(-i)$-th homology group, we obtain an isomorphism 
$$\E^i(F^{(t)}, \Lambda^{dp^t})\simeq\E^{i-d(p^t-1)}(F,\Lambda^d)^{(t)}\;.$$ 
It remains to check the compatibility with external products. This is a straightforward check at the level of cochain complexes, similar to the identification of the convolution product in the proof proposition \ref{prop-cplx-ext}. The case of $\E^{*}(F,\Gamma^d)$ is similar.
\end{proof}

If $C$ is a $\P_\k$-coalgebra, and $Y$ is a classical exponential functor, the convolution product of $\E(C,Y^{(t)})$ is obtained by combining the external product and the map induced by the comultiplication of $C$. So proposition \ref{prop-twistgauche} implies the following result, which we formulate using the regrading functor $\R_\alpha$ from section \ref{sec-regrading}.

\begin{corollary}\label{cor-decalage} Let $C$ be a $\P_\k$-coalgebra and let $Y$ be a classical exponential functor. For all $t\ge0$, there is an isomorphism $\P_\k$-graded algebras:
$$\left(\R_{\alpha(Y)}\E(C,Y)\right)^{(t)}\simeq \E(C^{(t)},Y)\;,$$
with $\alpha(S)=0$, $\alpha(\Lambda)=p^t-1$ and $\alpha(\Gamma)=2(p^t-1)$. 
\end{corollary}

For the convenience of the reader, we write out explicitly two cases of application of corollary \ref{cor-decalage}.

\begin{example}\label{ex-1ststep}
{\bf (1)} The divided powers $\Gamma^d$ are projective strict polynomial functors, and the Yoneda lemma yields an isomorphism of $\P_\k$-algebras:
$$\E(\Gamma,\Lambda)=\H(\Gamma,\Lambda)\simeq \Lambda\;.$$
Let $\Lambda\langle i\rangle^{(t)}$ denotes the $\P_\k$-graded algebra with $\Lambda^{d\,(t)}$ in cohomological degree $i$ and weight $dp^t$ (hence $\Lambda\langle i\rangle^{(t)}=\Lambda[-i]^{(t)}$). Corollary \ref{cor-decalage} implies:
$$\E(\Gamma^{(t)},\Lambda)\simeq \Lambda\langle p^t-1\rangle^{(t)}\;.$$

\noindent
{\bf (2)} Assume that $p=2$. Let $I^{(k)}\langle i\rangle$ denote a copy of the $k$-th Frobenius twist functor, placed in cohomological degree $i$ (hence $I^{(k)}\langle i\rangle=I^{(k)}[-i]$). We computed $E(S,\Lambda)$ in theorem \ref{thm-calculsanstwistII}, so corollary \ref{cor-decalage} implies:
$$\E(S^{(t)},\Lambda)\simeq \Gamma\big(\textstyle\bigoplus_{k\ge 0}I^{(k+t)}\langle p^{k+t}-1\rangle)\;.$$
\end{example}

\section{The twisting spectral sequence and Troesch complexes}\label{sec-tw-ss}

Let $r,s$ be nonnegative integers, and assume that $t:=r-s$ is nonnegative. In this section we explain how to recover, up to a filtration, the $\P_\k$-graded algebra $\E(X^{(r)},Y^{(s)})$ from the $\P_\k$-graded algebra $\E(X^{(t)},Y)$. This is a simple application of the results proved in \cite{TouzeTroesch}, but for the reader's convenience we recall here the results we need for our computation.

\subsection{Parameterization by graded vector spaces}\label{subsec-param-graded}

Let $V$ be a finite dimensional graded vector space. The parameterized functor $F_V:W\mapsto F(V\otimes W)$ actually carries a grading, defined in the following way. Let the multiplicative group $\mathbb{G}_m$ act rationally on a homogeneous vector $v\in V$ of degree $i$ by the formula $\lambda\cdot v:=\lambda^i v$. Since $F$ is a strict polynomial functor, this induces a rational action of $\mathbb{G}_m$ on $F_V(W)$ for all $W\in\V_\k$. Thus, $F_V$ splits as a direct sum $F_V=\bigoplus_{i\in\Z}(F_V)_i$, where $(F_V)_i(W)$ is the subspace of weight $i$ of $F_V(W)$ under the action of $\mathbb{G}_m$.

Thus, parameterization by a graded vector space $V$ defines a weight-preserving functor ($\P_\k^*$ is the category of graded strict polynomial functors): 
$$\begin{array}{ccc}\P_\k&\to &\P_\k^*\\
F&\mapsto & F_V 
\end{array}.$$

\begin{example}
If $F$ is a symmetric power, an exterior power, a divided power or a tensor product of these, the notion of degree defined on $F_V(W)$ coincides with the usual definition. For example, if $V$ has a homogeneous basis formed by $e_{-1}$ of degree $-1$ and $e_4$ of degree $4$, and if $w_i$ denote elements of $W$, the element $(e_{-1}\otimes w_1)(e_4\otimes w_2)\otimes (e_{-1}\otimes w_3)$ is an element of degree $-1+4-1=2$ of $(S^2\otimes S^1)_V(W)$. 
\end{example}

The following example will also be important for our computations.
\begin{example}
If $F=I^{(r)}$ and $V=\oplus V_i$ then $(I^{(r)})_V=\oplus V_i^{(r)}\otimes I^{(r)}$, where the summand $V_i^{(r)}\otimes I^{(r)}$ is homogeneous of degree $p^ri$. In particular, the functors $(I_V)^{(r)}$ and $(I^{(r)})_V$ do not bear the same gradings, although they are isomorphic in an ungraded way.
\end{example}

For our purposes, we need to define parameterization of $\P_\k$-graded algebras by graded vector spaces. Let $A$ be a $\P_\k$-graded algebra and denote by $A_{i,d}$ the homogeneous part of $A$ of degree $i$ and weight $d$. The parameterized functor $(A_{i,d})_V$ has a weight and a bigrading. To be more specific, the first partial degree is the degree $i$ coming from the grading of $A$, and the second partial degree is induced by the grading of $V$. To get rid of bidegrees, we define the degree on $(A_{i,d})_V$ as the total degree associated to the bidegree. In this way, parameterization by a graded vector space $V$ defines a functor:
$$\begin{array}{ccc}\{\P_\k\text{-g-alg}\}&\to &\{\P_\k\text{-g-alg}\}\\
A&\mapsto & A_V 
\end{array}. $$
\begin{example}
Let $V$ be a graded vector space with homogeneous basis $(e_2,e_4)$ with $e_i$ of degree $i$. Then $S[16]_V= S[18]\otimes S[20]$.
\end{example}

\subsection{The twisting spectral sequence}

Let us denote by $E_s$ the graded vector space $\Ext_{\P_\k}^*(I^{(s)},I^{(s)})$. It is one dimensional in cohomological degrees $2i$, for $0\le i<p^s$ and zero everywhere else, see e.g. \cite[Thm 4.5]{FS} or \cite[Cor 4.7]{TouzeTroesch}. If $G\in\P_\k$, the parameterized functor $G_{E_s}$ inherits from $E_s$ a cohomological grading, as explained in section \ref{subsec-param-graded}. We denote by $(G_{E_s})^j$ its homogeneous part of cohomological\footnote{As usual, we can convert cohomological degrees into homological degrees by the formula $M^i=M_{-i}$.} degree $j$.

Recall from \cite[Thm 7.1]{TouzeTroesch} that for strict polynomial functors $F,G$ with finite dimensional values and for $s\ge 0$, there is a `twisting spectral sequence', natural with respect to $F$ and $G$:
$$E^{i,j}(F,G,s):= \Ext_{\P_\k}^{i}(F,(G_{E_s})^j)\Longrightarrow \Ext^{i+j}_{\P_\k}(F^{(s)}, G^{(s)})\,.$$
Let us put the parameterized strict polynomial functor $G_{V^{(s)}}$ instead of $G$ in the spectral sequence above. Since we have 
$$ (G_{V^{(s)}})^{(s)} = (G^{(s)})_V\;\text{ and }  (G_{V^{(s)}})_{E_s}= G_{V^{(s)}\otimes E_s}\;,$$
the spectral sequence now takes the form:
$$E^{i,j}(F,G_{V^{(s)}},s)=(\E^i(F,G)_{E_s})^{(s)\,j}\Longrightarrow \E^{i+j}(F^{(s)}, G^{(s)})\;,$$
where for all $i\ge 0$, $(\E^i(F,G)_{E_s})^{(s)\,j}$ denotes the homogeneous part of cohomological degree $j$ of the graded strict polynomial functor $(\E^i(F,G)_{E_s})^{(s)}$.

If $G$ is a symmetric power, an exterior power, a divided power or a tensor product of these, it is proved in \cite[Thm 8.11]{TouzeTroesch} that the spectral sequence collapses at the second page. 

Finally, it is proved in \cite[Thm 7.1]{TouzeTroesch} that the twisting spectral is compatible with tensor products and natural with respect to $F$ and $G$. Thus, if $X$ and $Y$ are classical exponential functors and if $t\ge 0$, we can retrieve (up to a filtration) the convolution product of the abutment $\E(X^{(s+t)}, Y^{(s)})$ from the convolution product of the second page $(\E(X^{(t)},Y)_{E_s})^{(s)}$. The following statement sums up our discussion.

\begin{theoreme}\label{thm-twisting-ss}
Let $\k$ be a field of positive characteristic, let $X,Y$ be classical exponential functors, and let $t$ and $s$ be nonnegative integers. Then $\E(X^{(t+s)},Y^{(s)})$ is a filtered $\P_\k$-graded algebra and there is an isomorphism of $\P_\k$-graded algebras:
$$ (\E(X^{(t)},Y)_{E_s})^{(s)}\simeq \Gr\left( \E(X^{(t+s)},Y^{(s)}) \right)\;.$$
\end{theoreme}

The algebra appearing on the left hand side of the isomorphism of theorem \ref{thm-twisting-ss} is easy to compute with the help of corollary \ref{cor-decalage}. 

\begin{example}\label{ex-calc-twist}{\bf (1)} Let $I^{(k)}\langle i\rangle$ denote a copy of the $k$-th Frobenius twist, with cohomological degree $i$ (hence $I^{(k)}\langle i\rangle=I^{(k)}[-i]$). There is, up to a filtration, an isomorphism of $\P_\k$-graded algebras:
$$\E(\Gamma^{(t+s)}, \Lambda^{(s)})\simeq \Lambda\left(\textstyle\bigoplus_{0\le i<p^s} I^{(t+s)}\langle (2i+1)p^t-1\rangle\right)\;.$$

\noindent
{\bf (2)} Assume that $p=2$. There is, up to a filtration, an isomorphism of $\P_\k$-graded algebras:
$$\E(S^{(t+s)},\Lambda^{(s)})\simeq \Gamma\left(
\textstyle\bigoplus_{0\le i<p^s}\bigoplus_{k\ge 0} I^{(k+t+s)}\langle(2i+1)p^{k+t}-1\rangle
\right)\;.$$ 
\end{example}

\section{Splitting filtrations}\label{sec-split}
This section is of independent interest: we prove general splitting results for filtrations of $\P_\k$-graded algebras. We will use these results in section \ref{sec-final-results} to prove that the filtrations appearing in theorem \ref{thm-twisting-ss} split.

\subsection{General facts on filtrations}
We begin with elementary observations about filtered functors.
\begin{lemme}\label{lm-filtr-str}
Let $\k$ be a field and let $F$ be a filtered strict polynomial functor over $\k$ with finite dimensional values.
\begin{enumerate}
\item[(i)] The filtration of $F$ has finite length.
\item[(ii)] If $\Ext^1_{\P_\k}(\Gr F,\Gr F)=0$, there is an isomorphism $F\simeq \Gr F$.
\end{enumerate}
\end{lemme}
\begin{proof}
To prove (i), observe that a filtration of $F$ is equivalent to a filtration of the $S(\k^d,d)$-module $F(\k^d)$ where $d$ is the strict polynomial degree of $F$. Hence the filtration is finite for dimension reasons. 
Let us prove (ii). Let $F_1\subset F_2\subset \dots\subset F_n=F$ be the filtration of $F$. If $\Ext^1_{\P_\k}(\Gr F,\Gr F)$ equals zero, then $\Ext^1_{\P_\k}(F_1,F_2/F_1)$ equals zero. Hence the extension $F_1\hookrightarrow F_2\twoheadrightarrow F_2/F_1$ splits, that is $F_2\simeq F_1\oplus F_2/F_1$. In this way we build inductively an isomorphism $F\simeq \Gr F$. 
\end{proof}

Without further indication, filtrations of $\P_\k$-graded algebras will always mean \emph{bounded} filtrations (that is, on each homogeneous component $A_{i,d}$ of the $\P_\k$-graded algebra $A$, the filtration has finite length). For example, filtrations of exponential functors are always bounded (since the $E_{i,d}$ take finite dimensional values, this follows from lemma \ref{lm-filtr-str}). Also, filtered $\P_k$-graded algebras appearing as abutments of first quadrant spectral sequences like the twisting spectral sequence from section \ref{sec-tw-ss} have bounded filtrations.
In our study of filtered $\P_\k$-graded algebras, we shall use $(1,\epsilon)$-commutativity defined in section \ref{subsec-convol}. The following lemma is an easy check.

\begin{lemme}\label{lm-1E-comm}
Let $A$ be a filtered $\P_\k$-graded algebra. If $A$ is $(1,\epsilon)$-commutative, so is $\Gr A$.
\end{lemme}

Finally, we observe that the exponential property behaves well with filtrations of algebras.

\begin{lemme}\label{lm-exp-gr}
Let $\k$ be a field and let $A$ be a filtered $\P_\k$-graded algebra.
If $\Gr A$ is a graded exponential functor, then so is $A$. 
\end{lemme} 
\begin{proof}
Consider $A(V)\otimes A(W)$ as a weighted graded $\k$-module with the tensor product filtration 
$$F^i \big(A(V)\otimes A(W)\big)=\textstyle\sum_{k+\ell=i} F^kA(V)\otimes F^\ell A(W)\;.$$ 
The multiplication induces a map of weighted graded $\k$-modules
$A(V)\otimes A(W)\xrightarrow[]{\mathrm{m}} A(V\oplus W)$
compatible with the filtrations. Since $\Gr A$ is exponential, the associated graded map $\mathrm{Gr} m$ is an isomorphism on each homogeneous part of degree $i$ and weight $k$.
The filtrations of the homogeneous part of degree $i$ and weight $d$ has finite length. So by iterated uses of the five lemma, the map $m:(A(V)\otimes A(W))_{i,k}\to A(V\oplus W)_{i,k}$ is also an isomorphism. Thus $A$ is exponential.
\end{proof}

\subsection{Splitting results}
\subsubsection{The odd characteristic case}

\begin{lemme}\label{lm-ext}
Let $\k$ be a field of odd characteristic. Let $m,n$ be positive integers, let $\lambda$ and $\alpha$ be $m$-tuples of nonnegative integers and let $\mu$ and $\beta$ are $n$-tuples of nonnegative integers. Let $G$ be a functor of the form $ G= \left(\textstyle\bigotimes_{i=1}^m S^{\lambda_i\, (\alpha_i)}\right)\otimes \left(\textstyle\bigotimes_{i=1}^n \Lambda^{\mu_i\, (\beta_i)}\right)$. Then $\Ext^1_{\P_\k}(G,G)=0$.
\end{lemme}
\begin{proof} By iterated uses of lemma \ref{lm-facile-Ext} (or use \cite[Cor 1.8]{FFSS}), it suffices to check that $\Ext^1(X,Y)=0$ if $X$ and $Y$ are of the form $S^{\ell \,(r)}$ or $\Lambda^{m\,(s)}$.

We first check that it is true if the functors are not twisted:
$\Ext^1(S^d,S^d)=\Ext^1(\Lambda^d,S^d)=0$ by injectivity of $S^d$, $\Ext^1(\Lambda^d,\Lambda^d)=0$ by remark \ref{rk-Ext-Lambda-Lambda}, and $\Ext^1(S^d,\Lambda^d)=0$ by theorem \ref{thm-calculsanstwistI}. By proposition \ref{prop-twistgauche}, $\Ext^1$ also vanish if one of the two functors $X,Y$ is twisted (indeed, the $\Ext$-degree is shifted by an even integer since $p$ is odd). Finally theorem \ref{thm-twisting-ss} shows the vanishing of $\Ext^1(X,Y)$ in the general case (indeed, if $A$ is a negatively graded $\P_\k$-algebra which is zero in degree $-1$, then so is $A_{E_s}$).
\end{proof}

\begin{proposition}[Splitting result I]\label{prop-trivFp_1}
Let $\k$ be a field of odd characteristic, and let $A$ be a filtered $\P_\k$-graded algebra. Assume that $A$ is graded commutative and that $\mathrm{Gr} A$ has one of the following form:
$$(i)\; S(F)\otimes \Lambda(G)\;,\quad (ii)\; \Gamma(F)\otimes \Lambda(G)\;, $$
where $F$ and $G$ are graded strict polynomial functors which are additive.
Then there is an isomorphism of $\P_\k$-graded algebras $A\simeq \mathrm{Gr}A$.
\end{proposition}
\begin{proof}
The proof relies essentially on lemma \ref{lm-ext} and the use of universal properties of $\mathrm{Gr}A$. We prove case $(ii)$, which is slightly more difficult.

{\bf Step 1: splitting without products.}
For all $i,d$, the extension group $\Ext^1_{\P_\k}(\mathrm{Gr} A_{i,d}, \mathrm{Gr} A_{i,d})$ equals $\Ext^1_{\P_\k}(\mathrm{Gr} A_{i,d}^\sharp, \mathrm{Gr} A_{i,d}^\sharp)$ by lemma \ref{lm-sym-res}, which equals zero by lemma \ref{lm-ext}. So we get an isomorphism of graded strict polynomial functors $A\simeq \mathrm{Gr}A$, whence a surjective map: $\phi:A\twoheadrightarrow F\oplus G$. There is an extra grading on $\Gr A$, namely the \emph{filtration degree} (that is $Gr_k A$ is the homogeneous part of $\Gr A$ of filtration degree $k$). If we define a filtration on $\Gr A$ (hence on $F\oplus G$) by $\F_k \Gr A=\bigoplus_{i\le k}\Gr_i A$, then $\phi$ preserves the filtrations, and $\Gr\phi$ equals the canonical subjection $\Gr A\twoheadrightarrow F\oplus G$. 

{\bf Step 2: coalgebra structures and universal property.}
Since $F$ and $G$ are additive, $\Gr A$ is a graded exponential functor. Thus, by lemma \ref{lm-exp-gr}, $A$ is a graded exponential functor. In particular, $A$ also has a coalgebra structure. Since the filtration of $A$ is natural and compatible with products, $A$ is a filtered coalgebra, and $\Gr(A)=\Gamma(F)\otimes^{\epsilon} \Lambda(G)$ as a coalgebra.

Now $\Gr A$ is a graded commutative algebra. Hence $\Gamma(F)$ and $\Lambda(G)$ are graded commutative. Hence $F$ (resp. $G$) is concentrated in even (resp. odd) degrees. Thus, $\Gamma(F)\otimes \Lambda(G)$ is the universal cofree graded commutative coalgebra on $F\oplus G$. Thus, by the universal property of the cofree algebra (see e.g. \cite[Chap. 1]{LodayValette}) $\phi$ induces a map $\psi:A\to \Gamma(F)\otimes \Lambda(G)$. 

{\bf Step 3: conclusion}
Let us consider the trivial filtration on $\Gamma(F)\otimes \Lambda(G)$ induced by the filtration $\F_k(F\otimes G)$ of $F\otimes G$ from Step 1. Then $\psi$ preserves the filtration and by construction $\Gr\psi$ equals the identity map  $\Gr A=\Gamma(F)\otimes \Lambda(G)$. In particular $\Gr\psi$ is an isomorphism. So $\psi$ is an isomorphism of coalgebras. Hence by lemma \ref{lm-morph}, $\psi$ is an isomorphism of algebras.
\end{proof}

We also prove splitting results for $(1,1)$-commutative algebras. Recall the signed tensor product $\otimes^1$ defined in section \ref{subsec-1E-com}.

\begin{proposition}[Splitting result II]\label{prop-trivFp_2}
Let $\k$ be a field of odd characteristic, and let $A$ be a $(1,1)$-commutative filtered $\P_\k$-graded algebra over $\k$. Assume that $\mathrm{Gr} A$ has one of the following forms
\begin{align*}
&(i)\; S(F)\otimes^1 \Lambda(G)\;,\quad (ii)\;  {^\t S(F)\otimes^1 \Lambda(G)}\;, \quad (iii)\;  S(F)\otimes^1 {^\t\Lambda(G)}\;,\\
&(vi)\; \Gamma(F)\otimes^1 \Lambda(G)\;,\quad (v)\;  {^\t \Gamma(F)\otimes^1 \Lambda(G)}\;, \quad (vi)\;  \Gamma(F)\otimes^1 {^\t\Lambda(G)}\;,
\end{align*}
where $F$ and $G$ are graded strict polynomial functors which are additive. Then there is an isomorphism of $\P_\k$-graded algebras $A\simeq \Gr A$.
\end{proposition}
\begin{proof}
The idea is to use the regrading functor $\R_1$ (and taking the opposite algebra of applying the functor $^\t$ if needed) to go back to the graded commutative case. Let us treat the case where $\Gr A=S(F)\otimes \Lambda(G)$. 

Since $A$ is $(1,1)$-commutative, so is $\Gr A$, hence so are $S(F)$ and $\Lambda(G)$ by lemma \ref{lm-tens-E}. Thus $F$ (resp. G) must be concentrated in odd (resp. even) degrees.
Thus $\R_1 A$ is graded commutative, with graded object $$\Gr(\R_1 A)=\R_1 (\Gr A)= {^\t S}(\R_1 F)\otimes \Lambda(\R_1 G)\;.$$ 
Let us denote by $B^\op$ the opposite algebra of an algebra $B$. Applying $^\op$ and $^\t$ (which do not alter graded commutativity) we get a graded commutative algebra $(^\t \R_1A)^{\op} $ with graded object $S(\R_1 F)\otimes \Lambda(\R_1 G)$. So by proposition \ref{prop-trivFp_1}, $(^\t \R_1A)^{\op} $ is isomorphic to $S(\R_1 F)\otimes \Lambda(\R_1 G)$. Applying $^\op$ and $\R_{-1}$ to the isomorphism obtained, we get an isomorphism $A\simeq \Gr A$.
\end{proof}

\subsubsection{The characteristic two case}

The following proposition is proved exactly in the same fashion as proposition \ref{prop-trivFp_1}.

\begin{proposition}[Splitting result III]\label{prop-trivF2_1}
Let $\k$ be a field of characteristic $2$, and let $A$ be a commutative filtered $\P_\k$-graded algebra. Assume that there is a graded strict polynomial functor $F$ which is additive, such that
$\mathrm{Gr}A=S(F)$ or $\mathrm{Gr}A=\Gamma(F)$. Then $A\simeq \mathrm{Gr}A$ as $\P_\k$-graded algebras. 
\end{proposition}

Proposition \ref{prop-trivF2_1} will be sufficient to prove that the filtrations on the abutment of the twisting spectral sequence split in characteristic $2$, except for the cases of $\E(\Lambda^{(t+s)}, S^{(s)})$ and $\E(\Gamma^{ (t+s)}, \Lambda^{ (s)})$. In these cases, the algebras are of exterior type.
So we need a analogue of proposition \ref{prop-trivF2_1} when $\mathrm{Gr}A=\Lambda(F)$. Two difficulties arise when we want to adapt the proof of proposition \ref{prop-trivFp_1} to this case. First, there might be non-trivial extensions between certain twisted exterior powers in characteristic $2$ (e.g. the extension $\Lambda^2\hookrightarrow \Gamma^2\twoheadrightarrow \Lambda^{1\,(1)}$). Second, exterior algebras still satisfy a universal property in characteristic $2$, but for strictly anticommutative algebras (recall that a $\k$-algebra is strictly anticommutative if for all $x\in A$, $x\cdot x=0$). So the best we can easily prove is the following statement.

\begin{proposition}[Splitting result IV]\label{prop-trivF2_2}
Let $\k$ be a field of characteristic $2$, and let $A$ be a strictly anticommutative filtered $\P_\k$-graded algebra. Assume that $\mathrm{Gr}A$ equals $\Lambda(F)$, where $F$ is a graded strict polynomial functor of the following very specific form. There exists an integer $r$ such that $F$ is a direct sum of copies graded copies of $I^{(r)}$. Then $A\simeq \mathrm{Gr}A$ as $\P_\k$-graded algebras. 
\end{proposition}
\begin{proof}
The proof that $\Ext^1_{\P_\k}(\mathrm{Gr} A_{i,d}, \mathrm{Gr} A_{i,d})$ equals zero reduces, by iterated uses of lemma \ref{lm-facile-Ext}, to checking that for all $d\ge 0$, $\Ext^1_{\P_\k}(\Lambda^{d\,(r)},\Lambda^{d\,(r)})$. But the latter fact follows from theorem \ref{thm-twisting-ss} and the vanishing of $\E^1(\Lambda^d,\Lambda^d)$. So lemma \ref{lm-filtr-str} yields an isomorphism of bigraded strict polynomial functors $A\simeq \mathrm{Gr}A$. To get an isomorphism of algebras, we use the universal property of exterior algebras, exactly as in the third step of the proof of proposition \ref{prop-trivFp_1}, and we conclude the proof in the same fashion. 
\end{proof}

\section{Final results}\label{sec-final-results}
\subsection{Statement of the results}\label{sec-final-state}

We now describe the $\P_\k$-graded algebras $\E(X^{(r)},Y^{(s)})$, for all classical exponential functors $X,Y$. Duality yields an isomorphism of exponential functors $\E(X^{(r)},Y^{(s)})\simeq \E(Y^{\sharp(s)},X^{\sharp(r)})$, so we can restrict to the case $r\ge s$. Hence we write $r=t+s$, with $t\ge 0$.  We recall the notations and conventions needed to read the statements.
\begin{enumerate}
\item In the $\P_\k$-graded algebra: 
$$\E(X^{(t+s)},Y^{(s)})=\textstyle\bigoplus_{i\ge 0,d\ge 0}\E^i(X^{d (s+t)},Y^{p^td (s)})\;,$$ 
$\E^i(X^{d (s+t)},Y^{p^td (s)})$ has cohomological degree $i$ and weight $dp^{s+t}$. The product is given by convolution, cf. definition \ref{def-conv}.

\item The functor $I^{(r)}\langle h\rangle$ denotes a copy of the $r$-th Frobenius twist, placed in cohomological degree $h$ (and weight $p^r$). If $X=S,\Lambda$ or $\Gamma$, and $F$ is a graded functor, then $X(F)$ denotes the composite $X\circ F$, with the usual gradings. The tensor product $\otimes$ is the usual tensor product of graded algebras (i.e. weights do not bring any Koszul sign).

\item In particular, in the theorems below, the functors $I^{(x+s+t)}\langle h\rangle$ appearing on the right hand side of the isomorphisms are direct summands of the functors
$\E^h(X^{p^x (s+t)},Y^{p^{x+t} (s)})$. 
\end{enumerate}

We order the classical exponential functors as follows: $\Gamma<\Lambda<S$.

\begin{theoreme}[Pairs $(X,Y)$ with $X\le Y$]\label{thm-XleY} Let $\k$ be a field of positive characteristic $p$, and let $s,t$ be positive integers, and let $X$ be an arbitrary exponential functor. There are isomorphisms of $\P_\k$-graded algebras:

\begin{align*}
&\E(X^{(t+s)},S^{(s)})\simeq {X}^\sharp \left(\bigoplus_{0\le i<p^s} I^{(t+s)}\langle 2ip^t\rangle\right)\;,\\
&\E(\Gamma^{(t+s)},\Lambda^{(s)})\simeq \Lambda\left(\bigoplus_{0\le i<p^s} I^{(t+s)}\langle (2i+1)p^t-1\rangle\right)\;, \\
&\E(\Lambda^{(t+s)},\Lambda^{(s)})\simeq \Gamma\left(\bigoplus_{0\le i<p^s} I^{(t+s)}\langle(2i+1)p^t-1\rangle\right)\;,\\
&\E(\Gamma^{(t+s)},\Gamma^{(s)})\simeq \Gamma\left(\bigoplus_{0\le i<p^s} I^{(t+s)}\langle(2i+2)p^t-2\rangle\right)\;.
\end{align*} 
\end{theoreme}

Now we turn to the pairs $(X,Y)$, with $X>Y$. These algebras were not computed in \cite{FFSS}, where the authors suspected that there are `no easy answer' for such pairs. Our approach explains why $\Ext$-groups for these pairs are much more difficult to compute. Indeed, for all pairs $(X,Y)$, the $\P_\k$-algebra $\E(X^{(t+s)},Y^{(t)})$ can be deduced from $\E(X,Y)$. The latter are very easy to compute if $X\le Y$ (they reduce to $\hom$-groups), but far from being trivial if $X>Y$ (they correspond to the homology of some Eilenberg Mac Lane spaces). The results of theorems \ref{thm-XYdurF2}, \ref{thm-XLY} and \ref{thm-SG} do not coincide with the results of \cite{Chalupnik2}, we compare them in section \ref{subsec-compare}.

\begin{theoreme}[Pairs $(X,Y)$ with $X> Y$ for $p=2$]\label{thm-XYdurF2}
Let $\k$ be a field of characteristic $p=2$, and let $s,t$ be positive integers. There are isomorphisms of $\P_\k$-graded algebras:
\begin{align*}
&\E(S^{(t+s)},\Lambda^{(s)})\simeq \Gamma\left(
\bigoplus_{
0\le i<p^s\,,\, 0\le k
} 
I^{(k+t+s)}\langle(2i+1)p^{k+t}-1\rangle\right)\;, \\
&\E(\Lambda^{(t+s)},\Gamma^{(s)})\simeq \Gamma\left(\bigoplus_{0\le i<p^s\,,\, 0\le k}
I^{(k+t+s)}\langle(2i+2)p^{k+t}-p^k-1\rangle\right)\;,\\
&\E(S^{(t+s)},\Gamma^{(s)})\simeq \Gamma\left(\bigoplus_{0\le i<p^s\,,\, 0\le k\,,\, 0\le \ell} I^{(k+\ell+t+s)}\langle(2i+2)p^{k+\ell+t}-p^k-1\rangle\right)\;.
\end{align*} 
\end{theoreme}

The computation of $\E(S^{(t+s)},\Lambda^{(s)})$ and $\E(\Lambda^{(t+s)},\Gamma^{(s)})$ in odd characteristic brings special signs. Recall from section \ref{subsec-1E-com} that $A\otimes^1 B$ denotes the `signed tensor product' of two $\P_\k$-graded algebras, and $^\t A$ denotes the weight twisted algebra associated to $A$ (defined in section \ref{sec-regrading}).

\begin{theoreme}\label{thm-XLY}
Let $\k$ be a field of odd characteristic $p$, and let $s,t$ be positive integers. The $\P_\k$-graded algebra $\E(S^{(t+s)},\Lambda^{(s)})$ is isomorphic to:
\begin{align*}\Lambda&\left(\bigoplus_{0\le i<p^s\,,\, 0\le k} I^{(k+t+s)}\langle(2i+1)p^{k+t}-1\rangle\right)\\&\break \otimes^1\;^\t\Gamma\left(\bigoplus_{0\le i<p^s\,,\, 0\le k} I^{(k+1+t+s)}\langle(2i+1)p^{k+1+t}-2\rangle\right)\;. \end{align*}
The $\P_\k$-graded algebra $\E(\Lambda^{(t+s)},\Gamma^{(s)})$ is isomorphic to:
\begin{align*}\Lambda&\left(\bigoplus_{0\le i<p^s\,,\, 0\le k} V^{(k+t+s)}\langle(2i+2)p^{k+t}-p^k-1\rangle\right)\\&\otimes^1 \;^\t\Gamma\left(\bigoplus_{0\le i<p^s\,,\, 0\le k} V^{(k+1+t+s)}\langle(2i+2)p^{k+1+t}-p^{k+1}-2\rangle\right) \;.\end{align*}
\end{theoreme}

\begin{theoreme}\label{thm-SG}
Let $\k$ be a field of odd characteristic $p$, and let $s,t$ be positive integers. The $\P_\k$-graded algebra $\E(S^{(t+s)},\Gamma^{(s)})$ is isomorphic to the tensor product:
\begin{align*}
&\Gamma\left(\bigoplus_{0\le i<p^s\,,\, 0\le k}
V^{(k+t+s)}\langle(2i+2)p^{k+t}-2\rangle
\right)\\
&\otimes\Lambda\left(\bigoplus_{0\le i<p^s\,,\, 0\le k\,,\,0\le\ell}
V^{(k+\ell+1+t+s)}\langle(2i+2)p^{k+\ell+1+t}-2p^k-1\rangle
\right)\\
&\otimes\Gamma\left(\bigoplus_{0\le i<p^s\,,\, 0\le k\,,\,0\le\ell}
V^{(k+\ell+2+t+s)}\langle(2i+2)p^{k+\ell+2+t}-2p^{k+1}-2\rangle
\right)
\end{align*}
\end{theoreme}

\subsection{Proof of theorems \ref{thm-XleY}-\ref{thm-SG}} We proceed in several steps.

{\bf Step 1: Computation of $\E(X,Y)$.} There are two cases. If $X\le Y$, then $\E(X,Y)$ reduces to $\H(X,Y)$, which is easy to compute, cf. lemma \ref{lm-element}. If $X> Y$, then $\E(X,Y)$ is rather complicated, and computed in theorems \ref{thm-calculsanstwistI} and \ref{thm-calculsanstwistII}.
In all cases, $\E(X,Y)$ is a symmetric, an exterior or a divided power algebra (or a tensor product of these) on some generators $I^{(b)}\langle a\rangle$.

{\bf Step 2: Computation of $\E(X^{(t)},Y)$.} By corollary \ref{cor-decalage}, $\E(X^{(t)},Y)$ is isomorphic to $(\R_{\alpha(Y)}\E(X,Y))^{(t)}$. Under this isomorphism, each generator $I^{(b)}\langle a\rangle$ of $\E(X,Y)$ corresponds to a generator $I^{(b+t)}\langle a+p^b \alpha(Y)\rangle$ of $\E(X^{(t)},Y)$, where $\alpha(S)=0$, $\alpha(\Lambda)=p^t-1$ and $\alpha(\Gamma)=2(p^t-1)$. Observe that in odd characteristic, $\alpha(Y)$ is even, so that no additional sign is introduced by the regrading functor $\R_{\alpha(Y)}$.

{\bf Step 3: Computation of $\E(X^{(t+s)},Y^{(s)})$, up to filtration.} By theorem \ref{thm-twisting-ss}, $\E(X^{(t+s)},Y^{(s)})$ is, up to a filtration, an algebra of the same kind as $\E(X^{(t)},Y)$, with more generators. To be more specific, each generator $I^{(b+t)}\langle a+p^b \alpha(Y)\rangle$ of $\E(X^{(t)},Y)$ gives birth to a family of generators (indexed by an integer $i$, with $0\le i<p^s$):
$$I^{(b+t+s)}\langle a+p^b \alpha(Y^*)+2ip^{b+t}\rangle=I^{(b+t+s)}\langle(2ip^t+\alpha(Y))p^{b} +a\rangle.$$

{\bf Step 4: Filtrations split.} Finally, we have to prove that the filtrations involved are trivial. This follows directly from the splitting results of section \ref{sec-split}, except in the cases of $\E(\Lambda^{(t+s)},S^{(s)})$ and $\E(\Gamma^{(t+s)},\Lambda^{(s)})$ when $\k$ has characteristic $2$. So let us assume that all the results but these two cases are proved. To prove the triviality of the filtrations of $\E(\Lambda^{(t+s)},S^{(s)})$ and $\E(\Gamma^{(t+s)},\Lambda^{(s)})$, we want to apply proposition \ref{prop-trivF2_2}. For this we need to check strict anticommutativity. So the following lemma finishes the proof.
\begin{lemme}\label{lm-strict-com}
Let $\k$ be a field of characteristic $2$. For all nonnegative integers $s,t$, the $\P_\k$-graded algebras 
$\E(\Lambda^{(t+s)}, S^{(s)})$ and $\E(\Gamma^{(t+s)}, \Lambda^{(s)})$
are strictly anticommutative (that is, for all $x$, $x\cdot x=0$ in these algebras).
\end{lemme}
\begin{proof}
Let us prove that $\E(\Gamma^{(t+s)}, \Lambda^{(s)})$ is strictly anticommutative (the proof for $\E(\Lambda^{(t+s)}, S^{(s)})$ is similar). Since $\k$ has characteristic $2$, we have an injective morphism of algebras $\alpha:\Lambda\hookrightarrow \Gamma$. It induces a morphism:
$$\beta:\E(\Gamma^{(t+s)}, \Lambda^{(s)})\to  \E(\Lambda^{ (t+s)}, \Lambda^{(s)})\;.$$
To prove that $\E(\Gamma^{(t+s)}, \Lambda^{(s)})$ is strictly anticommutative, it suffices to prove that $\beta$ is injective. Indeed, we already know that $\E(\Lambda^{(t+s)}, \Lambda^{(s)})$ is strictly anticommutative since it is a divided power algebra. 

For all $d\ge 0$, we have a commutative square, where the vertical arrows are injections induced by the canonical map $\Lambda^d\hookrightarrow \otimes^d$:
$$\xymatrix{
\H(\Gamma^d,\otimes^d)\ar[rr]^-{\H(\alpha,\otimes^d)} && \H(\Lambda^d,\otimes^d)\\
\H(\Gamma^d,\Lambda^d)\ar[rr]^-{\H(\alpha,\Lambda^d)}\ar@{^{(}->}[u] && \H(\Lambda^d,\Lambda^d)\ar@{^{(}->}[u]
}.$$
Now the map $\H(\alpha,\otimes^d)$ is an isomorphism (by lemma \ref{lm-facile}), so $\H(\alpha,\Lambda^d)$ is injective. Using proposition \ref{prop-twistgauche}, we get an injection 
$\H(\Gamma^{d\,(t)},\Lambda^{dp^t})\hookrightarrow\H(\Lambda^{d\,(t)},\Lambda^{dp^t})$. By theorem \ref{thm-twisting-ss}, $\mathrm{Gr}(\beta)$ equals the evaluation of this map on $E_s\otimes V^{(s)}$, hence it is injective. We deduce the injectivity of $\beta$.
\end{proof}

\subsection{Comparison with \cite{FFSS,Chalupnik2}}\label{subsec-compare} 

\subsubsection{The case $X\le Y$}
In this case, the algebras $$\Ext_{\P_\k}^*(X^{*\,(s+t)},Y^{*\,(s)})=\E(X^{(s+t)},Y^{(s)})(\k)$$ were first computed in \cite{FFSS}. Theorem \ref{thm-XleY} agrees with  \cite[Thm 5.8]{FFSS}. For example, we assert that $\E(\Lambda,\Lambda)(\k)$ is a divided power algebra on generators $g_i\in\Ext^{(2i+1)p^t-1}_{\P_\k}(I^{(t+s)},\Lambda^{p^t\,(s)})$, for $0\le i<p^s$. This is exactly \cite[Thm 5.8(6)]{FFSS}.

\subsubsection{The case $X>Y$} The algebras with $X>Y$ were first computed in \cite{Chalupnik2}. Our computations differ from the results of \cite{Chalupnik2}. 

First, theorems \ref{thm-XYdurF2} and \ref{thm-XLY} do not agree with \cite[Thm 3.2]{Chalupnik2}. It is already explained in remark \ref{rk-Chal} why \cite[Thm 3.2]{Chalupnik2} is false in characteristic $2$. 
In odd characteristic, the two results agree as graded functors, but the descriptions of the products only agree up to signs. Unlike our result, \cite[Thm 3.2]{Chalupnik2} does not yield a $(1,1)$-commutative algebra. Indeed, in an algebra of the form $\Lambda(A)\otimes \Gamma(B)$, where $A$ is concentrated in even degrees and odd weights, and $B$ is concentrated in odd degrees and odd weights, if there are two non proportional elements $a_1,a_2\in A$, we can choose $b\in B$ and form the products
$x=a_1\otimes b$ and $y=a_2\otimes\gamma_2(b)$, where $\gamma_2(b)=b\otimes b\in\Gamma^2(B)$. Then one has 
$$x\cdot y =  (a_1\wedge a_2)\otimes b\gamma_2(b) = - (a_2\wedge a_1)\otimes \gamma_2(b)b= - y\cdot x\;.$$
Since $x\cdot y\ne 0$, this contradicts $(1,1)$-commutativity. 

Finally, theorem \ref{thm-SG} does not agree with \cite[Cor 4.15]{Chalupnik2}. To be more specific, the two results look the same, but in our computation, the resulting algebra has more generators. For example, theorem \ref{thm-SG} asserts that $\E(S^p,\Gamma^p)$ is equal to $\Gamma^p\langle 0\rangle \oplus I^{(1)}\langle 2p-2\rangle \oplus I^{(1)}\langle 2p-3\rangle$, so the total dimension of $\E(S^p,\Gamma^p)(\k)$ is $3$. On the other hand, \cite[Cor 4.15]{Chalupnik2} predicts that this total dimension is $2$. So the following independent elementary computation argues in favor of our result.
\begin{lemme}\label{lm-ex} If $\k$ has characteristic $p$, then
$\Ext^i_{\P_\k}(S^p,\Gamma^p)$ equals $\k$ if $i=0, 2p-2$ or $2p-3$, and is zero otherwise.
\end{lemme}
\begin{proof}We need the following elementary facts.
For $0\le k < p$, the functor $\Lambda^{k}\otimes S^{p-k}$ is injective as a direct summand of $\otimes^{k}\otimes S^{p-k}$, and $\Gamma^k$ is injective since it is isomorphic to $S^k$. Moreover:
\begin{align*}
&\hom_{\P_\k}(S^p,S^p)=\k&&\text{ with basis the identity map, }\\
&\hom_{\P_\k}(S^p,S^1\otimes S^{p-1})=\k&& \text{ with basis the comultiplication, }\\ 
&\hom_{\P_\k}(S^p, \Lambda^k\otimes S^{p-k})=0&& \text{ for $k\ge 2$.}
\end{align*}
To prove lemma \ref{lm-ex}, we cut the exact Koszul complexes: 
\begin{align*}&\Lambda^p\hookrightarrow \Lambda^{p-1}\otimes S^1\to \dots \to \Lambda^1\otimes S^{p-1}\twoheadrightarrow S^p\\
&\Gamma^p\hookrightarrow \Gamma^{p-1}\otimes \Lambda^1\to \dots \to \Gamma^1\otimes \Lambda^{p-1}\twoheadrightarrow \Lambda^p
\end{align*}
into short exact sequences and analyze the long $\Ext^*_{\P_\k}(S^p,-)$-exact sequence associated to them.
We begin on the right with the first complex. Let $K$ be the kernel of the multiplication $S^1\otimes S^{p-1}\twoheadrightarrow S^p$. Since the composite of the comultiplication $S^p\to S^1\otimes S^{p-1}$ and the multiplication $S^1\otimes S^{p-1}\twoheadrightarrow S^p$ equals $p$ times the identity map (hence zero), we obtain that $\Ext^i(S^p,K)\simeq\k$ if $i=0,1$ and $0$ otherwise. For the other short exact sequences except the one involving $\Gamma^p$, the terms $\Ext^*_{\P_\k}(S^p, \Lambda^k\otimes S^{p-k})$ and $\Ext^*_{\P_\k}(S^p, \Lambda^k\otimes \Gamma^{p-k})$ are trivial so the long exact sequence induces a shifting. Finally the last short exact sequence $\Gamma^p\hookrightarrow \Gamma^1\otimes\Gamma^{p-1}\twoheadrightarrow C$ induces a shifting and also creates a nonzero element in $\hom_{\P_\k}(S^p, \Gamma^p)$.
\end{proof}

\nocite{*}
\bibliographystyle{cdraifplain}

\end{document}